\documentclass[11pt,a4paper]{article}

\usepackage[utf8]{inputenc}
\usepackage[T1]{fontenc}
\usepackage[english]{babel} 
\usepackage{amsmath}
\usepackage{amsfonts}
\usepackage{amsthm}
\usepackage{amssymb}
\usepackage{makeidx}
\usepackage{graphicx}
\usepackage{lmodern}
\usepackage[left=2cm,right=2cm,top=2cm,bottom=2cm]{geometry}

\usepackage{color}
\usepackage{comment}
\usepackage[dvipsnames]{xcolor}
\usepackage{graphicx}
\usepackage{framed}
\usepackage{tikz}
\usepackage{calrsfs}
\DeclareMathAlphabet{\pazocal}{OMS}{zplm}{m}{n}
\usepackage{bm}

\usepackage{fourier-orns}
\usepackage{mathtools}
\usepackage{relsize}
\usepackage{dsfont}
\usepackage{comment}

\usepackage{hyperref}
\hypersetup{
	linktocpage = true,
	colorlinks = true,
	linkcolor = {RoyalBlue},
	urlcolor = {RoyalBlue},
	citecolor = {Green}}

\newcommand{\B}{\mathbb{B}}
\newcommand{\R}{\mathbb{R}}

\newcommand{\Apazo}{\pazocal{A}}
\newcommand{\Epazo}{\pazocal{E}}
\newcommand{\Fpazo}{\pazocal{F}}
\newcommand{\Gpazo}{\pazocal{G}}

\newcommand{\Kpazo}{\pazocal{K}}

\newcommand{\Mpazo}{\pazocal{M}}
\newcommand{\Bpazo}{\pazocal{B}}
\newcommand{\Cpazo}{\pazocal{C}}
\newcommand{\Qpazo}{\pazocal{Q}}

\newcommand{\Dpazo}{\pazocal{D}}

\newcommand{\Spazo}{\pazocal{S}}
\newcommand{\Opazo}{\pazocal{O}}

\newcommand{\Dcal}{\mathcal{D}}
\newcommand{\Mcal}{\mathcal{M}}

\newcommand{\Lcal}{\mathcal{L}}
\newcommand{\Vcal}{\mathcal{V}}
\newcommand{\Pcal}{\mathcal{P}}

\newcommand{\Kcal}{\mathcal{K}}

\newcommand{\Id}{\textnormal{Id}}

\newcommand{\supp}{\textnormal{supp}}

\newcommand{\Lip}{\textnormal{Lip}}
\newcommand{\AC}{\textnormal{AC}}
\newcommand{\Graph}{\textnormal{Graph}}

\newcommand{\Div}{\textnormal{div}}

\newcommand{\dist}{\textnormal{dist}}
\newcommand{\textbn}[1]{\textnormal{\textbf{#1}}}
\newcommand{\co}{\overline{\textnormal{co}} \hspace{0.05cm}}

\newcommand{\dsf}{\textnormal{\textsf{d}}}
\newcommand{\esf}{\textnormal{\textsf{e}}}

\newcommand{\Bnu}{\boldsymbol{\nu}}

\newcommand{\Beta}{\boldsymbol{\eta}}

\renewcommand{\epsilon}{\varepsilon}

\newcommand{\INTDom}[3]{\int_{#2} #1 \textnormal{d} #3}
\newcommand{\INTSeg}[4]{\int_{#3}^{#4} #1 \textnormal{d} #2}
\newcommand{\NormL}[3]{\parallel \hspace{-0.1cm} #1 \hspace{-0.1cm} \parallel _ {L^{#2}(#3)}}
\newcommand{\NormC}[3]{\left\| #1  \right\| _ {C^{#2}(#3)}}
\newcommand{\Norm}[1]{\parallel \hspace{-0.1cm} #1 \hspace{-0.1cm} \parallel}

\newcommand{\dcc}{\dsf_{\textnormal{cc}}}
\newcommand{\tto}{\rightrightarrows}

\newtheorem{Def}{Definition}[section]
\newtheorem{thm}[Def]{Theorem}
\newtheorem{prop}[Def]{Proposition}
\newtheorem{rmk}[Def]{Remark}
\newtheorem{lem}[Def]{Lemma}
\newtheorem{cor}[Def]{Corollary}

\newtheorem{taggedhypx}{Hypotheses}
\newenvironment{taggedhyp}[1]
    {\renewcommand\thetaggedhypx{#1}\taggedhypx}
    {\endtaggedhypx}

\title{Carathéodory Theory and A Priori Estimates for Continuity Inclusions in the Space of Probability Measures\thanks{\textit{This material is based upon work supported by the Air Force Office of Scientific Research under award number FA9550-18-1-0254.}}}

\author{Benoît Bonnet-Weill\footnote{LAAS-CNRS, Université de Toulouse, CNRS, 7 avenue du colonel Roche, F-31400 Toulouse, France. \textit{E-mail}: \texttt{benoit.bonnet@laas.fr} (Corresponding author)}\; and Hélène Frankowska\footnote{CNRS,  IMJ-PRG,  UMR  7586,  Sorbonne  Université,  4  place  Jussieu,  75252  Paris,  France. \hfill \hspace{3.5cm} \textit{E-mail}: \texttt{helene.frankowska@imj-prg.fr}}}

\begin{document}

\maketitle

\begin{abstract}
In this article, we extend the foundations of the theory of differential inclusions in the space of compactly supported probability measures, introduced recently by the authors, to the setting of general Wasserstein spaces. In this context, we prove a novel existence result à la Peano for this class of dynamics under mere Carathéodory regularity assumptions. The latter is based on a natural, yet previously unexplored set-valued adaptation of the semi-discrete Euler scheme proposed by Filippov to study ordinary differential equations whose right-hand sides are measurable in the time variable. By leveraging some of the underlying methods along with new estimates for solutions of continuity equations, we also bring substantial improvements to the earlier versions of the Filippov estimates, compactness and relaxation properties of the solution sets of continuity inclusions, which are derived in the Cauchy-Lipschitz framework.  
\end{abstract}

{\footnotesize

\vspace{0.25cm}

\textbf{MSC2020 Subject Classification :} 28B20, 34A60, 34G20, 46N20, 49Q22
}

\tableofcontents


\section{Introduction}
\setcounter{equation}{0} \renewcommand{\theequation}{\thesection.\arabic{equation}}

In recent years, the study of continuity equations in the space of measures has been the object of an extensive interest in several mathematical communities. While the analysis of such partial differential equations was more commonly conveyed in Lebesgue or Sobolev spaces -- in which one could establish classical well-posedness results --, several research currents in pure and applied mathematics prompted the investigation of more general and weaker notions of solutions to these dynamics. 

Amongst these research endeavours, one of the most influential was the development of the modern theory of optimal transport -- rendered in the treatises \cite{AGS,OTAM,villani1} -- along with that of the theory of gradient flows in measure spaces. Indeed, the concepts introduced in the seminal works \cite{Jordan1998,Otto2001} and later formalised in \cite{AGS} expounded the fact that one could construct weak solutions to transport equations with very irregular driving fields by applying continuous-time steepest descent schemes to energy functionals defined over the space of probability measures. This far-reaching observation allowed to derive general well-posedness results for a wide variety of evolution equations encountered in physics, rational mechanics and biology, on the basis that they exhibit a variational structure, see e.g. \cite{AGS,CarrilloF2011,Carrillo2014,Erbar2010,Gianazza2009,Jordan1998,Maury2010,Otto2001,Santambrogio2017} and references therein. The infatuation for this innovative point of view stemmed both from its theoretical merits, and from its practical efficiency and adaptedness for designing numerical methods \cite{Benamou2016,Carlier2017,Peyre2015}. In a similar vein, the theory of mean-field games \cite{CDLL,Huang2006,Lasry2007} -- located halfway between control theory and the calculus of variations -- largely contributed to the popularisation of dynamical problems in measure spaces. Incidentally, some of the core concepts of these emerging research trends found relevant application outlets in a wealth of multiscale models aiming towards efficient descriptions of pedestrian flows \cite{CPT,Maury2011,Piccoli2018} as well as opinion propagation models \cite{AlbiPZ2017,Ha2009,ControlKCS} and swarming dynamics \cite{Albi2014,Carrillo2010,Carrillo2013}. Another active field of research that put this corpus of results to good use is that of mean-field control \cite{ContInc,PMPWass,Cavagnari2022,Fornasier2019,FornasierPR2014,Fornasier2014,Jimenez2020},  which saw the birth of several relevant extensions of the classical theory as further elaborated upon below. 

In the aforedescribed context, a growing body of literature at the intersection between PDE analysis, dynamical systems theory and optimal transport has been concerned with the derivation of general well-posedness results for Cauchy problems of the form 
\begin{equation}
\label{eq:IntroPDE}
\left\{
\begin{aligned}
& \partial_t \mu(t) + \Div_x (v(t,\mu(t)) \mu(t)) = 0, \\
& \mu(0) = \mu^0.
\end{aligned}
\right.
\end{equation}
Therein, the initial datum $\mu^0 \in \Pcal(\R^d)$ is a Borel probability measure, while the velocity field $v : [0,T] \times \Pcal(\R^d) \times \R^d \to \R^d$ is a Lebesgue-Borel map which may be \textit{nonlocal}, in the sense that it is allowed to depend on the measure variable itself. As alluded to in the previous paragraph, one of the main frameworks in which one can meaningfully derive well-posedness results for \eqref{eq:IntroPDE} is that of Wasserstein gradient flows (see e.g. \cite[Chapter 11]{AGS}), wherein $v(t,\mu(t)) \in L^p(\R^d,\R^d;\mu(t))$ is given as the (sub)gradient of a functional defined over the space of measures. Analogously, in the theory of mean-field games, the well-posedness of the forward measure dynamics frequently originates from a variational principle, which requires that the common cost of the agents satisfies a suitable convexity condition \cite{Lasry2007}. It is worth noting that general existence results are also available for irregular driving fields when the dynamics exhibits a Hamiltonian structure, see for instance \cite{AmbrosioGangbo}. On the other hand, as amply highlighted by the discussions in \cite[Chapter 8]{AGS}, the dynamics in \eqref{eq:IntroPDE} admits a natural interpretation as an ordinary differential equation in the Wasserstein manifold -- or rather bundle -- $(\Pcal_p(\R^d),W_p(\cdot,\cdot))$ for $p \in (1,+\infty)$. Thus, in the absence of an underlying variational structure, one should expect that the well-posedness of said Cauchy problems would rather stem from the regularity properties of the driving vector field. In that case, the results available in the literature can be split into two categories, depending on whether the velocity field is nonlocal or not. 
\begin{enumerate}
\item[$\diamond$] When $v : [0,T] \times \R^d \to \R^d$ is independent of the measure variable, the known optimal well-posedness settings are, on the one hand, Carathéodory assumptions (or small variations thereof) when the initial measure is arbitrary, and on the other hand Sobolev \cite{DiPernaLions} or BV \cite{AmbrosioPDE} regularity in the space variable, combined with integral bounds on the divergence or more general incompressibility assumptions \cite{Bianchini2020}, when the initial measure is absolutely continuous with respect to the Lebesgue measures. The structuring concepts of this latter class of solutions -- which are usually referred to as \textit{regular Lagrangian flows} --, are surveyed together with the classical Cauchy-Lipschitz theory in \cite{AmbrosioC2014}. 
\item[$\diamond$]  When $v : [0,T] \times \Pcal(\R^d) \times \R^d \mapsto \R^d$ also depends on the measure variable, general sufficient conditions for the well-posedness of \eqref{eq:IntroPDE} are only available in the Cauchy-Lipschitz and Carathéodory regularity frameworks (see e.g. \cite{ContInc,Crippa2013,Pedestrian} and references therein). In particular, there are currently no known generalisations of the concept of Lagrangian flow to the setting of nonlocal continuity equations.
\end{enumerate}
We end this literature overview by advertising a recent body of work initiated in \cite{Piccoli2019} and furthered in \cite{Camilli2021,CavagnariSS2022,PiccoliCDC2018}, in which the authors investigate relaxations of \eqref{eq:IntroPDE} wherein the driving fields are replaced by probability measures over the tangent bundle. This line of study -- which is highly reminiscent of the earlier work \cite{Bernard2008} -- bears strong resemblance with the theory of Young measures, and has already produced very promising results shedding light on the interplay between contraction semigroups in measure spaces and various kinds of Euler schemes. 

\bigskip

In the lineage of this body of work, the aim of this paper is to introduce several far-reaching refinements of the theory of \textit{continuity inclusions} in Wasserstein spaces, whose elaboration started in our previous work \cite{ContInc}. Therein, given a compactly supported measure $\mu^0 \in \Pcal_c(\R^d)$ and a \textit{set-valued map} $(t,\mu) \in [0,T] \times \Pcal_c(\R^d) \tto V(t,\mu) \subset L^p(\R^d,\R^d;\mu)$, we defined solutions of the Cauchy problem
\begin{equation}
\label{eq:IntroInc}
\left\{
\begin{aligned}
& \partial_t \mu(t) \in - \Div_x \Big( V(t,\mu(t)) \mu(t) \Big), \\
& \mu(0) = \mu^0, 
\end{aligned}
\right.
\end{equation}
as the collection of all absolutely continuous curves $\mu(\cdot) \in \AC([0,T],\Pcal_p(\R^d))$ for which there exists an $\Lcal^1$-\textit{measurable selection} $t \in [0,T] \mapsto v(t) \in V(t,\mu(t))$ such that \eqref{eq:IntroPDE} holds. The main motivation for considering such objects is that differential inclusions, and set-valued analysis at large, play an instrumental role in control theory and in the calculus of variations, as evidenced by the reference works \cite{Aubin1984,Aubin1990,Clarke,Vinter} and \cite{CannarsaS2004,Cesari2012}. They provide convenient tools to establish existence results for variational problems \cite{Filippov1962}, to derive first- and second-order optimality conditions -- both in the form of a Pontryagin Maximum Principle \cite{Frankowska1987,Frankowska2018,FrankowskaO2023,FrankowskaZZ2018} or Hamilton-Jacobi-Bellman equations \cite{Frankowska1989,Frankowska1995} --, and to investigate qualitative properties of optimal trajectories \cite{Cannarsa1991}. Owing to their mathematical versatility, these schemes have recently started to percolate in the communities of mean-field control and mean-field games \cite{SetValuedPMP,SemiSensitivity,Cannarsa2021,Cavagnari2022,Jimenez2020,PiccoliCDC2018}.  Compared with other notions that were put forth to define solutions of differential inclusions in measure spaces, such as those of \cite{CavagnariMQ2021,Jimenez2020}, our approach presents the advantage of being coherent with the modern theory of differential inclusions in vector spaces, surveyed e.g. in \cite{Aubin1984,Aubin1990}, as well as conceptually compatible with the geometric structure of Wasserstein spaces and the interpretation of continuity equations as generalised ODEs following \cite{AGS,Otto2001}. Besides, our construct is well-adapted to the study of control problems, as it ensures that there is a one-to-one correspondence between solutions of controlled measure dynamics and their set-valued counterparts.

The first main contribution of this article is an existence result à la Peano for \eqref{eq:IntroInc}, presented in  Theorem \ref{thm:Peano}, and derived under mere Carathéodory regularity assumptions. To be more precise, we show therein that if $V : [0,T] \times \Pcal_p(\R^d) \tto C^0(\R^d,\R^d)$ is $\Lcal^1$-measurable with respect to $t \in [0,T]$, continuous with respect to $\mu \in \Pcal_p(\R^d)$ and has convex images, then \eqref{eq:IntroInc} admits a solution from every initial datum. The requirement that the admissible velocities are convex is commonly known to be unavoidable to prove the existence of solutions to differential inclusions whose right-hand sides are not Lipschitz (see e.g. \cite[Chapter 2]{Aubin1984}), even in the familiar context of finite-dimensional vector spaces. Our strategy for proving this result revolves around an astute set-valued adaptation of the semi-discrete Euler scheme due to Filippov for Carathéodory ODEs, see e.g. \cite[Chapter 1]{Filippov2013}. To the best of our knowledge, this approach was not previously investigated even for classical differential inclusions. It also relies on rather new ways of studying the compactness of solutions to continuity equations, which we believe are of independent interest. 

The second contribution of this article lies in the transposition of the findings of \cite{ContInc}, derived for curves of compactly supported measures, to the setting of general Wasserstein spaces. The first of these key results are the so-called Filippov estimates, presented in Theorem \ref{thm:Filippov}, which provide the existence of a solution to \eqref{eq:IntroInc} whose distance to an a priori fixed measure curve is precisely controlled. Such estimates are extremely useful to produce admissible trajectories for control systems when conducting perturbative arguments and performing linearisations. The second of these results is the compactness of the solution set when the right-hand side of the dynamics is convex, which is needed in virtually every existence proof based on weak compactness arguments, both in optimal control theory and in the calculus of variations. The third and last is the so-called relaxation theorem, which asserts that in the absence of convexity, the closure of the solution set of \eqref{eq:IntroInc} coincides with that of the Cauchy problem in which the dynamics has been convexified. These three pivotal properties have been widely used in conjunction with one another to derive sharp first- and second-order optimality conditions for optimal control problems in various contexts, see e.g. \cite{Frankowska1990,Frankowska2018} for ordinary differential equations, \cite{Frankowska1990,FrankowskaMM2018} for general evolution equations in infinite-dimensional Banach spaces, and our previous work \cite{SetValuedPMP} concerned with mean-field optimal control problems. It should be noted that the relaxation theorem is also essential when investigating the fine properties of solutions to Hamilton-Jacobi-Bellman equations, as it allows to posit without loss of generality the existence of optimal trajectories associated with the corresponding value function. 

From a more technical standpoint, the extension of these results required new insights on the definition of continuity inclusions, as one needed a new functional setting that was amenable to working with velocities having unbounded support, but still enjoyed good topological properties. These reflections lead us to study continuity inclusions driven by velocity selections valued in $C^0(\R^d,\R^d)$ endowed with the topology of local uniform convergence, which happens to be a separable Fréchet space whose compact subsets are precisely characterised by the Ascoli-Arzel\`a theorem. It is our personal opinion that adopting this fresh viewpoint allows for a cleaner, more general, and perhaps more streamlined angle to study measure dynamics. We finally stress that while some of the working assumptions used throughout this manuscript are more stringent than those in our previous work, mainly to palliate the fact that the measures under consideration may have unbounded supports, the results presented below strictly contain those of \cite{ContInc}, up to minor technical adjustments. It should finally be noted that some of our developments are based on new quantitative stability estimates for continuity equations under low regularity requirements, which should again be of independent interest for our fellow practitioners.  

\bigskip

The contributions and organisation of the article can be summarised as follows. In Section \ref{section:Preliminary}, after recalling some concepts pertaining to measure theory, optimal transport, set-valued and functional analysis, we derive new compactness and stability estimates for solutions of continuity equations in Proposition \ref{prop:Moment} and Proposition \ref{prop:Gronwall} respectively. We then define solutions of \eqref{eq:IntroInc} for set-valued maps $V : [0,T] \times \Pcal_p(\R^d) \tto C^0(\R^d,\R^d)$ in terms of velocity selections which are measurable for the standard Fr\'echet topology of the space of continuous functions (see Definition \ref{def:CompactConv} below). Then, in Section \ref{section:Peano}, we prove a novel existence result for solutions of \eqref{eq:IntroInc} in Theorem \ref{thm:Peano}, under the mere assumption that the right-hand side of the dynamics is Caratheodory and convex-valued. We subsequently move on in Section \ref{section:Filippov} to the Cauchy-Lipschitz setting, and discuss the main results of the theory of differential inclusions that may be derived therein. In Section \ref{subsection:Filippov}, we start by establishing two far-reaching versions of the Filippov estimates, starting with a local one in Theorem \ref{thm:Filippov} and proceeding with a global one in Corollary \ref{cor:GlobalFilippov}. Then, we turn our attention to the topological properties of the solution sets of \eqref{eq:IntroInc} in Section \ref{subsection:Compactness}. In Theorem \ref{thm:Compactness}, we start by showing that the latter are compact for the topology of uniform convergence when the right-hand side of the dynamics has convex images. In Theorem \ref{thm:Relaxation}, we then relax the convexity requirement and prove that in this case, the closure of the solution set coincides with that of the convexified Cauchy problem. We finally close the paper by an appendix containing the proofs of several technical results and estimates.


\section{Preliminaries}
\label{section:Preliminary}
\setcounter{equation}{0} \renewcommand{\theequation}{\thesection.\arabic{equation}}

In the coming sections, we expose preliminary results pertaining to measure theory, set-valued analysis, optimal transport and measure dynamics in general. 


\subsection{Measure theory and optimal transport in Banach spaces}

In this first preliminary section, we recollect basic notions of measure theory and optimal transport in Banach spaces, for which we largely refer to \cite{AmbrosioFuscoPallara,BogachevI,DiestelUhl} and \cite{AGS} respectively.

\paragraph*{Elements of measure and integration theory.} Given a separable Banach space $(X,\Norm{\cdot}_X)$, we will denote by $X^*$ its topological dual and write $\langle \cdot,\cdot \rangle_X$ for the underlying duality pairing. In the sequel, the notation $C_b^0(X,\R^d)$ will refer to the vector space of continuous bounded maps from $X$ into $\R^d$, and in  the particular case where $X = \R^m$ for some $m \geq 1$, we will denote by $C^{\infty}_c(\R^m,\R)$ the vector space of infinitely differentiable functions with compact support. Letting $(Y,d_{Y}(\cdot,\cdot))$ be a complete separable metric space, we shall more generally write $C^0(X,Y)$ for the set of continuous maps from $X$ into $Y$, as well as $\AC(I,Y)$ for that of absolutely continuous arcs defined over an interval $I \subset \R$ with values in $Y$. In addition, $\Lip(\Omega,Y)$ will stand for the space of Lipschitz maps from a subset $\Omega \subset X$ into $Y$, and we shall write $\Lip(\phi \, ; \Omega)$ for the Lipschitz constant of an element $\phi \in \Lip(\Omega,Y)$.

In what follows, given closed set $\Omega \subset X$, we will consider the vector space $\Mcal(\Omega,\R^d)$ of $\R^d$-valued finite Radon measures. By the Riesz representation theorem (see e.g. \cite[Theorem 1.54]{AmbrosioFuscoPallara}), it is known that when $\Omega$ is compact, the latter is isomorphic to the topological dual of the Banach space $(C^0(\Omega,\R^d),\NormC{\cdot}{0}{\Omega,\R^d})$ under the action of the duality pairing
\begin{equation}
\label{eq:DualityContinuous}
\langle \Bnu , \phi \rangle_{C^0(\Omega,\R^d)} := \sum_{i=1}^d \INTDom{\phi_i(x)}{\Omega}{\Bnu_i(x)}, 
\end{equation}
defined for all $\Bnu \in \Mcal(\Omega,\R^d)$ and $\phi \in C^0(\Omega,\R^d)$. Throughout the article, we denote by $\Pcal(\Omega)$ the space of Borel probability measures over $\Omega$ endowed with the \textit{narrow topology}, that is the coarsest topology for which the mappings
\begin{equation}
\label{eq:NarrowTopo}
\mu \in \Pcal(\Omega) \mapsto \INTDom{\phi(x)}{X}{\mu(x)} \in \R 
\end{equation}
are continuous for every element $\phi \in C^0_b(\Omega,\R)$. It is a standard fact in measure theory (see e.g. \cite[Remark 5.1.2]{AGS}) that $\Pcal(\Omega)$ endowed with the narrow topology is a separable space, and we will write
\begin{equation*}
\mu_n \underset{n \to +\infty}{~ \rightharpoonup^*} \mu, 
\end{equation*}
for the notion of convergence induced by \eqref{eq:NarrowTopo} over $\Pcal(\Omega)$.

Given two separable Banach spaces $(X,\Norm{\cdot}_X)$ and $(Y,\Norm{\cdot}_Y)$ along with some $p \in [1,+\infty)$, we will write $(L^p(\Omega,Y;\mu),\NormL{\cdot}{p}{X,Y; \,\mu})$ for the space of maps from a subset $\Omega \subset X$ into $Y$ which are $p$-integrable in the sense of Bochner (see e.g. \cite[Chapter II]{DiestelUhl}) with respect to a measure $\mu \in \Mcal(\Omega,\R_+)$. Analogously, we let $L^{\infty}(\Omega,Y;\mu)$ be the space of $\mu$-essentially bounded maps from $X$ into $Y$, and use the denser notation $(L^p(I,Y),\NormL{\cdot}{p}{I})$ for $p \in [1,+\infty]$ when $X := I$ is an interval and $\mu := \Lcal^1$ is the standard $1$-dimensional Lebesgue measure. We recall below a powerful criterion for weak compactness in $L^1(I,X)$, whose statement can be found in \cite[Corollary 2.6]{Diestel1993} (see also the earlier versions \cite{Diestel1977,Ulger1991}).

\begin{thm}[A weak compactness criterion for Bochner integrable maps]
\label{thm:L1WeakCompactness}
Let $(X,\Norm{\cdot}_X)$ be a separable Banach space, $I \subset \R$ be an interval and $(v_n(\cdot)) \subset L^1(I,X)$. Suppose that there exists a map $m(\cdot) \in L^1(I,\R_+)$ and a family $(K_t)_{t \in I}$ of weakly compact subsets of $X$ such that 
\begin{equation*}
\Norm{v_n(t)}_X \, \leq m(t) \qquad \text{and} \qquad v_n(t) \in K_t 
\end{equation*}
for $\Lcal^1$-almost every $t \in [0,T]$ and each $n \geq 1$. Then, there exists a subsequence $(v_{n_k}(\cdot)) \subset L^1(I,X)$ that converges weakly to an element $v(\cdot) \in L^1(I,X)$. In particular
\begin{equation*}
\INTDom{\big\langle \Bnu(t) , v(t) - v_{n_k}(t) \big\rangle_X \, }{I}{t} ~\underset{k \to +\infty}{\longrightarrow}~ 0,
\end{equation*}
for every $\Bnu(\cdot) \in L^{\infty}(I,X^*) \subset L^1(I,X)^*$. 
\end{thm}

\paragraph*{Optimal transport and Wasserstein spaces.} Throughout this article, the notation $\Pcal_p(X)$ will refer to the subset of probability measures whose \textit{moment of order} $p \in [1,+\infty)$ is finite, that is
\begin{equation*}
\Mpazo_p^p(\mu) := \INTDom{|x|^p}{X}{\mu(x)} < +\infty. 
\end{equation*}
In the following definition, we recall the known concepts of \textit{image measure} through a Borel map, as well as that of \textit{transport plan}. 

\begin{Def}[Image measures and transport plans]
The \textnormal{image} of a measure $\mu \in \Pcal(X)$ through a Borel map $f : X \to Y$, denoted by $f_{\sharp} \mu \in \Pcal(Y)$, is defined as
\begin{equation*}
f_{\sharp} \mu(B) = \mu(f^{-1}(B)), 
\end{equation*}
for every Borel set $B \subset Y$. Given two probability measures $\mu,\nu \in \Pcal(X)$, we say that an element $\gamma \in \Pcal(X \times X)$ is a \textnormal{transport plan} between $\mu$ and $\nu$, denoted by $\gamma \in \Gamma(\mu,\nu)$, provided that
\begin{equation*}
\pi^1_{\sharp} \gamma = \mu \qquad \text{and} \qquad \pi^2_{\sharp} \gamma = \nu,
\end{equation*}
where $\pi^1,\pi^2 : X \times X \to X$ stand for the projections onto the first and second factors.  
\end{Def}

When $(X,\Norm{\cdot}_X) := (\R^d,|\cdot|)$ is a $d$-dimensional real vector space endowed with its usual Euclidean structure, it is a standard result in optimal transport theory that for any $p \in [1,+\infty)$, the quantity 
\begin{equation}
\label{eq:WassDef}
W_p(\mu,\nu) := \inf_{\gamma \in \Gamma(\mu,\nu)} \bigg( \INTDom{|x-y|^p}{\R^{2d}}{\gamma(x,y)} \bigg)^{1/p}
\end{equation}
defined for each $\mu,\nu \in \Pcal_p(\R^d)$ is a distance over $\Pcal_p(\R^d)$. Moreover, it comes as an easy consequence of the direct method of the calculus of variations that the infimum in \eqref{eq:WassDef} is always attained, and we denote by $\Gamma_o(\mu,\nu)$ the corresponding set of \textit{$p$-optimal transport plans}. In the following propositions, we recall some of the main properties of the \textit{Wasserstein spaces}, along with a handy distance estimate.

\begin{prop}[Topology of Wasserstein spaces]
\label{prop:WassDef}
The metric spaces $(\Pcal_p(\R^d),W_p(\cdot,\cdot))$ are complete and separable, and the Wasserstein distances metrise the narrow topology \eqref{eq:NarrowTopo}, in the sense that
\begin{equation*}
W_p(\mu_n,\mu) ~\underset{n \to +\infty}{\longrightarrow}~ 0 \qquad \text{if and only if} \qquad \left\{
\begin{aligned}
\mu_n & ~\underset{n \to +\infty}{~ \rightharpoonup^*}~ \mu, \\
\INTDom{|x|^p}{\R^d}{\mu_n(x)}  & ~\underset{n \to +\infty}{\longrightarrow}~ \INTDom{|x|^p}{\R^d}{\mu(x)}, 
\end{aligned}
\right.
\end{equation*}
for any sequence $(\mu_n) \subset \Pcal_p(\R^d)$ and each $\mu \in \Pcal_p(\R^d)$. Moreover, a subset $\Kcal \subset \Pcal_p(\R^d)$ is relatively compact for the topology induced by the metric $W_p(\cdot,\cdot)$ \textnormal{if and only if} it satisfies
\begin{equation*}
\sup_{\mu \in \Kcal} \INTDom{|x|^p}{\{x \; \text{s.t.} \; |x| \geq k\}}{\mu(x)} ~\underset{k \to +\infty}{\longrightarrow}~ 0,  
\end{equation*}
namely if and only if it is uniformly $p$-integrable. 
\end{prop}

\begin{proof}
We point the interested reader to \cite[Chapter 7]{AGS} or \cite[Chapter 6] {villani1}.
\end{proof}

\begin{prop}[Classical optimal transport estimate]
\label{prop:WassEst}
Given some $\mu,\nu \in \Pcal_p(\R^d)$, it holds that 
\begin{equation}
\label{eq:WassEst}
\INTDom{\phi(x)}{\R^d}{(\mu-\nu)(x)} ~\leq~ \Lip(\phi \, ; \R^d) W_1(\mu,\nu) ~\leq~ \Lip(\phi \, ; \R^d) W_p(\mu,\nu)
\end{equation}
for every $\phi \in \Lip(\R^d,\R)$. 
\end{prop}


\subsection{Set-valued analysis and topological properties of continuous functions}

In this section, we recall some notations and results of set-valued and functional analysis. We shall mostly rely on the treatises \cite{Aubin1984,Aubin1990} for the former and reference monographs \cite{Horvath2012,Rudin1987} for the latter.

In what follows, we write $(E,d_E(\cdot,\cdot))$ to denote a separable \textit{Fr\'echet space}, i.e. a complete separable locally convex topological vector space $E$ whose topology is induced by a translation-invariant metric $d_E(\cdot,\cdot)$. In this context, we define the \textit{closed convex hull} of a set $B \subset E$ as
\begin{equation*}
\co(B) := \overline{\bigcup_{N \geq 1} \Bigg\{ \sum_{i=1}^N \alpha_i b_i ~\, ~\textnormal{s.t.}~~ b_i \in B, ~ \alpha_i \in [0,1] ~~ \text{for $i \in \{1,\dots,N\}$} ~~ and ~~ \sum_{i=1}^N \alpha_i =1  \Bigg\}}^E,
\end{equation*}
where ``$\overline{\bullet}^E$'' stands for the closure with respect to $d_E(\cdot,\cdot)$. We will also use the generic notation
\begin{equation*}
\dist_{E}(x \, ; \Qpazo) := \inf_{y \in \Qpazo} d_{E}(x,y)
\end{equation*}
for the distance between an element $x \in E$ and a closed set $\Qpazo \subset E$. 

\paragraph*{Set-valued analysis.} Given two complete separable metric spaces $(X,d_X(\cdot,\cdot))$ and $(Y,d_Y(\cdot,\cdot))$, we write $\Fpazo : X \tto Y$ to mean that $\Fpazo(\cdot)$ is a \textit{set-valued map} -- or \textit{multifunction} -- from $X$ into $Y$. A set-valued map is said to have closed images (respectively convex, when the notion makes sense) if the sets $\Fpazo(x) \subset Y$ are closed (respectively convex) for each $x \in X$. In addition, we define its \textit{graph} by 
\begin{equation*}
\Graph(\Fpazo) := \bigg\{ (x,y) \in X \times Y \, ~\textnormal{s.t.}~ y \in \Fpazo(x) \bigg\}.
\end{equation*}
Below, we recall the standard concept of measurability for set-valued mappings defined over some subinterval $I \subset \R$ of the real line, endowed with the complete $\sigma$-algebra of Lebesgue-measurable sets. 

\begin{Def}[Measurable set-valued maps and measurable selections]
\label{def:Meas}
A multifunction $\Fpazo : I \tto Y$ is said to be \textnormal{$\Lcal^1$-measurable} provided that 
\begin{equation*}
\Fpazo^{-1}(\Opazo) := \Big\{ t \in I \, ~\textnormal{s.t.}~ \Fpazo(t) \cap \Opazo \neq \emptyset \Big\}
\end{equation*}
is $\Lcal^1$-measurable for every open set $\Opazo \subset Y$. A mapping $f : I \to Y$ is called a \textnormal{measurable selection} of $\Fpazo(\cdot)$ if it is $\Lcal^1$-measurable and such that $f(t) \in \Fpazo(t)$ for $\Lcal^1$-almost every $t \in I$. 
\end{Def}

In the following theorem, we recollect a deep result of set-valued analysis which provides the existence of measurable selections for measurable multifunctions whose images are closed subsets of a complete separable metric space. 

\begin{thm}[Existence of measurable selections]
\label{thm:Selection}
Every $\Lcal^1$-measurable set-valued map $\Fpazo : I \tto Y$ with nonempty closed image admits a measurable selection. 
\end{thm}

In  our subsequent developments, we will use the notions of \textit{continuity} and \textit{Lipschitz regularity} for set-valued mappings, both of which are recalled in the following definitions. Therein, we denote by $\B_X(x,r)$ and $\B_Y(y,r)$ the closed ball of radius $r>0$ centered at $x \in X$ and $y \in Y$ respectively, and use the condensed notation $\B_X(\Omega,r) := \{ x \in X ~\,\text{s.t.}~ d_X(x,x') \leq r ~\text{for some $x' \in \Omega$}~ \}$ for any $\Omega \subset X$. 

\begin{Def}[Continuous set-valued maps]
\label{def:UscSetvalued}
A multifunction $\Fpazo : X \tto Y$ is said to be \textnormal{continuous} at $x \in X$ if both the following conditions hold: 
\begin{enumerate}
\item[$(i)$] $\Fpazo(\cdot)$ is \textnormal{upper-semicontinuous} at $x \in X$, i.e. for every $\epsilon >0$, there exists $\delta > 0$ such that 
\begin{equation*}
\Fpazo(x') \subset \B_Y(\Fpazo(x) , \epsilon)
\end{equation*}
for all $x' \in \B_X(x,\delta)$.
\item[$(ii)$] $\Fpazo(\cdot)$ is \textnormal{lower-semicontinuous} at $x \in X$, i.e. for every $\epsilon > 0$ and each $y \in \Fpazo(x)$, there exists $\delta  > 0$ such that 
\begin{equation*}
\Fpazo(x') \cap \B_Y(y,\epsilon) \neq \emptyset
\end{equation*}
for all $x' \in \B_X(x,\delta)$.

\end{enumerate}
\end{Def} 

\begin{Def}[Lipschitz continuous set-valued maps]
A multifunction $\Fpazo : X \tto Y$ is said to be \textnormal{Lipschitz continuous} if there exists a constant $L > 0$ such that 
\begin{equation*}
\Fpazo(x') \subset \B_Y \Big( \Fpazo(x) \hspace{0.05cm} , \hspace{0.05cm} L d_X(x,x') \Big),
\end{equation*}
for every $x,x' \in X$. 
\end{Def}  

In the sequel, we will frequently resort to the general notion of \textit{Carathéodory} set-valued map between metric spaces, which is defined as follows. 

\begin{Def}[Carathéodory set-valued maps]
A set-valued map $\Gpazo : I \times X \tto Y$ is said to be \textnormal{Carathéodory} if $t \in I \tto \Gpazo(t,x)$ is $\Lcal^1$-measurable for all $x \in X$ and $x \in X \tto \Gpazo(t,x)$ is continuous for $\Lcal^1$-almost every $t \in I$.  
\end{Def}

In the following lemma -- whose proof is outlined in Appendix \ref{section:AppendixMeas} --, we state measurable selection principles adapted from \cite[Section 8.1]{Aubin1990} and \cite[Section 9]{Wagner1977} that we shall extensively use in the sequel. Therein, we let $\varphi : I \times Y \to \R_+ \cup \{+\infty\}$ be an extended real-valued function satisfying
\begin{equation}
\label{eq:varphisup}
\varphi(t,y) = \sup_{n \geq 1} \varphi_n(t,y)  
\end{equation}
for $\Lcal^1$-almost every $t \in I$ and all $y \in Y$, with $(\varphi_n(\cdot,\cdot))$ being a pointwisely non-decreasing sequence of Carathéodory maps. This implies in particular that $t \in I \mapsto \varphi(t,y)$ is $\Lcal^1$-measurable for each $y \in Y$, whereas $y \in Y \mapsto \varphi(t,y)$ is lower-semicontinuous for $\Lcal^1$-almost every $t \in I$.

\begin{lem}[Measurable selections principles]
\label{lem:MeasurableSel}
Let $\Fpazo : I \tto Y$ and $\Gpazo : I \times X \tto Y$ be two set-valued maps with nonempty images and $\varphi : I \times Y \to \R_+$ be as above. In addition, fix two $\Lcal^1$-measurable functions $x : I \to X$ and $L : I \to \R_+$. 
\begin{enumerate}
\item[(a)] Suppose that $\Fpazo(\cdot)$ is $\Lcal^1$-measurable with compact images, and that the set-valued map 
\begin{equation*}
t \in I \tto \Fpazo(t) \cap \Big\{ y \in Y \, ~\textnormal{s.t.}~ \varphi(t,y) \leq L(t) \Big\} 
\end{equation*}
has nonempty images. Then the latter is $\Lcal^1$-measurable, and there exists a measurable selection $t \in I \mapsto f(t) \in \Fpazo(t)$ such that $\varphi(t,f(t)) \leq L(t)$ for $\Lcal^1$-almost every $t \in I$.
\item[(b)] Suppose that $\Fpazo(\cdot)$ is $\Lcal^1$-measurable with compact images. Then, the set-valued map 
\begin{equation*}
t \in I \tto \Fpazo(t) \cap \bigg\{ y \in Y ~\, \text{s.t.}~ \varphi(t,y) = \inf_{z \in \Fpazo(t)} \varphi(t,z) \bigg\} 
\end{equation*}
is $\Lcal^1$-measurable with nonempty closed images, and as such admits a measurable selection. 
\item[(c)] Suppose that $\Gpazo(\cdot,\cdot)$ is Carathéodory with closed images. Then, the set-valued map $t \in I \tto \Gpazo(t,x(t))$ is $\Lcal^1$-measurable, and as such admits a measurable selection.
\end{enumerate}
\end{lem} 

We end this primer in set-valued analysis by stating an adaptation of Aumann's integral convexity theorem, for which we point the interested reader to \cite[Theorem 8.6.4]{Aubin1990}. 

\begin{thm}[Convexity of the Aumann integral]
\label{thm:Aumann}
Suppose that $(X,\Norm{\cdot}_X)$ is a separable Banach space and let $\Fpazo : I \to X$ be a set-valued map with closed nonempty images which is \textnormal{integrably bounded}, in the sense that there exists some $k(\cdot) \in L^1(I,\R_+)$ such that 
\begin{equation*}
\Fpazo(t) \subset k(t) \B_X,
\end{equation*}
for $\Lcal^1$-almost every $t \in I$. Then, for every measurable set $\Omega \subset I$, any measurable selection $t \in I \mapsto f(t) \in \co \Fpazo(t)$ and all $\delta >0$, there exists a measurable selection $t \in I \mapsto f_{\delta}(t) \in \Fpazo(t)$ such that 
\begin{equation*}
\bigg\| \INTDom{f(t)}{\Omega}{t} - \INTDom{f_{\delta}(t)}{\Omega}{t} \, \bigg\|_X \leq \delta.
\end{equation*}
\end{thm}


\paragraph*{Topologies and metrics over $C^0(\R^d,\R^d)$.} Throughout the manuscript, we will almost exclusively work with multifunctions whose values are subsets of $C^0(\R^d,\R^d)$. Following the pioneering work \cite{Warner1958}, it is known that the most natural topology to endow this space with is that of \textit{uniform convergence on compact sets} -- or \textit{compact convergence} --, whose definition and distinctive features are recalled below. 

\begin{Def}[The topology of compact convergence]
\label{def:CompactConv}
A sequence of continuous maps $(v_n) \subset C^0(\R^d,\R^d)$ is said to converge \textnormal{uniformly on compact sets} to some $v \in C^0(\R^d,\R^d)$ provided that 
\begin{equation*}
\NormC{v - v_n}{0}{K,\R^d} ~\underset{n \to + \infty}{\longrightarrow}~ 0,
\end{equation*}
for every compact set $K \subset \R^d$. The topology that this notion of convergence induces on $C^0(\R^d,\R^d)$ is complete, separable, and metrised by the translation-invariant metric 
\begin{equation}
\label{eq:Defdcc}
\dcc(v,w) := \sum_{k=1}^{+ \infty} 2^{-k} \min \Big\{ 1 \, , \NormC{v-w}{0}{B(0,k),\R^d} \Big\}
\end{equation}
that is defined for each $v,w \in C^0(\R^d,\R^d)$. As such, the latter endows $(C^0(\R^d,\R^d),\dcc(\cdot,\cdot))$ with the structure of a separable Fr\'echet space.  
\end{Def}


In our subsequent developments, we will always consider $C^0(\R^d,\R^d)$ endowed with the separable Fr\'echet topology induced by $\dcc(\cdot,\cdot)$, which interestingly carries a functional characterisation of Carathéodory vector fields, as illustrated by the following result borrowed from \cite[page 511]{Papageorgiou1986}. 

\begin{lem}[Carathéodory vector fields as measurable functions in $C^0(\R^d,\R^d)$]
\label{lem:Carathéodory}
A vector field $(t,x) \in I \times \R^d \mapsto v(t,x) \in \R^d$ is Carathéodory \textnormal{if and only if} its functional lift $t \in I \mapsto v(t) \in C^0(\R^d,\R^d)$ is $\Lcal^1$-measurable for the topology of uniform convergence on compact sets.
\end{lem}

In light of this result, we will systematically identify $\Lcal^1$-measurable maps $t \in I \mapsto v(t) \in C^0(\R^d,\R^d)$ with Carathéodory vector fields $v : I \times \R^d \to \R^d$, and thus work within the vector space
\begin{equation*}
\Lcal(I,C^0(\R^d,\R^d)) := \bigg\{ v : I \times \R^d \to \R^d \,~\text{s.t. $v(\cdot,\cdot)$ is a Carathéodory vector field}  \bigg\}. 
\end{equation*}
In addition to its amenable topological properties, the notion of uniform convergence on compact sets is particularly well-tailored to the formulation of compactness results, as illustrated by the following general version of the Ascoli-Arzel\`a theorem from \cite[Theorem 11.28]{Rudin1987}.

\begin{thm}[Ascoli-Arzel\`a compactness theorem]
\label{thm:Ascoli}
Let $(v_n) \subset C^0(\R^d,\R^d)$ be a sequence of maps which are uniformly bounded and equi-continuous on compact sets. Then, there exists an element $v \in C^0(\R^d,\R^d)$ for which
\begin{equation*}
\dcc(v_{n_k},v) ~\underset{k \to +\infty}{\longrightarrow}~ 0,
\end{equation*}
along a subsequence $(v_{n_k}) \subset C^0(\R^d,\R^d)$. Similarly if $K \subset \R^d$ is compact and $(v_n) \subset C^0(K,\R^d)$ is a sequence of uniformly bounded and equi-continuous maps, then there exists $v \in C^0(K,\R^d)$ such that
\begin{equation*}
\NormC{v-v_{n_k}}{0}{K,\R^d} ~\underset{k \to +\infty}{\longrightarrow}~ 0,
\end{equation*}
along a subsequence $(v_{n_k}) \subset C^0(K,\R^d)$. 
\end{thm}

\begin{proof}
The proof of the first compactness statement can be found in \cite[Theorem 11.28]{Rudin1987}, while that of the second one is simply the standard Ascoli-Arzel\`a theorem. 
\end{proof}

In order to formulate the well-posedness results of Section \ref{section:Filippov}, we will also need a global notion of vicinity for continuous maps, that will be inherently stronger than the local one provided by the metric $\dsf_{cc}(\cdot,\cdot)$. For this reason, we will also consider the extended supremum metric defined by 
\begin{equation}
\label{eq:ExtendedC0}
\dsf_{\sup}(v,w) := \sup_{x \in \R^d} |v(x) - w(x)| \in \R \cup \{+\infty\},
\end{equation}
for every $v,w \in C^0(\R^d,\R^d)$. In the following lemma, we recollect for the sake of completeness a classical result which shows that the natural lift of $\dsf_{\sup}(\cdot,\cdot)$ to the topological vector space $\Lcal(I,C^0(\R^d,\R^d))$ is complete, in the sense that its Cauchy sequences converge.

\begin{lem}[Completeness of the integral of the extended metric]
\label{lem:CompleteMetric}
Consider a sequence of maps $(v_n(\cdot)) \subset \Lcal(I,C^0(\R^d,\R^d))$ satisfying the Cauchy condition
\begin{equation}
\label{eq:ExtendedCauchy}
\INTSeg{\dsf_{\sup}(v_n(t),v_m(t))}{t}{0}{T} ~\underset{n,m \to +\infty}{\longrightarrow}~ 0. 
\end{equation}
Then, there exists an element $v(\cdot) \in \Lcal(I,C^0(\R^d,\R^d))$ such that
\begin{equation*}
\INTSeg{\dsf_{\sup}(v_n(t),v(t))}{t}{0}{T} ~\underset{t \to +\infty}{\longrightarrow}~ 0.
\end{equation*}
\end{lem}

\begin{proof}
The proof of this result being somewhat standard, it is deferred to Appendix \ref{section:AppendixCompleteness} below. 
\end{proof}


\subsection{Continuity equations and inclusions in the space of probability measures}

In this section, we recollect known results about continuity equations in measures spaces, following the usual Cauchy-Lipschitz and superposition-type theories surveyed e.g. in \cite{AmbrosioC2014,AGS}, as well as their set-valued counterpart introduced by the authors of the present manuscript in \cite{ContInc}. 

Given a real number $p \in [1,+\infty)$, a time horizon $T > 0$ and a velocity field $v : [0,T] \times \R^d \to \R^d$, we will focus our attention on the well-posedness and qualitative properties of the Cauchy problem
\begin{equation}
\label{eq:CE}
\left\{
\begin{aligned}
& \partial_t \mu(t) + \Div_x (v(t)\mu(t)) = 0, \\
& \mu(\tau) = \mu_{\tau}, 
\end{aligned}
\right.
\end{equation}
wherein $(\tau,\mu_{\tau}) \in [0,T] \times \Pcal_p(\R^d)$ is a fixed datum and $\mu(\cdot) \in C^0(I,\Pcal(\R^d))$ is a narrowly continuous curve of measure. The dynamics appearing in the first line of \eqref{eq:CE} is referred to as a \textit{continuity equation}, and understood in the sense of distributions 
\begin{equation*}
\INTSeg{\INTDom{\Big( \partial_t \phi(t,x) + \big\langle \nabla_x \phi(t,x) , v(t,x) \big\rangle \Big)}{\R^d}{\mu(t)(x)}}{t}{0}{T} = 0, 
\end{equation*}
for every test function $\phi \in C^{\infty}_c((0,T) \times \R^d,\R)$. Throughout this article, we will mostly work with velocity fields satisfying all or parts of the following assumptions which are standard when studying continuity equations in the Cauchy-Lipschitz framework.

\begin{taggedhyp}{\textbn{(CE)}}
\label{hyp:CE} \hfill
\begin{enumerate}
\item[$(i)$] The vector field $v : [0,T] \times \R^d \to \R^d$ is Carathéodory, and there exists a map $m(\cdot) \in L^1([0,T],\R_+)$ such that 
\begin{equation*}
|v(t,x)| \leq m(t) \big( 1 + |x| \big),
\end{equation*}
for $\Lcal^1$-almost every $t \in [0,T]$ and all $x \in \R^d$. 
\item[$(ii)$] There exists a map $l(\cdot) \in L^1([0,T],\R_+)$ such that 
\begin{equation*}
\Lip(v(t) \, ; \R^d) \leq l(t), 
\end{equation*}
for $\Lcal^1$-almost every $t \in [0,T]$. 
\end{enumerate}
\end{taggedhyp}

Let it be noted that, while the velocity fields are generally assumed to be locally Lipschitz with respect to $x \in \R^d$ in \cite{ContInc} and other references in the literature dealing with compactly supported measures, the global condition written in \ref{hyp:CE}-$(ii)$ is necessary for the derivation of quantitative stability estimates in $(\Pcal_p(\R^d),W_p(\cdot,\cdot))$. In the following definition, we recall the notion of \textit{flows of homeomorphisms} generated by a Carathéodory vector field.

\begin{Def}[Flows of homeomorphisms]
\label{def:Flows}
Let $v : [0,T] \times \R^d \to \R^d$ be a velocity field satisfying hypotheses \ref{hyp:CE}. Then, we denote by $(\Phi^v_{(\tau,t)}(\cdot))_{\tau,t \in [0,T]} \subset C^0(\R^d,\R^d)$ the unique \textnormal{semigroup of homeomorphisms} that solve the Cauchy problems 
\begin{equation*}
\Phi^v_{(\tau,t)}(x) = x + \INTSeg{v \big( s , \Phi_{(\tau,s)}^v(x) \big)}{s}{\tau}{t}, 
\end{equation*}
for all times $\tau,t \in [0,T]$ and every $x \in \R^d$. 
\end{Def}

Under the Cauchy-Lipschitz assumptions of Hypotheses \ref{hyp:CE}, the following strong well-posedness result holds for solutions of continuity equations in $(\Pcal_p(\R^d),W_p(\cdot,\cdot))$. In what follows given some $q \in [1,+\infty]$, we will frequently write $\Norm{\cdot}_q := \NormL{\cdot}{q}{[0,T]}$ for the sake of conciseness.

\begin{thm}[Classical well-posedness of continuity equations]
\label{thm:Wellposedness}
Let $(\tau,\mu_{\tau}) \in [0,T] \times \Pcal_p(\R^d)$ and $v : [0,T] \times \R^d \to \R^d$ be a velocity field satisfying Hypotheses \ref{hyp:CE}-$(i)$. Then, the Cauchy problem \eqref{eq:CE} admits forward solutions $\mu(\cdot) \in \AC([\tau,T],\Pcal_p(\R^d))$, and there exists a constant $c_p > 0$ which only depends on the magnitudes of $p,\Mpazo_p(\mu^0)$ and $\Norm{m(\cdot)}_1$, such that 
\begin{equation}
\label{eq:ACEstimate}
W_p(\mu(t_1),\mu(t_2)) \leq c_p\INTSeg{m(s)}{s}{t_1}{t_2},
\end{equation}
for all times $0 \leq t_1 \leq t_2 \leq T$. If in addition Hypothesis \ref{hyp:CE}-$(ii)$ holds, then the solution of \eqref{eq:CE} is unique, globally defined on $[0,T]$, and can be represented explicitly as 
\begin{equation}
\label{eq:FlowRep}
\mu(t) = \Phi^v_{(\tau,t)}(\cdot)_{\sharp} \mu_{\tau}, 
\end{equation}
for all times $t \in [0,T]$. 
\end{thm}

\begin{proof}
These statements follow e.g. from a combination of several standard results from \cite[Section 8.1]{AGS} along with the moment estimates displayed in Proposition \ref{prop:Moment} below.
\end{proof}

Definition \ref{def:Flows} and Theorem \ref{thm:Wellposedness} inform us together that, in the Cauchy-Lipschitz setting, solutions of continuity equations can be explicitly written as the transports of the initial datum along the characteristic curves generated by the velocity field. When the latter is less regular, it is still possible to give a rigorous meaning to this intuition by leveraging the concept of \textit{superposition measure}. In what follows, we denote by $\Sigma_T := C^0([0,T],\R^d)$ the space of continuous arcs from $[0,T]$ into $\R^d$.

\begin{Def}[Superposition measures]
\label{def:SuperpositionMeas}
An element $\Beta \in \Pcal(\R^d \times \Sigma_T)$ is called a \textnormal{superposition measure} associated with a Lebesgue-Borel velocity field $v : [0,T] \times \R^d \to \R^d$ from time $\tau \in [0,T]$ if it is concentrated on the set of pairs $(x,\sigma) \in \R^d \times \AC([0,T],\R^d)$ satisfying
\begin{equation}
\label{eq:SuperpositionMeasureDef}
\sigma(t) = x + \INTSeg{v(s,\sigma(s))}{s}{\tau}{t} 
\end{equation}
for all times $t \in [\tau,T]$. 
\end{Def}

A direct link can be provided between superposition measures and solutions of \eqref{eq:CE} under the action of the so-called \textit{evaluation map}, which is defined for all times $t \in [0,T]$ by 
\begin{equation*}
\esf_t : (x,\sigma) \in \R^d \times \Sigma_T \mapsto \sigma(t) \in \R^d.
\end{equation*}
More precisely, it can be checked that if $\Beta \in \Pcal(\R^d \times \Sigma_T)$ is a superposition measure associated with a velocity field $v : [0,T] \times \R^d \to \R^d$ such that $(\pi_{\R^d})_{\sharp} \Beta = \mu_{\tau}$ -- where $\pi_{\R^d} : (x,\sigma) \in \R^d \times \Sigma_T \mapsto x \in \R^d$ stands for the projection onto the space component --, and if the following local integrability condition 
\begin{equation*}
\INTSeg{\INTDom{\mathds{1}_K(\sigma(t)) |v(t,\sigma(t))|}{\R^d \times \Sigma_T}{\Beta(x,\sigma)}}{t}{\tau}{T} < +\infty 
\end{equation*}
holds for each compact set $K \subset \R^d$, then the curve defined by $ \mu(t) := (\esf_t)_{\sharp} \Beta$ for all  times $t \in [\tau,T]$ is a solution of \eqref{eq:CE}. We recall in the following theorem the converse of this statement -- colloquially known in the literature as the \textit{superposition principle} --, for which we refer e.g. to \cite[Theorem 3.4]{AmbrosioC2014}.

\begin{thm}[Superposition principle]
\label{thm:Superposition}
Let $(\tau,\mu_{\tau}) \in [0,T] \times \Pcal_p(\R^d)$ and $\mu(\cdot) \in \AC([\tau,T],\Pcal_p(\R^d))$ be a solution of \eqref{eq:CE} driven by a Lebesgue-Borel velocity field $v : [0,T] \times \R^d \rightarrow \R^d$ satisfying 
\begin{equation}
\label{eq:SuperpositionIntegrability}
\INTSeg{\INTDom{\frac{|v(t,x)|}{1+|x|}}{\R^d}{\mu(t)(x)}}{t}{\tau}{T} < +\infty. 
\end{equation}
Then, there exists a superposition measure $\Beta_{\mu} \in \Pcal(\R^d \times \Sigma_T)$ in the sense of Definition \ref{def:SuperpositionMeas} such that $(\esf_t)_{\sharp} \Beta_{\mu} = \mu(t)$ for all times $t \in [\tau,T]$. 
\end{thm}

Throughout the remainder of this section, we derive a series of useful moment and stability estimates for solutions of \eqref{eq:CE} driven by Lebesgue-Borel velocity fields $v : [0,T] \times \R^d \to \R^d$, under the additional assumption that there exists a map $m(\cdot) \in L^1([0,T],\R_+)$ such that 
\begin{equation}
\label{eq:MeasureSublin}
|v(t,x)| \leq m(t)(1+|x|) 
\end{equation}
holds for $\Lcal^1$-almost every $t \in [0,T]$ and $\mu(t)$-almost every $x \in \R^d$. While some of the underlying techniques were already explored in our previous work \cite{ContInc}, the results below are new, proven under less restrictive regularity assumptions, and requested several additional arguments. We start by stating an adaptation of a result established in \cite[Lemma 1]{ContInc}, whose proof\footnote{In \cite{ContInc}, the proof of Lemma 1 contains a small caveat in the definition of the coercive functional used to establish the relative narrow compactness of a relevant sequence of superposition measures. Nevertheless, the conclusion remains correct up to setting the latter to $+\infty$ outside the closed set of pairs $(x,\sigma) \in \R^d \times \Sigma_T$ satisfying $|\dot \sigma(t)| \leq m(t)(1+|\sigma(t)|)$.} follows straightforwardly from the latter up to some minor technical adjustments. 

\begin{lem}[Superposition plans inducing optimal transports]
\label{lem:SuperpositionPlan}
Consider two curves of measures $\mu(\cdot),\nu(\cdot) \in \AC([\tau,T],\Pcal_p(\R^d))$ driven by the Lebesgue-Borel velocity fields $v,w : [0,T] \times \R^d \to \R^d$ complying with the pointwise sublinearity estimates 
\begin{equation*}
|v(t,x)| \leq m(t)(1+|x|) \qquad \text{and} \qquad |w(t,y)| \leq m(t)(1+|y|)
\end{equation*}
for $\Lcal^1$-almost every $t \in [\tau,T]$ and $\mu(t) \times \nu(t)$-almost every $(x,y) \in \R^d \times \R^d$. Moreover, let $\Beta_{\mu},\Beta_{\nu} \in \Pcal(\R^d \times \Sigma_T)$ be two superposition measures given by Theorem \ref{thm:Superposition}, such that
\begin{equation*}
\mu(t) = (\esf_t)_{\sharp} \Beta_{\mu} \qquad \text{and} \qquad \nu(t) = (\esf_t)_{\sharp} \Beta_{\nu}
\end{equation*}
for all times $t \in [\tau,T]$. Then, for every $\gamma_{\tau} \in \Gamma_o(\mu(\tau),\nu(\tau))$, there exists $\hat{\Beta}_{\mu,\nu} \in \Gamma(\Beta_{\mu},\Beta_{\nu})$ such that 
\begin{equation*}
(\pi_{\R^d},\pi_{\R^d})_{\sharp} \hat{\Beta}_{\mu,\nu} = \gamma_{\tau} \qquad \text{and} \qquad (\esf_t , \esf_t)_{\sharp} \hat{\Beta}_{\mu,\nu} \in \Gamma_o(\mu(t),\nu(t)) 
\end{equation*}
for all times $t \in [\tau,T]$.
\end{lem}

In what follows, we leverage the general superposition results of Theorem \ref{thm:Superposition} and Lemma \ref{lem:SuperpositionPlan} to prove moment and equi-integrability inequalities for solutions of \eqref{eq:CE} in Proposition \ref{prop:Moment}, as well as two stability estimates with respect to initial data and driving fields in Proposition \ref{prop:Gronwall}. For the sake of readability, we postpone the proof of these results to Appendices \ref{section:AppendixMoment} and \ref{section:AppendixGronwall} respectively.

\begin{prop}[Moment and equi-integrability estimates]
\label{prop:Moment}
Let $(\tau,\mu_{\tau}) \in \Pcal_p(\R^d)$ be given and $\mu(\cdot) \in \AC([\tau,T],\Pcal_p(\R^d))$ be a solution of \eqref{eq:CE} driven by a Lebesgue-Borel velocity field $v : [0,T] \times \R^d \rightarrow \R^d$ satisfying the sublinearity estimate \eqref{eq:MeasureSublin}. Then, the following moment bound
\begin{equation}
\label{eq:MomentEst}
\Mpazo_p(\mu(t)) \leq C_p \bigg( \Mpazo_p(\mu_{\tau}) + \INTSeg{m(s)}{s}{\tau}{t} \bigg) \exp \Big( C_p' \NormL{m(\cdot)}{1}{[\tau,t]}^p \hspace{-0.05cm} \Big),
\end{equation}
holds for all times $t \in [\tau,T]$, where the constants $C_p,C_p' > 0$ are given explicitly by 
\begin{equation}
\label{eq:ConstantDef}
 C_p := 2^{(p-1)/p} \qquad \text{and} \qquad C_p' := \tfrac{2^{p-1}}{p}.
\end{equation}
In addition for every $R > 0$, the following uniform equi-integrability estimate
\begin{equation}
\label{eq:UnifIntegEst}
\sup_{t \in [\tau,T]} \INTDom{|x|^p \, }{\big\{ x ~ \textnormal{s.t.} \; |x| \geq R \big\}}{\mu(t)(x)} \, \leq \, C_T^p \INTDom{\big( 1+|x| \big)^p \,}{\big\{ x ~ \textnormal{s.t.} \; |x| \geq R/C_T-1 \big\}}{\mu_{\tau}(x)}, 
\end{equation}
holds with $C_T := \max \big\{ 1 , \Norm{m(\cdot)}_1 \hspace{-0.1cm} \big\} \exp \big( \hspace{-0.1cm} \Norm{m(\cdot)}_1 \hspace{-0.1cm} \big)$.  
\end{prop}

\begin{rmk}[A refined moment inequality]
\label{rmk:Moment}
In the sequel, we will often use the fact that, if a velocity field $v : [0,T] \times \R^d \to \R^d$ satisfies a slightly more general sublinearity inequality of the form
\begin{equation*}
|v(t,x)| \leq m(t) \Big( 1 + |x| + M(t) \Big),  
\end{equation*}
for $\Lcal^1$-almost every $t \in [0,T]$ and all $x \in \R^d$, where $M(\cdot) \in L^{\infty}([0,T],\R_+)$ is a priori given, then the corresponding curve of measures $\mu(\cdot) \in \AC([\tau,T],\Pcal_p(\R^d))$ is such that 
\begin{equation*}
\Mpazo_p(\mu(t)) \leq C_p \bigg( \Mpazo_p(\mu_{\tau}) + \INTSeg{m(s)(1+M(s))}{s}{\tau}{t} \bigg) \exp \Big( C_p' \Norm{m(\cdot)}_{L^1([\tau,t])}^p \hspace{-0.05cm} \Big), 
\end{equation*}
for all times $t \in [\tau,T]$. The latter inequality can be established by repeating verbatim the arguments detailed in Appendix \ref{section:AppendixMoment} while replacing $m(\cdot)$ by $m(\cdot)(1+M(\cdot))$.
\end{rmk}

\begin{prop}[Two general stability estimates for continuity equations]
\label{prop:Gronwall}
Let $\mu_{\tau},\nu_{\tau} \in \Pcal_p(\R^d)$ and $\mu(\cdot),\nu(\cdot) \in \AC([\tau,T],\Pcal_p(\R^d))$ be two solutions of \eqref{eq:CE}, driven respectively by a Carathéodory velocity field $v : [0,T] \times \R^d \to \R^d$ satisfying Hypotheses \ref{hyp:CE}, and by a Lebesgue-Borel velocity field $w : [0,T] \times \R^d \to \R^d$ complying with the pointwise sublinearity estimate
\begin{equation}
\label{eq:MeasureSublinBis}
|w(t,y)| \leq m(t) \big( 1+|y| \big), 
\end{equation}
for $\Lcal^1$-almost every $t \in [0,T]$ and $\nu(t)$-almost every $y \in \R^d$. Then, if the map defined by 
\begin{equation*}
t \in [\tau,T] \mapsto \NormL{v(t) - w(t)}{\infty}{\R^d,\R^d ; \, \nu(t)}
\end{equation*} 
is Lebesgue integrable, the following global stability estimate
\begin{equation}
\label{eq:Gronwall1}
W_p(\mu(t),\nu(t)) \leq C_p \bigg( W_p(\mu_{\tau},\nu_{\tau}) + \INTSeg{\NormL{v(s) - w(s)}{\infty}{\R^d,\R^d ; \, \nu(s)}}{s}{\tau}{t} \bigg) \exp \Big( C_p' \NormL{l(\cdot)}{1}{[\tau,t]}^p \hspace{-0.05cm} \Big),
\end{equation}
holds for all times $t \in [\tau,T]$, wherein $C_p,C_p'>0$ are defined in \eqref{eq:ConstantDef}. More generally, under our assumptions, the application 
\begin{equation*}
t \in [\tau,T] \mapsto \NormL{v(t) - w(t)}{\infty}{B(0,R),\R^d ;\, \nu(t)},
\end{equation*}
is Lebesgue integrable for every $R >0$, and the following localised stability estimate 
\begin{equation}
\label{eq:Gronwall2}
\begin{aligned}
W_p(\mu(t),\nu(t)) & \leq C_p \Bigg( W_p(\mu_{\tau},\nu_{\tau}) + \INTSeg{\NormL{v(s) - w(s)}{\infty}{B(0,R),\R^d ;\, \nu(s)}}{s}{\tau}{t} + \Epazo_{\nu}(\tau,t,R) \Bigg) \\
& \hspace{8.5cm} \times \exp \Big( C_p' \NormL{l(\cdot)}{1}{[\tau,t]}^p \hspace{-0.05cm} \Big),
\end{aligned}
\end{equation}
holds for all times $t \in [\tau,T]$, where the additional error term is given by  
\begin{equation}
\label{eq:ErrorTerm}
\Epazo_{\nu}(\tau,t,R) := 2 \NormL{m(\cdot)}{1}{[\tau,t]} (1+C_T) \bigg( \INTDom{(1+|y|)^p}{\{ y \; \textnormal{s.t.}\; |y| \geq R/C_T-1 \}}{\nu_{\tau}(y)} \bigg)^{1/p}.
\end{equation}
with $C_T := \max\{ 1,\Norm{m(\cdot)}_1 \} \exp(\Norm{m(\cdot)}_1)$. 
\end{prop}

\begin{rmk}[Comparison with the estimates of \cite{ContInc}]
The stability estimates displayed in Proposition \ref{prop:Gronwall} above improve on those of \cite{ContInc} in the two following ways. Firstly, the global inequality \eqref{eq:Gronwall1} now holds for a general Lebesgue-Borel velocity field $w : [0,T] \times \R^d \to \R^d$ satisfying the sublinearity inequality \eqref{eq:MeasureSublin}, without requiring the latter to be Carathéodory. Secondly, the inequality \eqref{eq:Gronwall2} involving the quantitative error term $\Epazo_{\nu}(\tau,t,R)$ is completely new to the best of our knowledge, and will prove crucial to palliate the fact that the metric of compact convergence $\dcc(\cdot,\cdot)$ only grants access to local and non-uniform discrepancy estimates between functions.   
\end{rmk}

We end this preliminary section by recalling the definition of \textit{continuity inclusions} in Wasserstein spaces. The latter was introduced by the authors of the present manuscript in \cite{ContInc} as a natural set-valued generalisation of continuity equations. Indeed, its definition adopts the well-established viewpoint that solutions of differential inclusions should be understood as absolutely continuous curves whose derivatives are measurable selections of admissible velocities, see e.g. \cite[Chapter 10]{Aubin1990}.

\begin{Def}[Continuity inclusions in $(\Pcal_p(\R^d),W_p(\cdot,\cdot))$]
\label{def:ContInc}
Let $V : [0,T] \times \Pcal_p(\R^d) \tto C^0(\R^d,\R^d)$ be a set-valued map. We say that a curve $\mu(\cdot) \in \AC([0,T],\Pcal_p(\R^d))$ solves the \textnormal{continuity inclusion}
\begin{equation}
\label{eq:ContIncDef}
\partial_t \mu(t) \in - \Div_x \Big( V(t,\mu(t)) \mu(t) \Big)
\end{equation}
if there exists an $\Lcal^1$-measurable selection $t \in [0,T] \mapsto v(t) \in V(t,\mu(t)) \subset C^0(\R^d,\R^d)$ such that the \textnormal{trajectory-selection pair} $(\mu(\cdot),v(\cdot))$ is a solution of the continuity equation
\begin{equation*}
\partial_t \mu(t) + \Div_x (v(t)\mu(t)) = 0, 
\end{equation*}
in the sense of distributions. 
\end{Def}


\section{Existence à la Peano in the Carathéodory framework}
\label{section:Peano}
\setcounter{equation}{0} \renewcommand{\theequation}{\thesection.\arabic{equation}}

In this section, we establish a general existence result for set-valued Cauchy problems of the form
\begin{equation}
\label{eq:ContIncCarathéodory}
\left\{
\begin{aligned}
& \partial_t \mu(t) \in - \Div_x \Big( V(t,\mu(t)) \mu(t) \Big), \\ 
& \mu(0) = \mu^0, 
\end{aligned}
\right.
\end{equation}
formulated in the Wasserstein space $(\Pcal_p(\R^d),W_p(\cdot,\cdot))$, under the Carathéodory regularity assumptions listed below. Our strategy is based on a careful adaptation of the semi-discrete Euler scheme due to Filippov for differential equations, see e.g. \cite[Chapter 1 -- Theorem 1]{Filippov2013}. To the best of our knowledge, this approach is completely new, as it does not seem to have been investigated for differential inclusions even in finite-dimensional vector spaces. In the context of measure dynamics, the latter can also be seen as a kind of relative to the methods developed e.g. in \cite{CavagnariSS2022,Piccoli2019,Pedestrian}. We would also like to stress that our proof strategy greatly improves on several related compactness arguments explored in the literature of meanfield control, see e.g. \cite{ContInc,Fornasier2014}, thanks to the weak $L^1$-compactness criterion of Theorem \ref{thm:L1WeakCompactness}. 



\begin{taggedhyp}{\textbn{(P)}}
\label{hyp:CIPC}
\hfill
\begin{enumerate}
\item[$(i)$] The set-valued map $(t,\mu) \in [0,T] \times \Pcal_p(\R^d) \tto V(t,\mu) \subset C^0(\R^d,\R^d)$ is Carathéodory with nonempty closed and convex images.
\item[$(ii)$] There exists a map $m(\cdot) \in L^1([0,T],\R_+)$ such that for $\Lcal^1$-almost every $t \in [0,T]$, any $\mu \in \Pcal_p(\R^d)$ and each $v \in V(t,\mu)$, there holds 
\begin{equation*}
|v(x)| \leq m(t) \Big(1 + |x| + \Mpazo_p(\mu) \Big).
\end{equation*}
for all $x \in \R^d$.
\item[$(iii)$] For all compact sets $\Kcal \subset \Pcal_p(\R^d)$ and $K \subset \R^d$, there exist a map $l(\cdot) \in L^1([0,T],\R_+)$ and a continuous modulus of continuity $\omega : \R_+ \to \R_+$ such that for $\Lcal^1$-almost every $t \in [0,T]$, any $\mu \in \Kcal$ and each $v \in V(t,\mu)$, there holds 
\begin{equation*}
|v(t,x) - v(t,y)| \leq l(t) \omega(|x-y|)
\end{equation*}
for all $x,y \in K$.
\item[$(iv)$] For every compact set $\Kcal \subset \Pcal_p(\R^d)$, there exists a map $L(\cdot) \in L^1([0,T],\R_+)$ and a continuous modulus of continuity $\omega : \R_+ \to \R_+$ such that for $\Lcal^1$-almost every $t \in [0,T]$, any $\mu,\nu \in \Kcal$ and each $v \in V(t,\mu)$, there exists $w \in V(t,\nu)$ such that 
\begin{equation*}
\dcc(v,w) \leq L(t) \, \omega(W_p(\mu,\nu)).
\end{equation*}
\end{enumerate}
\end{taggedhyp}

Two remarks are in order concerning the previous set of assumptions. First, we stress that the maps $l(\cdot),L(\cdot) \in L^1([0,T],\R_+)$ and the local continuity moduli $\omega : \R_+ \to \R_+$ appearing in Hypotheses \ref{hyp:CIPC}-$(iii)$ and $(iv)$ are all localised on compact sets. That being said, we did not make this dependence transparent so as to keep the notations as light as possible. Second, as mentioned in the introduction, it is a well-known fact in the theory of differential inclusions that outside of the Cauchy-Lipschitz framework, one must in general impose a convexity assumption on the admissible velocities to establish the existence of solutions, even in finite-dimensional euclidean spaces (see e.g. \cite[Chapter 2]{Aubin1984}). 

\begin{rmk}[Examples of set-valued mapping satisfying our assumptions]
A relevant example of set-valued map $V : [0,T] \times \Pcal_p(\R^d) \tto C^0(\R^d,\R^d)$ satisfying Hypotheses \ref{hyp:CIPC} is given by the set of admissible velocities of a controlled system of the form
\begin{equation*}
V(t,\mu) := \Big\{ v(t,\mu,u) \in C^0(\R^d,\R^d) ~\, \textnormal{s.t.}~ u \in U \Big\}. 
\end{equation*}
Therein, $(U,d_U(\cdot,\cdot))$ is a compact metric space representing admissible control inputs, while $v : [0,T] \times \Pcal_p(\R^d) \times U \times \R^d \to \R^d$ is a vector field that is Carathéodory in $(t,u) \in [0,T] \times U$ as well as locally uniformly continuous in $(\mu,x) \in \Pcal_p(\R^d) \times \R^d$ with constants that are Lebesgue integrable functions of time. In the simpler case in which the maps $(\mu,x) \in \Pcal_p(\R^d) \times \R^d \mapsto v(t,\mu,u,x) \in \R^d$ are globally Lipschitz with constants given by Lebesgue integrable functions of time, this control theoretic model also fits Hypotheses \ref{hyp:CICL} of Section \ref{section:Filippov}.
\end{rmk} 


In what follows, we state and prove the main result of this section which provides the existence of solutions to \eqref{eq:ContIncCarathéodory} under Hypotheses \ref{hyp:CIPC}. 

\begin{thm}[Existence à la Peano for continuity inclusions]
\label{thm:Peano}
Let $V : [0,T] \times \Pcal_p(\R^d) \tto C^0(\R^d,\R^d)$ be a set-valued map satisfying Hypotheses \ref{hyp:CIPC}. Then for every $\mu^0 \in \Pcal_p(\R^d)$, the Cauchy problem
\begin{equation*}
\left\{
\begin{aligned}
& \partial_t \mu(t) \in - \Div_x \Big( V(t,\mu(t)) \mu(t) \Big), \\
& \mu(0) = \mu^0,
\end{aligned}
\right.
\end{equation*}
admits at least one solution $\mu(\cdot) \in \AC([0,T],\Pcal_p(\R^d))$.
\end{thm}

The proof of Theorem \ref{thm:Peano} follows a constructive scheme which is split into five steps. We start in Step 1 by constructing a sequence of trajectory-selection pairs solving continuity equations with delayed velocity inclusions, and proceed by showing in Step 2 that the elements of the latter comply with uniform moment, regularity and equi-integrability bounds. In Step 3, we then prove that such estimates imply the existence of suitable weak cluster points for the sequence of trajectory-selection pairs. To conclude the proof, we further establish in Step 4 that the limit curve of measures solves a continuity equation driven by the corresponding velocity selection, and finally prove in Step 5 that the latter is in turn a measurable selection in the set of admissible velocities.

\begin{proof}[Proof of Theorem \ref{thm:Peano}]
In what follows, our goal is to build a sequence of trajectory-selection pairs $(\mu_n(\cdot),v_n(\cdot)) \in \AC([0,T],\Pcal_p(\R^d)) \times \Lcal([0,T],C^0(\R^d,\R^d))$ solutions of the Cauchy problems
\begin{equation}
\label{eq:PeanoCauchy}
\left\{
\begin{aligned}
& \partial_t \mu_n(t) + \Div_x (v_n(t) \mu_n(t)) = 0, \\
& \mu_n(0) = \mu^0.
\end{aligned}
\right.
\end{equation}
These latter will be chosen so as to satisfy the delayed pointwise velocity inclusion
\begin{equation}
\label{eq:PeanoInclusion}
v_n(t) \in V \Big( t,\mu_n \big(t- \tfrac{T}{n} \big) \Big),
\end{equation}
for $\Lcal^1$-almost every $t \in [0,T]$ -- where here and in what follows we set $\mu_n(t) := \mu^0$ for $t \in [-\tfrac{T}{n},0]$ by convention --, along with the uniform moment and regularity bounds 
\begin{equation}
\label{eq:PeanoBounds}
\Mpazo_p(\mu_n(t)) \leq \Cpazo \qquad \text{and} \qquad W_p(\mu_n(\tau),\mu_n(t)) \leq c_p \INTSeg{m(s)}{s}{\tau}{t}
\end{equation}
for all times $0 \leq \tau \leq t \leq T$ and each $n \geq 1$. Therein, the constants $\Cpazo,c_p > 0$ only depend on the magnitudes of $p,\Mpazo_p(\mu^0)$ and $\Norm{m(\cdot)}_1$.


\paragraph*{Step 1 -- Construction of the sequence.} Given an integer $n \geq 1$, we explicitly build the pair $(\mu_n(\cdot),v_n(\cdot))$ satisfying the aforedescribed conditions by performing an induction on $k \in \{0,\dots,n-1\}$. First, let $I_n^0 := [0,\tfrac{T}{n}]$ and observe that under Hypotheses \ref{hyp:CIPC}-$(i)$ and $(ii)$, the set-valued map 
\begin{equation*}
t \in I_n^0 \tto V(t,\mu^0) \subset C^0(\R^d,\R^d) 
\end{equation*}
is $\Lcal^1$-measurable with nonempty and closed images for the topology of local uniform convergence. Hence, it admits a measurable selection $t \in I_n^0 \mapsto v^0_n(t) \in V(t,\mu^0)$ by Theorem \ref{thm:Selection}, and the underlying Carathéodory velocity field $v^0_n : I_n^0 \times \R^d \mapsto \R^d$ satisfies the sublinearity estimate 
\begin{equation*}
|v^0_n(t,x)| \leq m(t) \Big( 1 + |x| + \Mpazo_p(\mu^0) \Big),
\end{equation*}
for $\Lcal^1$-almost every $t \in I_n^0$ and all $x \in \R^d$, as a consequence of Hypothesis \ref{hyp:CIPC}-$(ii)$ along with Lemma \ref{lem:Carathéodory}. It complies in particular with Hypothesis \ref{hyp:CE}-$(i)$, and by Theorem \ref{thm:Wellposedness} applied on the interval $I := I_n^0$ with $v(t,x) := v^0_n(t,x)$, there exists a solution $\mu_n^0(\cdot) \in \AC(I_n^0,\Pcal_p(\R^d))$ to the Cauchy problem
\begin{equation*}
\left\{
\begin{aligned}
& \partial_t \mu^0_n(t) + \Div_x(v^0_n(t) \mu_n^0(t)) = 0, \\
& \mu_n^0(0) = \mu^0.
\end{aligned}
\right.
\end{equation*}
By repeating this process for $k \in \{1,\dots,n-1\}$, we detail below how one can inductively build a family of trajectory-selection pairs $(\mu_n^k(\cdot),v_n^k(\cdot)) \in \AC(I_k,\Pcal_p(\R^d)) \times \Lcal(I_k,C^0(\R^d,\R^d))$ defined over the family of time intervals $I_n^k := [\tfrac{kT}{n},\tfrac{(k+1)T}{n}]$. 

First, note that for $\Lcal^1$-almost every $t \in [0,T]$, the sets $V(t,\mu^{k-1}_n(t)) \subset C^0(\R^d,\R^d)$ are compact for the topology induced by $\dsf_{cc}(\cdot,\cdot)$ as a consequence of Hypotheses \ref{hyp:CIPC}-$(ii)$ and $(iii)$ combined with the Ascoli-Arzel\`a theorem. Second, under Hypothesis \ref{hyp:CIPC}-$(i)$, it follows from Lemma \ref{lem:MeasurableSel}-$(c)$ that
\begin{equation*}
t \in I_n^k \tto V \Big(t,\mu^{k-1}_n(t-\tfrac{T}{n}) \Big) \subset C^0(\R^d,\R^d)
\end{equation*}
is an $\Lcal^1$-measurable set-valued map, whose images are nonempty. Thence, by what precedes, it admits a measurable selection
\begin{equation}
\label{eq:PeanoInductionVel}
t \in I_n^k \mapsto v^k_n(t) \in V \Big( t , \mu_n^{k-1} \big( t - \tfrac{T}{n} \big)\Big). 
\end{equation}
Besides by Lemma \ref{lem:Carathéodory} and Hypothesis \ref{hyp:CIPC}-$(ii)$, the vector field $v_n^k : I_n^k \times \R^d \to \R^d$ is Carathéodory and satisfies the sublinearity estimate
\begin{equation}
\label{eq:PeanoSublin}
|v_n^k(t,x)| \leq m(t) \Big( 1 + |x| + \Mpazo_p \big( \mu_n^{k-1} \big( t - \tfrac{T}{n} \big) \big) \Big)
\end{equation}
for $\Lcal^1$-almost every $t \in I_n^k$, all $x \in \R^d$ and each $k \in \{1,\dots,n-1\}$. Thus by applying Theorem \ref{thm:Wellposedness} on the time interval $I := I_n^k$ with $v(t,x) := v_n^k(t,x)$, one can subsequently define $\mu_n^k(\cdot) \in \AC(I_n^k,\Pcal_p(\R^d))$ as being one of the solutions of the Cauchy problem
\begin{equation*}
\left\{
\begin{aligned}
& \partial_t \mu_n^k(t) + \Div_x (v_n^k(t)\mu_n^k(t)) = 0, \\
& \mu_n^k (\tfrac{kT}{n}) = \mu_n^{k-1} \big( \tfrac{kT}{n} \big).
\end{aligned}
\right.
\end{equation*}
By classical concatenation properties for solutions of continuity equations (see e.g. \cite[Lemma 4.4]{Dolbeault2009}), the trajectory-selection pair $(\mu_n(\cdot),v_n(\cdot)) \in \AC([0,T],\Pcal_p(\R^d)) \times \Lcal([0,T],C^0(\R^d,\R^d))$ defined by 
\begin{equation}
\label{eq:PeanoGlobalPair}
\mu_n(t) := \mu_n^k(t) \qquad \text{and} \qquad v_n(t) := v_n^k(t), 
\end{equation}
for $t \in I_n^k$ and $k \in \{0,\dots,n-1\}$ solves the Cauchy problem \eqref{eq:PeanoCauchy}. Moreover, it can be checked using \eqref{eq:PeanoInductionVel} that this trajectory-selection pair satisfies the shifted pointwise inclusion \eqref{eq:PeanoInclusion}. 


\paragraph*{Step 2 -- Uniform moment and regularity estimates.} Our goal now is to establish the uniform regularity and moment bounds posited in \eqref{eq:PeanoBounds}. First, notice that as a consequence of the construction detailed in Step 1 and Proposition \ref{prop:Moment}, the curves $(\mu_n(\cdot)) \subset \AC([0,T],\Pcal_p(\R^d))$ satisfy
\begin{equation*}
\sup_{t \in [0,T]} \Mpazo_p(\mu_n(t)) < +\infty 
\end{equation*}
for each $n \geq 1$. This allows us to apply the moment estimate of Remark \ref{rmk:Moment}, which in context writes
\begin{equation}
\label{eq:PeanoInductionMoment}
\Mpazo_p(\mu_n(t)) \leq C_p \bigg( \Mpazo_p(\mu^0) + \INTSeg{m(s) \Big( 1 + \Mpazo_p \big(\mu_n \big( s - \tfrac{T}{n} \big) \big) \Big)}{s}{0}{t} \bigg) \exp \Big( C_p' \NormL{m(\cdot)}{1}{[0,t]}^p \hspace{-0.05cm} \Big),
\end{equation}
for all times $t \in [0,T]$, using again the convention that $\mu_n(t) = \mu^0$ when $t \in [-\tfrac{T}{n},0]$. Defining the map $\hat{m}(\cdot) \in L^1([0,2T],\R_+)$ as
\begin{equation}
\label{eq:hatm}
\hat{m}(t) := \left\{
\begin{aligned}
& m(t) ~~ & \text{if $t \in [0,T]$,} \\
& 0 ~~ & \text{if $t \in [T,2T]$}, 
\end{aligned}
\right.
\end{equation}
one may rewrite the estimate displayed in \eqref{eq:PeanoInductionMoment} as 
\begin{equation*}
\begin{aligned}
\Mpazo_p(\mu_n(t)) & \leq C_p \Bigg( \Mpazo_p(\mu^0) + \INTSeg{m(s)}{s}{0}{t} + \INTSeg{m(s) \Mpazo_p(\mu^0)}{s}{0}{T/n} \\
& \hspace{2.6cm} + \INTSeg{m \big(s + \tfrac{T}{n} \big) \Mpazo_p(\mu_n(s))}{s}{0}{\max \big\{ 0, \, t-T/n \big\}} \Bigg) \exp \Big( C_p' \NormL{m(\cdot)}{1}{[0,t]}^p \Big) \\ 
& \leq C_p \big( 1+\Mpazo_p(\mu^0) \big) \Big( 1 \, + \NormL{m(\cdot)}{1}{[0,t]} \hspace{-0.05cm} \Big) \exp \Big( C_p' \NormL{m(\cdot)}{1}{[0,t]}^p \Big) \\
& \hspace{1.3cm} + C_p \bigg( \INTSeg{\hat{m} \big(s + \tfrac{T}{n} \big) \Mpazo_p(\mu_n(s))}{s}{0}{t} \bigg)  \exp \Big( C_p' \NormL{m(\cdot)}{1}{[0,t]}^p \hspace{-0.05cm} \Big),
\end{aligned}
\end{equation*}
where the first inequality follows from a simple change of variable, while the second one stems from the non-negativity of $\hat{m}(\cdot)$ and $\Mpazo_p(\mu_n(\cdot))$. Thence, a standard application of Gr\"onwall's lemma yields 
\begin{equation*}
\begin{aligned}
\Mpazo_p(\mu_n(t)) & \leq C_p \big( 1 + \Mpazo_p(\mu^0) \big) \Big( 1 \, + \NormL{m(\cdot)}{1}{[0,t]} \hspace{-0.05cm} \Big) \\
& \hspace{1.4cm} \times \exp \Bigg( C_p' \NormL{m(\cdot)}{1}{[0,t]}^p + \, C_p \NormL{\hat{m}(\cdot)}{1}{[0,t+T/n]} \exp \Big( C_p' \NormL{m(\cdot)}{1}{[0,t]}^p \Big) \Bigg)
\end{aligned}
\end{equation*}
for all times $t \in [0,T]$. In turn, upon noticing that
\begin{equation*}
\NormL{\hat{m}(\cdot)}{1}{[0,t+T/n]} ~\leq~ \Norm{m(\cdot)}_1
\end{equation*}
for all times $t \in [0,T]$ and each $n \geq 1$ as a consequence of \eqref{eq:hatm}, there further exists a constant $\Cpazo > 0$ which only depends on the magnitudes of $p,\Mpazo_p(\mu^0)$ and $\Norm{m(\cdot)}_1$, such that 
\begin{equation*}
\sup_{t \in [0,T]} \Mpazo_p(\mu_n(t)) \leq \Cpazo
\end{equation*}
for each $n \geq 1$. In particular, the velocity fields $(v_n(\cdot)) \subset \Lcal([0,T],C^0(\R^d,\R^d))$ satisfy the uniform sublinearity bounds 
\begin{equation}
\label{eq:PeanoUnifSublinSeq}
|v_n(t,x)| \leq (1+\Cpazo) m(t) \big( 1 + |x| \big)
\end{equation}
for $\Lcal^1$-almost every $t \in [0,T]$, all $x \in \R^d$ and each $n \geq 1$. Therefore, by Theorem \ref{thm:Wellposedness}, there exists a constant $c_p > 0$ depending only on the magnitudes of $p,\Mpazo_p(\mu^0)$ and $\Norm{m(\cdot)}_1$ such that 
\begin{equation*}
W_p(\mu_n(\tau),\mu_n(t)) \leq c_p \INTSeg{m(s)}{s}{\tau}{t}
\end{equation*}
for all times $0 \leq \tau \leq t \leq T$. To summarise, we have shown that the sequence of curves $(\mu_n(\cdot)) \subset \AC([0,T],\Pcal_p(\R^d))$ satisfies the uniform moment and absolute continuity estimates \eqref{eq:PeanoBounds}.


\paragraph*{Step 3 -- Relative compactness of the sequence of trajectory-selection pairs.}

Next, we show that the regularity estimates of Step 2 yield the relative compactness of the sequence of pairs $(\mu_n(\cdot),v_n(\cdot)) \subset \AC([0,T],\Pcal_p(\R^d)) \times \Lcal([0,T],C^0(\R^d,\R^d))$ built in Step 1. Indeed, notice first that by invoking the moment inequality \eqref{eq:MomentEst} and the equi-integrability bounds \eqref{eq:UnifIntegEst} of Proposition \ref{prop:Moment}, possibly with different constants $C_T > 0$, one may infer from Proposition \ref{prop:WassDef} the existence of a compact set $\Kcal \subset \Pcal_p(\R^d)$, depending only on $\mu^0$ and $\Norm{m(\cdot)}_1$, such that 
\begin{equation}
\label{eq:PeanoCompactInc}
\mu_n(t) \in \Kcal
\end{equation}
for all times $t \in [0,T]$. Whence, taking into account the uniform regularity estimate \eqref{eq:PeanoBounds} satisfied by the sequence $(\mu_n(\cdot)) \subset \AC([0,T],\Pcal_p(\R^d))$, it follows from the Ascoli-Arzel\`a theorem for complete separable metric spaces (see e.g. \cite[Chapter 7 -- Theorem 18]{Kelley1975}) that 
\begin{equation}
\label{eq:PeanoUnifConv}
\sup_{t \in [0,T]} W_p(\mu_n(t),\mu(t)) \underset{n \to +\infty}{\longrightarrow}~ 0, 
\end{equation}
for some curve of measures $\mu(\cdot) \in \AC([0,T],\Pcal_p(\R^d))$, along a subsequence that we do not relabel. 

Concerning the sequence of velocity selections, observe that for each $m \geq 1$, the space restrictions to $B(0,m)$ of $(v_n(\cdot)) \subset \Lcal([0,T],C^0(\R^d,\R^d))$ form an integrably bounded sequence owing to the sublinearity estimate of \eqref{eq:PeanoUnifSublinSeq}. Besides, by Hypothesis \ref{hyp:CIPC}-$(iii)$ combined with \eqref{eq:PeanoCompactInc} and the classical Ascoli-Arzel\`a theorem, there exists a family of compact sets $(\Kpazo_t^m)_{t \in [0,T]} \subset C^0(B(0,m),\R^d)$ such that 
\begin{equation*}
V \Big(t,\mu_n(t-\tfrac{T}{n}) \Big)_{|B(0,m)} := \bigg\{ v_{|B(0,m)} \in C^0(B(0,m),\R^d) \, ~\textnormal{s.t.}~ v \in V \Big(t,\mu_n(t-\tfrac{T}{n}) \Big) \bigg\} \subset \Kpazo_t^m
\end{equation*}
for $\Lcal^1$-almost every $t \in [0,T]$ and each $n,m \geq 1$, wherein $v_{|B(0,m)}$ stands for the restriction of the function $v \in C^0(\R^d,\R^d)$ to $B(0,m)$. Thus, by iteratively applying Theorem \ref{thm:L1WeakCompactness} while performing a standard diagonal argument, we obtain the existence of a Carathéodory vector field $v(\cdot) \in \Lcal([0,T],C^0(\R^d,\R^d))$ such that for every $R>0$, one has that
\begin{equation}
\label{eq:PeanoVelConv}
\qquad \quad v_n(\cdot) ~\underset{n \to +\infty}{\rightharpoonup}~ v(\cdot)
\end{equation}
weakly in $L^1([0,T],C^0(B(0,R),\R^d))$, along a subsequence that depends on $R>0$. In particular, this convergence property implies that 
\begin{equation}
\label{eq:PeanoVelConvParticular}
\INTSeg{\big\langle \Bnu(t) , v(t) - v_n(t) \big\rangle_{C^0(B(0,R),\R^d)}}{t}{0}{T} ~\underset{n \to +\infty}{\longrightarrow}~ 0
\end{equation}
whenever $\Bnu(\cdot) \in L^{\infty}([0,T],\Mcal(B(0,R),\R^d))$, again as a consequence of Theorem \ref{thm:L1WeakCompactness}. Besides, it can be straightforwardly verified that 
\begin{equation*}
\Mpazo_p(\mu(t)) \leq \Cpazo \qquad \text{and} \qquad W_p(\mu(\tau),\mu(t)) \leq c_p \INTSeg{m(s)}{s}{\tau}{t}
\end{equation*}
for all times $t \in [0,T]$ as a consequence of \eqref{eq:PeanoUnifConv}, while \eqref{eq:PeanoUnifSublinSeq} and \eqref{eq:PeanoVelConv} yield up to an application of Mazur's lemma (see e.g. \cite[Corollary 3.8]{Brezis}) that
\begin{equation}
\label{eq:PeanoLimitSublin}
|v(t,x)| \leq (1+\Cpazo) m(t) \big( 1 + |x| \big), 
\end{equation}
for $\Lcal^1$-almost every $t \in [0,T]$ and all $x \in \R^d$.


\paragraph*{Step 4 -- Dynamics of the limit trajectory-selection pair.} At this stage, we need to show that the limit trajectory-selection pair $(\mu(\cdot),v(\cdot)) \in \AC([0,T],\Pcal_p(\R^d)) \times \Lcal([0,T],C^0(\R^d,\R^d))$ stemming from the compactness argument of Step 3 is a distributional solution of the Cauchy problem 
\begin{equation}
\label{eq:PeanoIntermCauchy}
\left\{
\begin{aligned}
& \partial_t \mu(t) + \Div_x(v(t) \mu(t)) = 0, \\
& \mu(0) = \mu^0.
\end{aligned}
\right.
\end{equation}
To begin with, note that it trivially follows from \eqref{eq:PeanoUnifConv} along with the construction detailed in Step 1 that $\mu(0) = \mu^0$. Thus, there remains to show that one may pass to the limit as $n \to +\infty$ in the weak formulation of \eqref{eq:PeanoCauchy}, of which we remind that it is given by 
\begin{equation}
\label{eq:PeanoIntermDistrib}
\INTSeg{\INTDom{\Big( \partial_t \phi(t,x) + \big\langle \nabla_x \phi(t,x) , v_n(t,x) \big\rangle \Big)}{\R^d}{\mu_n(t)(x)}}{t}{0}{T} = 0
\end{equation}
for all $\phi \in C^{\infty}_c((0,T) \times \R^d)$ and each $n \geq 1$. By leveraging the uniform convergence property of \eqref{eq:PeanoUnifConv} together with Proposition \ref{prop:WassDef}, it can be checked straightforwardly that 
\begin{equation}
\label{eq:PeanoWeakConvergenceVel1}
\INTSeg{\INTDom{\partial_t \phi(t,x)}{\R^d}{\mu_n(t)(x)}}{t}{0}{T} ~\underset{n \to +\infty}{\longrightarrow}~ \INTSeg{\INTDom{\partial_t \phi(t,x)}{\R^d}{\mu(t)(x)}}{t}{0}{T}.
\end{equation}
In addition, upon noting that the convergence result provided in \cite[Theorem 12.1.1]{AGS} for bounded Carathéodory vector fields is still valid for integrably bounded ones -- up to a relevant adaptation in the spirit of \cite[Lemma 5.1.7]{AGS} --, it follows from \eqref{eq:PeanoLimitSublin} combined with \eqref{eq:PeanoUnifConv} and Proposition \ref{prop:WassDef} that
\begin{equation}
\label{eq:PeanoWeakConvergenceVel2}
\INTSeg{\INTDom{\big\langle \nabla_x \phi(t,x) , v(t,x) \big\rangle}{\R^d}{\mu_n(t)(x)}}{t}{0}{T} ~\underset{n \to +\infty}{\longrightarrow}~ \INTSeg{\INTDom{\big\langle \nabla_x \phi(t,x) , v(t,x) \big\rangle}{\R^d}{\mu(t)(x)}}{t}{0}{T}.
\end{equation}
Thus, by merging \eqref{eq:PeanoWeakConvergenceVel1} and \eqref{eq:PeanoWeakConvergenceVel2}, one can check that in order recover the weak form of the dynamics in \eqref{eq:PeanoIntermCauchy} by letting $n \to +\infty$ in \eqref{eq:PeanoIntermDistrib}, there only remains to show that 
\begin{equation}
\label{eq:PeanoDistribGoal}
\INTSeg{\INTDom{\big\langle \nabla_x \phi(t,x) , v(t,x) - v_n(t,x) \big\rangle}{\R^d}{\mu_n(t)(x)}}{t}{0}{T} ~\underset{n \to +\infty}{\longrightarrow}~ 0
\end{equation}
for each $\phi \in C^{\infty}_c((0,T) \times \R^d,\R)$, possibly up to a subsequence. To do so, we consider test functions of the form 
\begin{equation*}
\phi(t,x) := \zeta(t) \psi(x)
\end{equation*}
for all $(t,x) \in [0,T] \times \R^d$ and some $(\zeta,\psi) \in C^{\infty}_c((0,T),\R) \times C^{\infty}_c(\R^d,\R)$, whose linear span is dense in $C^{\infty}_c((0,T)\times \R^d,\R)$ (see e.g. \cite[Chapter 8]{AGS}). We then let $R_{\psi} > 0$ be such that $\supp(\psi) \subset B(0,R_{\psi})$, and note that by Step 3, there exists a subsequence of $(v_n(\cdot)) \subset \Lcal([0,T],C^0(\R^d,\R^d))$ that we do not relabel for which \eqref{eq:PeanoVelConv} holds in $L^1([0,T],C^0(B(0,R),\R^d))$, with 
\begin{equation}
\label{eq:Rdef}
R := (1+R_{\psi}) \max\Big\{ 1 \, , (1+\Cpazo) \Norm{m(\cdot)}_1 \Big\} \exp \Big( (1+\Cpazo) \Norm{m(\cdot)}_1 \Big).
\end{equation}
In that case, we claim that 
\begin{equation}
\label{eq:PeanoSuperposition0}
\begin{aligned}
& \sup_{\sigma \in \AC([0,T],\R^d)} \Bigg\{ \bigg| \INTSeg{ \zeta(t) \, \big \langle \nabla \psi(\sigma(t)), v(t,\sigma(t)) - v_n(t,\sigma(t)) \big\rangle}{t}{0}{T} \, \bigg| ~\, \textnormal{s.t.}~ \dot \sigma(t) = v_n(t,\sigma(t)) \\
& \hspace{9.5cm} \text{for $\Lcal^1$-almost every $t \in [0,T]$} \Bigg\} \,\underset{n \to +\infty}{\longrightarrow}\, 0.
\end{aligned}
\end{equation}
Indeed, suppose by contradiction that there exists some $\epsilon > 0$, a subsequence of Carathéodory vector fields $(v_{n_k}(\cdot)) \subset \Lcal([0,T],C^0(\R^d,\R^d))$ and a sequence of curves $(\sigma_k(\cdot)) \subset \AC([0,T],\R^d)$ such that 
\begin{equation*}
\dot \sigma_k(t) = v_{n_k}(t,\sigma_k(t)) 
\end{equation*}
for $\Lcal^1$-almost every $t \in [0,T]$, and
\begin{equation}
\label{eq:PeanoSuperposition1}
\bigg| \INTSeg{ \zeta(t) \, \big\langle \nabla \psi(\sigma_k(t)), v(t,\sigma_k(t)) - v_{n_k}(t,\sigma_k(t)) \big\rangle \big|}{t}{0}{T} \, \bigg| \geq \epsilon
\end{equation}
for each $k \geq 1$. Then, there must exist a sequence $(\tau_k) \subset \supp(\zeta)$ such that $\sigma_k(\tau_k) \in B(0,R_{\psi})$ for every $k \geq 1$, which together with the estimates of \eqref{eq:PeanoUnifSublinSeq} and Gr\"onwall's lemma implies that 
\begin{equation*}
|\sigma_k(t)| \leq R \qquad \text{and} \qquad |\sigma_k(t) - \sigma_k(\tau)| \leq (1+\Cpazo)(1+R) \INTSeg{m(s)}{s}{\tau}{t}
\end{equation*}
for all times $0 \leq \tau \leq t \leq T$ and each $k \geq 1$, where $R>0$ is given as in \eqref{eq:Rdef}. Thence, by the Ascoli-Arzel\`a theorem, there exists a curve $\sigma(\cdot) \in \AC([0,T],\R^d)$ such that 
\begin{equation}
\label{eq:PeanoSuperposition1bis}
\sup_{t \in [0,T]} |\sigma(t) - \sigma_k(t)| ~\underset{k \to +\infty}{\longrightarrow}~ 0,
\end{equation}
along a subsequence which we do not relabel. Besides, it follows from Hypothesis \ref{hyp:CIPC}-$(iii)$ and Lebesgue's dominated convergence theorem that 
\begin{equation}
\label{eq:PeanoSuperposition2}
\sup_{n \geq 1} \INTSeg{|v_n(t,\sigma_k(t)) - v_n(t,\sigma(t))|}{t}{0}{T} ~\underset{k \to +\infty}{\longrightarrow}~ 0, 
\end{equation}
which together with \eqref{eq:PeanoSuperposition1} further implies that 
\begin{equation}
\label{eq:PeanoSuperposition2bis}
\bigg| \INTSeg{\zeta(t) \, \big\langle \nabla \psi(\sigma(t)), v(t,\sigma(t)) - v_{n_k}(t,\sigma(t)) \big\rangle}{t}{0}{T} \, \bigg| \geq \frac{\epsilon}{2},
\end{equation}
whenever $k \geq 1$ is sufficiently large. However, by the weak-compactness property of \eqref{eq:PeanoVelConvParticular} applied to the mapping $t \in [0,T] \mapsto \Bnu(t) := \zeta(t) \nabla \psi \cdot \delta_{\sigma(t)} \in \Mcal(B(0,R),\R^d)$, it necessarily hold that
\begin{equation}
\label{eq:PeanoSuperposition3}
\INTSeg{\zeta(t) \, \big\langle \nabla \psi(\sigma(t)) , v(t,\sigma(t)) - v_{n_k}(t,\sigma(t)) \big\rangle}{t}{0}{T} ~\underset{k \to +\infty}{\longrightarrow}~ 0
\end{equation}
along a subsequence that we do not relabel, which contradicts \eqref{eq:PeanoSuperposition2bis}. 

To conclude, we observe that by Theorem \ref{thm:Superposition}, there exists a sequence of superposition measures $(\Beta_n) \subset \Pcal(\R^d \times \Sigma_T)$ associated with $(v_n(\cdot)) \subset \Lcal([0,T],C^0(\R^d,\R^d))$ which satisfy
\begin{equation*}
(\pi_{\R^d})_{\sharp} \Beta_n = \mu^0 \qquad \text{and} \qquad (\esf_t)_{\sharp} \Beta_n = \mu_n(t),
\end{equation*}
for all times $t \in [0,T]$ and each $n \geq 1$. Then, by \eqref{eq:PeanoSuperposition0} and Fubini's theorem, one can check that
\begin{equation*}
\begin{aligned}
& \bigg| \INTSeg{\zeta(t) \INTDom{\big\langle \nabla \psi(x) , v(t,x) - v_n(t,x) \big\rangle}{\R^d}{\mu_n(t)(x)}}{t}{0}{T} \; \bigg| \\
& \hspace{1.5cm} = \bigg| \INTSeg{\INTDom{\zeta(t) \, \big\langle \nabla \psi(\sigma(t)) , v(t,\sigma(t)) - v_n(t,\sigma(t)) \big\rangle}{\R^d \times \Sigma_T}{\Beta_n(x,\sigma)}}{t}{0}{T} \; \bigg| \\
& \hspace{1.5cm} \leq \sup_{(x,\sigma) \in \supp(\Beta_n)} \bigg| \INTSeg{\zeta(t) \, \big\langle \nabla \psi(\sigma(t)) , v(t,\sigma(t)) - v_n(t,\sigma(t)) \big\rangle}{t}{0}{T} \, \bigg|\\
& \hspace{1.5cm} \leq \sup_{\sigma \in \AC([0,T],\R^d)} \Bigg\{ \bigg| \INTSeg{ \zeta(t) \, \big \langle \nabla \psi(\sigma(t)), v(t,\sigma(t)) - v_n(t,\sigma(t)) \big\rangle}{t}{0}{T} \, \bigg| ~\textnormal{s.t.}\\
& \hspace{6cm} \dot \sigma(t) = v_n(t,\sigma(t)) ~ \text{for $\Lcal^1$-almost every $t \in [0,T]$} \Bigg\} \,\underset{n \to +\infty}{\longrightarrow}\, 0
\end{aligned}
\end{equation*}
as a consequence of \eqref{eq:PeanoSuperposition0}, which then yields \eqref{eq:PeanoDistribGoal}. Thence, we have established the trajectory-selection pair $(\mu(\cdot),v(\cdot)) \in \AC([0,T],\Pcal_p(\R^d)) \times \Lcal([0,T],C^0(\R^d,\R^d))$ is a solution of \eqref{eq:PeanoIntermCauchy}.


\paragraph*{Step 5 -- Admissibility of the limit velocity selection.}

In this last step, we conclude the proof of Theorem \ref{thm:Peano} by showing that the limit Carathéodory vector field $v(\cdot) \in \Lcal([0,T],C^0(\R^d,\R^d))$ is an admissible velocity selection for \eqref{eq:PeanoCauchy}, namely
\begin{equation}
\label{eq:PeanoVelInterm}
v(t) \in V(t,\mu(t)) 
\end{equation}
for $\Lcal^1$-almost every $t \in [0,T]$. To this end, we recall that the elements of the sequence $(\mu_n(\cdot)) \subset \AC([0,T],\Pcal_p(\R^d))$ are uniformly equi-continuous by \eqref{eq:PeanoBounds}, which implies that for every $\delta > 0$, there exists an integer $N_{\delta} \geq 1$ such that
\begin{equation*}
W_p \Big(\mu_n\big( t-\tfrac{T}{n} \big), \mu_n(t) \Big) \leq \delta, 
\end{equation*}
for all times $t \in [0,T]$, whenever $n \geq N_{\delta}$. Consider now the curves $\tilde{\mu}_n(\cdot) \in \AC([0,T],\Pcal_p(\R^d))$ defined by $\tilde{\mu}_n(t) := \mu_n(t-\tfrac{T}{n})$ for all times $t \in [0,T]$, and notice that 
\begin{equation}
\label{eq:PeanoUnifConvBis}
\sup_{t \in [0,T]} W_p (\tilde{\mu}_n(t) , \mu(t)) ~\underset{n \to +\infty}{\longrightarrow}~ 0
\end{equation}
owing to what precedes, and also
\begin{equation}
\label{eq:PeanoTildInc1}
v_n(t) \in V(t,\tilde{\mu}_n(t)), 
\end{equation}
for $\Lcal^1$-almost every $t \in [0,T]$. At this point, we recall that the set-valued map $t \in [0,T] \tto V(t,\mu(t)) \subset C^0(\R^d,\R^d)$ has nonempty and compact images by Hypotheses \ref{hyp:CIPC}-$(ii)$ and $(iii)$, and that it is $\Lcal^1$-measurable by Lemma \ref{lem:MeasurableSel}-$(c)$. Moreover, observe that as a consequence of Hypothesis \ref{hyp:CIPC}-$(iv)$ together with \eqref{eq:PeanoCompactInc}, there exists a map $l(\cdot) \in L^1([0,T],\R_+)$ and a continuous modulus $\omega : \R_+ \to \R_+$, both possibly depending on $\Kcal \subset \Pcal_p(\R^d)$, such that 
\begin{equation*}
V(t,\mu(t)) \cap \bigg\{ w \in C^0(\R^d,\R^d) ~\, \textnormal{s.t.}~ \dcc(v_n(t),w) \leq l(t) \, \omega \Big( W_p(\tilde{\mu}_n(t),\mu(t)) \Big) \bigg\}
\end{equation*}
is nonempty for $\Lcal^1$-almost every $t \in [0,T]$. Hence, it follows from Lemma \ref{lem:MeasurableSel}-$(a)$ that for each $n \geq 1$, there exists a measurable selection 
\begin{equation}
\label{eq:PeanoTildInc2}
t \in [0,T] \mapsto \tilde{v}_n(t) \in V(t,\mu(t))
\end{equation}
which is such that
\begin{equation*}
\dcc(\tilde{v}_n(t),v_n(t)) \leq l(t) \, \omega \Big(W_p \big(\tilde{\mu}_n(t),\mu(t)) \Big)
\end{equation*}
for $\Lcal^1$-almost every $t \in [0,T]$. By Lebesgue's dominated convergence theorem combined with \eqref{eq:PeanoUnifConv} and  \eqref{eq:PeanoVelConv}, this further implies that 
\begin{equation}
\label{eq:PeanoVelConvBis}
\qquad \quad \tilde{v}_n(\cdot) ~\underset{n \to +\infty}{\rightharpoonup}~ v(\cdot) 
\end{equation}
weakly in $L^1([0,T],C^0(B(0,R),\R^d))$, along possibly different subsequences depending on $R>0$. To conclude, we finally consider the function sets defined by 
\begin{equation*}
\Vcal_R := \Bigg\{ w(\cdot) \in L^1([0,T],C^0(B(0,R),\R^d)) \, ~\textnormal{s.t.}~ w(t) \in V(t,\mu(t))_{| B(0,R)}~ \text{for $\Lcal^1$-almost every $t \in [0,T]$} \Bigg\}
\end{equation*}
for each $R >0$, and note that the latter are closed for the strong topology of $L^1([0,T],C^0(B(0,R),\R^d))$. As these are convex by Hypothesis \ref{hyp:CIPC}-$(i)$, it follows e.g. from \cite[Theorem 3.7]{Brezis} that they are also weakly closed in $L^1([0,T],C^0(B(0,R),\R^d))$. Thus, by \eqref{eq:PeanoTildInc2} and \eqref{eq:PeanoVelConvBis}, we may infer that the maps
\begin{equation*}
v(\cdot)_{|B(0,R)} : t \in [0,T] \mapsto v(t)_{|B(0,R)} \in C^0(B(0,R),\R^d)
\end{equation*}
belong to $\Vcal_R$ for each $R > 0$, which equivalently means that $v(t) \in V(t,\mu(t))$ for $\Lcal^1$-almost every $t \in [0,T]$ and ends the proof of Theorem \ref{thm:Peano}.
\end{proof}

\begin{rmk}[On the choice of proving Theorem \ref{thm:Peano} by means of an Euler scheme]
To prove the existence of solutions to a continuity inclusion with a Carathéodory right-hand side, a tempting strategy, implemented e.g. in \cite[Section 2.1 -- Theorem 3]{Aubin1984} or \cite[Theorem 2.9]{Frankowska1995}, could be to consider first a sufficiently regular exact or approximate velocity selection $(t,\mu) \in [0,T] \times \Pcal_p(\R^d) \mapsto v(t,\mu) \in C^0(\R^d,\R^d)$, and then to show that a continuity equation driven by this latter admits at least a solution.

While such a program may work in practice, carrying it out seemed difficult -- and perhaps suboptimal -- for the following reasons. Firstly, the results ensuring the existence of Carathéodory selections for Carathéodory set-valued maps such as \cite[Theorem 9.5.2]{Aubin1990} do not exist for multifunctions valued in infinite-dimensional spaces. Even if an adequate counterpart were to be found in our context, one would still then need to prove that the corresponding nonlocal continuity equations admit solutions, most likely by means of an Euler scheme. Secondly, even though results providing families of regular approximate selections for upper-semicontinuous set-valued mappings such as \cite[Theorem 9.2.1]{Aubin1990} may still hold for Fr\'echet instead of Banach spaces, establishing the compactness of the underlying sequences may prove to be challenging, owing to the lack of uniformity in the regularity of said selections.
\end{rmk}


\section{A priori estimates, compactness and relaxation in the Cauchy-Lipschitz framework}
\label{section:Filippov}
\setcounter{equation}{0} \renewcommand{\theequation}{\thesection.\arabic{equation}}

In this section, we derive quantitative well-posedness results and discuss some of the fine topological properties of the reachable sets of the set-valued Cauchy problem 
\begin{equation}
\label{eq:ContIncCauchy}
\left\{
\begin{aligned}
& \partial_t \mu(t) \in - \Div_x \Big( V(t,\mu(t)) \mu(t) \Big), \\ 
& \mu(0) = \mu^0. 
\end{aligned}
\right.
\end{equation}
To this end, we depart from the Carathéodory framework investigated in Section \ref{section:Peano}, and work under the following stronger Cauchy-Lipschitz assumptions. 

\begin{taggedhyp}{\textbn{(CI)}} 
\label{hyp:CICL} \hfill
\begin{enumerate}
\item[$(i)$] The set-valued map $(t,\mu) \in [0,T] \times \Pcal_p(\R^d) \tto V(t,\mu) \subset C^0(\R^d,\R^d)$ is Carathéodory with nonempty closed images.
\item[$(ii)$] There exists a map $m(\cdot) \in L^1([0,T],\R_+)$ such that for $\Lcal^1$-almost every $t \in [0,T]$, any $\mu \in \Pcal_p(\R^d)$, every $v \in V(t,\mu)$ and all $x \in \R^d$, there holds 
\begin{equation*}
|v(x)| \leq m(t) \Big(1 + |x| + \Mpazo_p(\mu) \Big)
\end{equation*}
for all $x \in \R^d$. 
\item[$(iii)$] There exists a map $l(\cdot) \in L^1([0,T],\R_+)$ such that for $\Lcal^1$-almost every $t \in [0,T]$, any $\mu \in \Pcal_p(\R^d)$ and every $v \in V(t,\mu)$, there holds 
\begin{equation*}
\Lip(v \, ; \R^d) \leq l(t). 
\end{equation*} 
\item[$(iv)$] There exists a map $L(\cdot) \in L^1([0,T],\R_+)$ such that for $\Lcal^1$-almost every $t \in [0,T]$, any $\mu,\nu \in \Pcal_p(\R^d)$ and each $v \in V(t,\mu)$, there exists an element $w \in V(t,\nu)$ for which
\begin{equation*}
\dsf_{\sup}(v,w) \leq L(t) W_p(\mu,\nu).
\end{equation*}
\end{enumerate}
\end{taggedhyp}

\begin{rmk}[Comparison with the main assumptions of \cite{ContInc}]
The study of curves of compactly supported measures generated by locally Lipschitz set-valued maps initiated in \cite{ContInc} is completely included in our subsequent developments, up to some minor technical adjustments. Likewise, we could have opted in the present manuscript for assumptions in which the Lipschitz constants of the dynamics are localised on compact subsets of $\Pcal_p(\R^d)$.
\end{rmk}


\subsection{Existence of solutions and Filippov estimates}
\label{subsection:Filippov}

In this section, we prove a new variant of the Filippov estimates for solutions of the Cauchy problem \eqref{eq:ContIncCauchy}. Given an arbitrary measure curve, these latter involve the sum of a local discrepancy term between the velocity of said curve and that of the solution of the inclusion on an arbitrary ball, to which one adds a remainder controlled by the tail of the initial datum of said auxiliary curve.

\begin{thm}[Filippov estimates]
\label{thm:Filippov}
Let $V : [0,T] \times \Pcal_p(\R^d) \tto C^0(\R^d,\R^d)$ be a set-valued map satisfying Hypotheses \ref{hyp:CICL} and $\nu(\cdot) \in \AC([0,T],\Pcal_p(\R^d))$ be a solution of the continuity equation
\begin{equation*}
\partial_t \nu(t) + \Div_y (w(t)\nu(t)) = 0
\end{equation*}
driven by a Lebesgue-Borel velocity field $w : [0,T] \times \R^d \to \R^d$ satisfying the sublinearity estimate
\begin{equation}
\label{eq:Filippov_wsublin}
|w(t,y)| \leq m(t)(1+|y|), 
\end{equation}
for $\Lcal^1$-almost every $t \in [0,T]$ and $\nu(t)$-almost every $y \in \R^d$. For every $R > 0$, define the \textnormal{localised mismatch function} $\eta_R(\cdot) \in L^1([0,T],\R_+)$ by 
\begin{equation}
\label{eq:mismatch}
\eta_R(t)  := \dist_{L^{\infty}( B(0,R),\R^d ; \, \nu(t))} \Big( w(t) \, ; V(t,\nu(t)) \Big)
\end{equation}
for $\Lcal^1$-almost every $t \in [0,T]$.  

Then for every $\mu^0 \in \Pcal_p(\R^d)$ and each $R > 0$, there exists a trajectory-selection pair $(\mu(\cdot),v(\cdot)) \in \AC([0,T],\Pcal_p(\R^d)) \times \Lcal([0,T],C^0(\R^d,\R^d))$ solution of the Cauchy problem \eqref{eq:ContIncCauchy} which satisfies the distance estimate
\begin{equation}
\label{eq:FilippovEstimate1}
\begin{aligned}
W_p(\mu(t),\nu(t)) \leq \Dpazo_p(t,R)
\end{aligned}
\end{equation}
for all times $t \in [0,T]$, where
\begin{equation}
\label{eq:DpDef}
\Dpazo_p(t,R) := C_p \bigg( W_p(\mu^0,\nu(0)) + \INTSeg{\eta_{R}(s)}{s}{0}{t} +\Epazo_{\nu}(t,R) \bigg) \exp \Big( C_p' \NormL{l(\cdot)}{1}{[0,t]}^p + \, \chi_p(t) \Big). 
\end{equation}
Therein, the constants $C_p,C_p' > 0$ are as in \eqref{eq:ConstantDef}, the error term $\Epazo_{\nu}(t,R)$ is given explicitly by
\begin{equation*}
\Epazo_{\nu}(t,R) := 2 \NormL{m(\cdot)}{1}{[0,t]} (1+\Cpazo_T) \bigg( \INTDom{(1+|y|)^p}{\{y \; \textnormal{s.t.} \; |y| \geq R/\Cpazo_T-1 \}}{\nu(0)(y)} \bigg)^{1/p} 
\end{equation*}
for some constant $\Cpazo_T > 0$ that only depends on the magnitudes of $p,\Mpazo_p(\mu^0)$ and $\Norm{m(\cdot)}_1$, and the map $\chi_p(\cdot) \in L^{\infty}([0,T],\R_+)$ writes for all $t \in [0,T]$ as
\begin{equation}
\label{eq:ChiDef}
\chi_p(t) := C_p \NormL{L(\cdot)}{1}{[0,t]} \exp \Big( C_p' \NormL{l(\cdot)}{1}{[0,t]}^p \hspace{-0.05cm} \Big).
\end{equation}
Moreover, the velocity selection $t \in [0,T] \mapsto v(t) \in V(t,\mu(t))$ complies with the pointwise estimate 
\begin{equation}
\label{eq:FilippovEstimate2}
\NormL{v(t) - w(t)}{\infty}{B(0,R),\R^d;\, \nu(t)} \, \leq \, \eta_R(t) + L(t) \Dpazo_p(t,R)
\end{equation}
for $\Lcal^1$-almost every $t \in [0,T]$. 
\end{thm}

Before moving on to the proof of Theorem \ref{thm:Filippov}, we state a technical lemma dealing with chained integral estimates, whose proof is the matter of a straightforward induction argument. 

\begin{lem}[A uniform bound on families of functions satisfying recurrent integral estimates]
\label{lem:InductiveIneq}
Let $m(\cdot) \in L^1([0,T],\R_+)$ and $f^0,\alpha > 0$ be two given constants. Then, every at most countable family of functions $(f_n(\cdot)) \subset C^0([0,T],\R_+)$ satisfying $\NormC{f_0(\cdot)}{0}{[0,T],\R_+} \leq f^0$ as well as the recurrence estimate 
\begin{equation*}
f_{n+1}(t) \leq \alpha \bigg( 1 + \INTSeg{m(s) f_n(s)}{s}{0}{t} \bigg)
\end{equation*}
for all times $t \in [0,T]$ and each $n \geq 0$ is uniformly bounded, with 
\begin{equation*}
\sup_{n \geq 0} \NormC{f_n(\cdot)}{0}{[0,T],\R_+} \leq (\alpha+f^0) \exp \big( \alpha \hspace{-0.05cm} \Norm{m(\cdot)}_1 \hspace{-0.05cm} \big). 
\end{equation*}
\end{lem}

The proof of Theorem \ref{thm:Filippov} is based on a constructive scheme, much like that of Theorem \ref{thm:Peano}, and is split into four steps. In Step 1, we detail the initialisation of the underlying induction argument by means of a selection principle applied along $\nu(\cdot)$, and proceed in Step 2 by building the whole sequence of trajectory-selection pairs. We then show in Step 3 that the latter is a Cauchy sequence for a suitable extended metric, and finally prove in Step 4 that the corresponding limit pair is a solution of \eqref{eq:ContIncCauchy} which satisfies the stability estimates \eqref{eq:FilippovEstimate1} and \eqref{eq:FilippovEstimate2}. 

\begin{proof}[Proof of Theorem \ref{thm:Filippov}]
Our goal in what follows is to build a sequence of trajectory-selection pairs $(\mu_n(\cdot),v_n(\cdot)) \subset \AC([0,T],\Pcal_p(\R^d)) \times \Lcal([0,T],C^0(\R^d,\R^d))$ whose elements solve the Cauchy problems
\begin{equation}
\label{eq:InductionCauchy}
\left\{
\begin{aligned}
& \partial_t \mu_n(t) + \Div_x (v_n(t) \mu_n(t)) = 0, \\
& \mu_n(0) = \mu^0, 
\end{aligned}
\right.
\end{equation}
while also satisfying the conditions
\begin{equation}
\label{eq:InductionCond1}
v_{n+1}(t) \in V(t,\mu_n(t)) \qquad \text{and} \qquad \dsf_{\sup}(v_n(t),v_{n+1}(t)) \leq L(t) W_p(\mu_{n-1}(t),\mu_n(t)) 
\end{equation}
for $\Lcal^1$-almost every $t \in [0,T]$, and complying with the uniform moment and regularity bounds 
\begin{equation}
\label{eq:InductionCond2}
\Mpazo_p(\mu_n(t)) \leq \Cpazo \qquad \text{and} \qquad W_p(\mu_n(\tau),\mu_n(t)) \leq c_p \INTSeg{m(s)}{s}{\tau}{t}
\end{equation}
for all times $0 \leq \tau \leq t \leq T$ and each $n \geq 1$. Therein, the constants $\Cpazo,c_p > 0$ only depend on the magnitudes of $p,\Mpazo_p(\mu^0)$ and $\Norm{m(\cdot)}_1$.  


\paragraph*{Step 1 -- Initialisation of the sequence.} Observe first that by combining Hypotheses \ref{hyp:CICL}-$(ii)$ and $(iii)$ along with Theorem \ref{thm:Ascoli}, the admissible velocity sets $V(t,\nu(t)) \subset C^0(\R^d,\R^d)$ are compact in the topology of uniform convergence on compact sets for $\Lcal^1$-almost every $t \in [0,T]$. Besides under Hypotheses \ref{hyp:CICL}-$(i)$ and $(iv)$, the set-valued map $t \in [0,T] \tto V(t,\nu(t)) \subset C^0(\R,\R^d)$ is $\Lcal^1$-measurable as a consequence of Lemma \ref{lem:MeasurableSel}-$(c)$. Furthermore, it can be checked that 
\begin{equation*}
(t,v) \in [0,T] \times C^0(\R^d,\R^d) \mapsto \NormL{v - w(t)}{\infty}{B(0,R),\R^d \; ; \nu(t)} \in \R_+
\end{equation*}
is a Carathéodory map for each $R>0$, in the sense that it is $\Lcal^1$-measurable with respect to $t \in [0,T]$ and continuous with respect to $v \in C^0(\R^d,\R^d)$ for the topology induced by $\dcc(\cdot,\cdot)$. Thus, recalling the definition \eqref{eq:mismatch} of the mismatch function $\eta_R(\cdot)$ and applying Lemma \ref{lem:MeasurableSel}-$(b)$ to the multifunction
\begin{equation*}
t \in [0,T] \tto V(t,\nu(t)) \cap \bigg\{ v \in C^0(\R^d,\R^d) \, ~\textnormal{s.t.}~  \NormL{v - w(t)}{\infty}{B(0,R),\R^d, \; \nu(t)} \, = \, \eta_R(t) \bigg\}, 
\end{equation*}
we obtain the existence of an $\Lcal^1$-measurable map $t \in [0,T] \mapsto v_1(t) \in V(t,\nu(t))$ such that 
\begin{equation}
\label{eq:FirstVelIneq}
\NormL{v_1(t) - w(t)}{\infty}{B(0,R),\R^d ; \, \nu(t)}  \, = \, \eta_R(t)  
\end{equation}
for $\Lcal^1$-almost every $t \in [0,T]$. 

Remark now that as a by-product of Hypotheses \ref{hyp:CICL}-$(ii)$ and $(iii)$ along with Lemma \ref{lem:Carathéodory}, the velocity field $v_1 : [0,T] \times \R^d \to \R^d$ is Carathéodory and such that
\begin{equation*} 
|v_1(t,x)| \leq m(t) \Big( 1 + |x| + \Mpazo_p(\nu(t)) \Big) \qquad \text{and} \qquad \Lip(v_1(t) \, ; \R^d) \leq l(t)
\end{equation*}
for $\Lcal^1$-almost every $t \in [0,T]$ and all $x \in \R^d$. In particular, it satisfies Hypotheses \ref{hyp:CE}, and by Theorem \ref{thm:Wellposedness} there exists a unique curve $\mu_1(\cdot) \in \AC([0,T],\Pcal_p(\R^d))$ solution of the Cauchy problem
\begin{equation*}
\left\{ 
\begin{aligned}
& \partial_t \mu_1(t) + \Div_x (v_1(t) \mu_1(t)) = 0, \\
& \mu_1(0) = \mu^0. 
\end{aligned}
\right.
\end{equation*}
Owing to the moment bound of Remark \ref{rmk:Moment}, the curve $\mu_1(\cdot)$ further complies with the estimate
\begin{equation}
\label{eq:EstimatesMu11}
\Mpazo_p(\mu_1(t)) \leq C_p \bigg( \Mpazo_p(\mu^0) + \INTSeg{m(s) \Big(1 + \Mpazo_p(\nu(s)) \Big)}{s}{0}{t} \bigg) \exp \Big( C_p' \NormL{m(\cdot)}{1}{[0,t]}^p \Big), \\
\end{equation}
for all times $t \in [0,T]$. Consequently, by Lemma \ref{lem:InductiveIneq} and \eqref{eq:MomentEst} of Proposition \ref{prop:Moment} applied to the curve $\nu(\cdot) \in \AC([0,T],\Pcal_p(\R^d))$, there exists a constant $\Cpazo > 0$ that only depends on the magnitudes of $p,\Mpazo_p(\mu^0),\Mpazo_p(\nu(0))$ and $\Norm{m(\cdot)}_1$, such that 
\begin{equation}
\label{eq:MomentBound1}
\max \Big\{ \Mpazo_p(\nu(t)) \, , \, \Mpazo_p(\mu_1(t)) \Big\} \leq \Cpazo  
\end{equation}
for all times $t \in [0,T]$. In particular, the velocity field $v_1 : [0,T] \times \R^d \to \R^d$ satisfies 
\begin{equation}
\label{eq:SublinStep1}
|v_1(t,x)| \leq (1+\Cpazo)m(t) \big(1 + |x| \big), 
\end{equation}
for $\Lcal^1$-almost every $t \in [0,T]$ and all $x \in \R^d$, which combined with \eqref{eq:ACEstimate} of Theorem \ref{thm:Wellposedness}  yields the existence of a constant $c_p > 0$ depending only on the magnitudes of $p,\Mpazo_p(\mu^0),\Mpazo_p(\nu(0))$ and $\Norm{m(\cdot)}_1$, such that 
\begin{equation*}
W_p(\mu_1(\tau),\mu_1(t)) \leq c_p \INTSeg{m(s)}{s}{\tau}{t}
\end{equation*}
for all times $0 \leq \tau \leq t \leq T$. Moreover, by applying the approximate stability inequality \eqref{eq:Gronwall2} of Proposition \ref{prop:Gronwall} while taking into account the sublinearity estimate \eqref{eq:SublinStep1}, it further holds that 
\begin{equation}
\label{eq:EstimatesMu12}
 W_p(\mu_1(t),\nu(t)) \leq C_p \bigg( W_p(\mu^0,\nu(0)) + \INTSeg{\eta_R(s)}{s}{0}{t} + \Epazo_{\nu}(t,R) \bigg) \exp \Big( C_p' \NormL{l(\cdot)}{1}{[0,t]}^p \hspace{-0.05cm} \Big)
\end{equation}
for all times $t \in [0,T]$, where
\begin{equation*}
\Epazo_{\nu}(t,R) := 2 \NormL{m(\cdot)}{1}{[0,t]} (1+\Cpazo_T) \bigg( \INTDom{(1+|y|)^p}{\{ \; \textnormal{s.t.} \; |y| \geq R/\Cpazo_T-1 \}}{\nu(0)(y)} \bigg)^{1/p}
\end{equation*}
with $\Cpazo_T := \max\big\{ 1 , (1+\Cpazo) \hspace{-0.1cm} \Norm{m(\cdot)}_1 \hspace{-0.1cm} \big\} \exp\big( (1+\Cpazo) \hspace{-0.1cm} \Norm{m(\cdot)}_1 \hspace{-0.1cm} \big)$. This together with the moment estimate of \eqref{eq:EstimatesMu11} concludes the initialisation step of our induction argument. 


\paragraph*{Step 2 -- Construction of the sequence of trajectory-selection pairs.}

At this stage, note that
\begin{equation*}
\dsf_{\sup}(v_1(t),w) = \sup_{m \geq 1} \NormC{v_1(t) - w}{0}{B(0,m),\R^d} 
\end{equation*}
for $\Lcal^1$-almost every $t \in [0,T]$ and all $w \in C^0(\R^d,\R^d)$, and that by construction, the maps $(t,w) \in [0,T] \times C^0(\R^d,\R^d) \mapsto \NormC{v_1(t)-w}{0}{B(0,m),\R^d} \in \R_+$ are Carathéodory for each $m \geq 1$. Thence, as a consequence of Lemma \ref{lem:MeasurableSel}-$(a)$ combined with Hypothesis \ref{hyp:CICL}-$(iv)$, there exists an $\Lcal^1$-measurable map $t \in [0,T] \mapsto v_2(t) \in V(t,\mu_1(t))$ satisfying
\begin{equation*}
\dsf_{\sup}(v_1(t),v_2(t)) \leq L(t) W_p(\mu_1(t),\nu(t)) 
\end{equation*}
for $\Lcal^1$-almost every $t \in [0,T]$. From Hypotheses \ref{hyp:CICL}-$(ii)$ and $(iii)$ combined with Lemma \ref{lem:Carathéodory}, we may then infer that the velocity field $v_2 : [0,T] \times \R^d \rightarrow \R^d$ is Carathéodory, and such that
\begin{equation*}
|v_2(t,x)| \leq  m(t) \Big( 1 + |x| + \Mpazo_p(\mu_1(t)) \Big) \qquad \text{and} \qquad \Lip(v_2(t) \, ; \R^d) \leq l(t), 
\end{equation*}
for $\Lcal^1$-almost every $t \in [0,T]$ and all $x \in \R^d$. In particular, it satisfies Hypotheses \ref{hyp:CE}, and thus generates a unique solution $\mu_2(\cdot) \in \AC([0,T],\Pcal_p(\R^d))$ of the corresponding Cauchy problem
\begin{equation*}
\left\{
\begin{aligned}
& \partial_t \mu_2(t) + \Div_x (v_2(t)\mu_2(t)) = 0, \\
& \mu_2(0) = \mu^0. 
\end{aligned}
\right.
\end{equation*}
By the refined moment bound of Remark \ref{rmk:Moment} and the estimate \eqref{eq:Gronwall1} of Proposition \ref{prop:Gronwall}, it may again be checked that 
\begin{equation*}
\left\{
\begin{aligned}
& \Mpazo_p(\mu_2(t)) \leq C_p \bigg( \Mpazo_p(\mu^0) + \INTSeg{m(s)\Big( 1 + \Mpazo_p(\mu_1(s)) \Big)}{s}{0}{t} \bigg) \exp \Big( C_p' \NormL{m(\cdot)}{1}{[0,t]}^p \hspace{-0.05cm} \Big), \\
& W_p(\mu_1(t),\mu_2(t)) \leq C_p \bigg( \INTSeg{\dsf_{\sup}(v_1(s),v_2(s))}{s}{0}{t} \bigg) \exp \Big( C_p' \NormL{l(\cdot)}{1}{[0,t]}^p \hspace{-0.05cm} \Big). 
\end{aligned}
\right.
\end{equation*}
As a consequence of Lemma \ref{lem:InductiveIneq}, this implies in particular that 
\begin{equation*}
\max \Big\{ \Mpazo_p(\nu(t)) \, , \, \Mpazo_p(\mu_1(t)) \, , \, \Mpazo_p(\mu_2(t)) \Big\} \leq \Cpazo 
\end{equation*}
for all times $t \in [0,T]$, where $\Cpazo > 0$ is the same constant as in Step 1 which only depends on the magnitudes of $p,\Mpazo_p(\mu^0),\Mpazo_p(\nu(0))$ and $\Norm{m(\cdot)}_1$. 

By repeating this process, we can inductively build a sequence $(\mu_n(\cdot),v_n(\cdot)) \subset \AC([0,T],\Pcal_p(\R^d)) \times \Lcal([0,T],C^0(\R^d,\R^d))$ of trajectory-selection pairs solution of the Cauchy problems \eqref{eq:InductionCauchy}, whose elements satisfy the conditions of \eqref{eq:InductionCond1} as well as the bounds in \eqref{eq:InductionCond2}. In addition, these latter comply for each $n \geq 1$ with the uniform sublinearity estimates
\begin{equation}
\label{eq:InductionEst1}
|v_n(t,x)| \leq (1+\Cpazo) m(t) \big( 1 + |x| \big),
\end{equation}
for $\Lcal^1$-almost every \text{$t \in [0,T]$} and all $x \in \R^d$, as well as the distance estimates
\begin{equation}
\label{eq:InductionEst2}
W_p(\mu_n(t),\mu_{n+1}(t)) \leq C_p \bigg( \INTSeg{\dsf_{\sup}(v_n(s),v_{n+1}(s))}{s}{0}{t}\bigg) \exp \Big( C_p' \NormL{l(\cdot)}{1}{[0,t]}^p \Big), 
\end{equation}
for all times $t \in [0,T]$. 
%


\paragraph*{Step 3 -- Convergence of the sequence of trajectory-selection pairs.} In what follows, we prove that the sequence of pairs $(\mu_n(\cdot),v_n(\cdot)) \subset \AC([0,T],\Pcal_p(\R^d)) \times \Lcal([0,T],C^0(\R^d,\R^d))$ built in Step 1 and Step 2 is Cauchy in a suitable sense. First, observe that by merging the distance estimates of \eqref{eq:InductionCond1}, \eqref{eq:EstimatesMu12} and \eqref{eq:InductionEst2}, one has that
\begin{equation}
\label{eq:CauchyEst1}
\begin{aligned}
& W_p(\mu_n(t),\mu_{n+1}(t)) \\ 
& \leq C_p \bigg( \INTSeg{\dsf_{\sup} ( v_n(s_n),v_{n+1}(s_n))}{s_n}{0}{t} \bigg) \exp \Big( C_p' \NormL{l(\cdot)}{1}{[0,t]}^p \hspace{-0.05cm} \Big) \\
& \leq C_p \bigg( \INTSeg{L(s_n) W_p(\mu_{n-1}(s_n),\mu_n(s_n))}{s_n}{0}{t} \bigg) \exp \Big( C_p' \NormL{l(\cdot)}{1}{[0,t]}^p \hspace{-0.05cm} \Big) \\
& \leq C_p^2 \bigg( \INTSeg{L(s_n) \INTSeg{\dsf_{\sup} (v_{n-1}(s_{n-1}),v_n(s_{n-1}))}{s_{n-1}}{0}{s_n} }{s_n}{0}{t} \bigg) \exp \Big( 2 C_p' \NormL{l(\cdot)}{1}{[0,t]}^p \hspace{-0.05cm} \Big) \\
& \vdots \\
& \leq C_p^n \bigg( \INTSeg{L(s_n) \INTSeg{L(s_{n-1}) \dots \INTSeg{ \hspace{-0.05cm} L(s_1)W_p(\mu_1(s_1),\nu(s_1))}{s_1}{0}{s_2} \dots }{s_{n-1}}{0}{s_n}}{s_n}{0}{t} \bigg) \exp \Big( n \hspace{0.05cm} C_p' \NormL{l(\cdot)}{1}{[0,t]}^p \hspace{-0.1cm} \Big) \\
& \leq \frac{C_p^{n+1}}{n!} \bigg( \INTSeg{L(s)}{s}{0}{t} \bigg)^n \bigg( W_p(\mu^0,\nu(0)) + \INTSeg{\eta_R(s)}{s}{0}{t} + \Epazo_{\nu}(t,R) \bigg) \exp \Big( (n+1) \hspace{0.05cm} C_p' \NormL{l(\cdot)}{1}{[0,t]}^p \hspace{-0.05cm} \Big), \\
& = \frac{\chi_p(t)^n}{n!} \, C_p  \bigg( W_p(\mu^0,\nu(0)) + \INTSeg{\eta_R(s)}{s}{0}{t} + \Epazo_{\nu}(t,R) \bigg) \exp \Big( C_p' \NormL{l(\cdot)}{1}{[0,t]}^p \hspace{-0.05cm} \Big)
\end{aligned}
\end{equation}
for all times $t \in [0,T]$ and each $n \geq 1$, where we recall that $\chi_p(\cdot) \in L^{\infty}([0,T],\R_+)$ is defined as in \eqref{eq:ChiDef}. Whence, for any pair of integers $m,n \geq 1$, one may deduce from \eqref{eq:CauchyEst1} that 
\begin{equation}
\label{eq:CauchyEst1Bis}
\begin{aligned}
& \sup_{t \in [0,T]} W_p(\mu_n(t),\mu_{n+m}(t)) \\
& \leq \sup_{t \in [0,T]} \sum_{k=n}^{n+m-1} W_p(\mu_k(t),\mu_{k+1}(t)) \\
& \leq \Bigg( \sum_{k=n}^{n+m-1} \frac{\Norm{\chi_p(\cdot)}_{\infty}^k}{k!} \Bigg) C_p  \bigg( W_p(\mu^0,\nu(0)) + \INTSeg{\eta_{R}(s)}{s}{0}{T} + \Epazo_{\nu}(T,R) \bigg) \exp \Big( C_p' \NormL{l(\cdot)}{1}{[0,t]}^p \Big) \underset{m,n \to +\infty}{\longrightarrow} 0,
\end{aligned}
\end{equation}
which means that $(\mu_n(\cdot)) \subset \AC([0,T],\Pcal_p(\R^d))$ is a Cauchy sequence in $C^0([0,T],\Pcal_p(\R^d))$. Noting that the latter is a complete metric space as a consequence e.g. of \cite[Chapter 7 -- Theorem 12]{Kelley1975}, there exists some $\mu(\cdot) \in C^0([0,T],\Pcal_p(\R^d))$ such that 
\begin{equation}
\label{eq:ConvergenceMeas}
\sup_{t \in [0,T]} W_p(\mu_n(t),\mu(t)) ~\underset{n \to +\infty}{\longrightarrow}~ 0.  
\end{equation}
Besides, as a consequence of the equi-absolute continuity estimate established in \eqref{eq:InductionCond2} for the sequence $(\mu_n(\cdot)) \subset \AC([0,T],\Pcal_p(\R^d))$, the limit curve is absolutely continuous as well, with 
\begin{equation*}
W_p(\mu(\tau),\mu(t)) \leq c_p \INTSeg{m(s)}{s}{\tau}{t}
\end{equation*}
for all times $0 \leq \tau \leq t \leq T$. 

Concerning the sequence of Carathéodory vector fields $(v_n(\cdot)) \subset \Lcal([0,T],C^0(\R^d,\R^d))$, it directly follows from \eqref{eq:InductionCond1} together with the estimates in \eqref{eq:CauchyEst1Bis} that for each $m,n \geq 1$, there holds
\begin{equation}
\label{eq:CauchyEst2}
\begin{aligned}
\INTSeg{\dsf_{\sup}(v_n(t),v_{n+m}(t))}{t}{0}{T} & \leq \sum_{k=n}^{m+n-1} \INTSeg{\dsf_{\sup}(v_k(t),v_{k+1}(t))}{t}{0}{T} \\
& \leq \, \Norm{L(\cdot)}_1 \sup_{t \in [0,T]} \sum_{k=n}^{m+n-1} W_p(\mu_{k-1}(t),\mu_k(t)) ~\underset{m,n \to +\infty}{\longrightarrow}~ 0. 
\end{aligned}
\end{equation}
Hence, by Lemma \ref{lem:CompleteMetric}, there exists a map $v(\cdot) \in \Lcal([0,T],C^0(\R^d,\R^d))$ such that
\begin{equation}
\label{eq:ConvergenceVel}
\INTSeg{\dsf_{\sup}(v_n(t),v(t))}{t}{0}{T} ~\underset{n \to +\infty}{\longrightarrow}~ 0.
\end{equation}
This yields in particular the existence of a subsequence $(v_{n_k}(\cdot)) \subset \Lcal([0,T],C^0(\R^d,\R^d))$ satisfying
\begin{equation*}
\dsf_{\sup}(v_{n_k}(t),v(t)) ~\underset{n \to +\infty}{\longrightarrow}~ 0 
\end{equation*}
for $\Lcal^1$-almost every $t \in [0,T]$, which combined with \eqref{eq:InductionEst1} directly implies that the limit velocity field $v : [0,T] \times \R^d \mapsto \R^d$ satisfies the sublinearity and regularity estimates 
\begin{equation}
\label{eq:LimitVelEst}
|v(t,x)| \leq (1+\Cpazo) m(t) \big( 1 + |x| \big) \qquad \text{and} \qquad \Lip(v(t) \, ; \R^d) \leq l(t),
\end{equation}
for $\Lcal^1$-almost every $t \in [0,T]$ and all $x \in \R^d$. 


\paragraph*{Step 4 -- Properties of the limit trajectory-selection pair.} 

By combining the estimate \eqref{eq:WassEst} of Proposition \ref{prop:WassEst} with the stability inequality \eqref{eq:Gronwall2} in Proposition \ref{prop:Gronwall} as well as \eqref{eq:LimitVelEst}, it can be straightforwardly deduced from the convergence results of \eqref{eq:ConvergenceMeas} and \eqref{eq:ConvergenceVel} that the trajectory-selection pair $(\mu(\cdot),v(\cdot)) \in \AC([0,T],\Pcal_p(\R^d)) \times \Lcal([0,T],C^0(\R^d,\R^d))$ solves the Cauchy problem
\begin{equation}
\label{eq:LimitCauchy}
\left\{ 
\begin{aligned}
& \partial_t \mu(t) + \Div_x (v(t) \mu(t)) = 0, \\
& \mu(0) = \mu^0. 
\end{aligned}
\right.
\end{equation}
The next step in our argument consists in showing that the latter actually solves the continuity inclusion \eqref{eq:ContIncCauchy}. By \eqref{eq:InductionCond1}, the sequence of trajectory-selection pairs is such that 
\begin{equation}
\label{eq:GraphInc}
v_{n+1}(t) \in V(t,\mu_n(t))
\end{equation}
for $\Lcal^1$-almost every $t \in [0,T]$ and each $n \geq 1$. By \eqref{eq:ConvergenceMeas} and \eqref{eq:ConvergenceVel}, it further holds that
\begin{equation}
\label{eq:ConvergenceMeasPoint}
W_p(\mu_n(t),\mu(t)) ~\underset{n \to +\infty}{\longrightarrow}~ 0, 
\end{equation}
for all times $t \in [0,T]$, as well as
\begin{equation}
\label{eq:ConvergenceVelPoint}
\dsf_{\sup}(v_n(t),v(t)) ~\underset{n \to +\infty}{\longrightarrow}~ 0,
\end{equation}
for $\Lcal^1$-almost every $t \in [0,T]$, along a subsequence that we do not relabel. Whence, upon combining \eqref{eq:GraphInc} with the pointwise convergence results \eqref{eq:ConvergenceMeasPoint} and \eqref{eq:ConvergenceVelPoint} and Hypothesis \ref{hyp:CICL}-$(iv)$, we obtain 
\begin{equation*}
v(t) \in V(t,\mu(t))
\end{equation*}
for $\Lcal^1$-almost every $t \in [0,T]$, which in light of Definition \ref{def:ContInc} implies together with \eqref{eq:LimitCauchy} that $\mu(\cdot) \in \AC([0,T],\Pcal_p(\R^d))$ is a solution \eqref{eq:ContIncCauchy}, driven by the selection $t \in [0,T] \mapsto v(t) \in V(t,\mu(t))$.  

Now, there only remains to derive the distance and velocity estimates displayed in \eqref{eq:FilippovEstimate1} and \eqref{eq:FilippovEstimate2}. To this end, note that the chain of pointwise inequalities of \eqref{eq:CauchyEst1} implies 
\begin{equation*}
\begin{aligned}
W_p(\mu_n(t),\nu(t)) & \leq W_p(\mu_1(t),\nu(t)) + \sum_{k=1}^{n-1}  W_p(\mu_k(t),\mu_{k+1}(t)) \\
& \leq \Bigg( 1 + \sum_{k=1}^{n-1} \frac{\chi_p(t)^k}{k!}\Bigg) C_p \bigg( W_p(\mu^0,\nu(0)) + \INTSeg{\eta_R(s)}{s}{0}{t} + \Epazo_{\nu}(t,R) \bigg) \exp \big( C_p' \Norm{l(\cdot)}_1^p \big) \\
& \leq C_p \bigg( W_p(\mu^0,\nu(0)) + \INTSeg{\eta_R(s)}{s}{0}{t} + \Epazo_{\nu}(t,R)\bigg) \exp \Big( C_p' \Norm{l(\cdot)}_1^p + \chi_p(t) \Big) \\
& = \Dpazo_p(t,R), 
\end{aligned}
\end{equation*}
for all times $t \in [0,T]$ and each $n \geq 1$, with $\Dpazo_p(\cdot)$ being defined as in \eqref{eq:DpDef}. By letting $n \rightarrow +\infty$ in the previous expression, we directly obtain \eqref{eq:FilippovEstimate1}. Concerning the velocity estimates, it holds as a consequence of \eqref{eq:InductionCond1}, \eqref{eq:FirstVelIneq} and \eqref{eq:EstimatesMu12} that
\begin{equation*}
\begin{aligned}
\NormL{v_n(t) - w(t)}{\infty}{B(0,R),\R^d;\, \nu(t)} & \leq \, \NormL{v_1(t) - w(t)}{\infty}{B(0,R),\R^d;\, \nu(t)} + \sum_{k=1}^{n-1} \dsf_{\sup}(v_{k+1}(t),v_k(t)) \\
& \leq \eta_R(t) + L(t) \sum_{k=1}^{n-1} W_p(\mu_{k-1}(t),\mu_k(t)) \\
& \leq \eta_R(t) + L(t) \Dpazo_p(t,R), 
\end{aligned}
\end{equation*}
for $\Lcal^1$-almost every $t \in [0,T]$ and each $n \geq 1$. We obtain \eqref{eq:FilippovEstimate2} by taking the limit as $n \to +\infty$ along a suitable subsequence in the previous expression, which concludes the proof of Theorem \ref{thm:Filippov}.
\end{proof}

In the following corollary, we state a global version of the Filippov estimates which can be obtained by a simple adaptation of the proof of Theorem \ref{thm:Filippov}.

\begin{cor}[Global version of Filippov estimates]
\label{cor:GlobalFilippov}
Suppose that the assumptions of Theorem \ref{thm:Filippov} hold, and in addition that the \textnormal{global mismatch function}, defined by 
\begin{equation*}
\eta(t) := \dist_{L^{\infty}(\R^d,\R^d ;\;  \nu(t))} \Big( w(t) \, ; V(t,\nu(t)) \Big)
\end{equation*}
for $\Lcal^1$-almost every $t \in [0,T]$, is Lebesgue integrable. Then for every $\mu^0 \in \Pcal_p(\R^d)$, there exists a trajectory-selection pair $(\mu(\cdot),v(\cdot)) \in \AC([0,T],\Pcal_p(\R^d)) \times \Lcal([0,T],C^0(\R^d,\R^d))$ solution of the Cauchy problem \eqref{eq:ContIncCauchy} which satisfies the distance estimate 
\begin{equation*}
W_p(\mu(t),\nu(t)) \leq  C_p \bigg( W_p(\mu^0,\nu(0)) + \INTSeg{\eta(s)}{s}{0}{t} \bigg) \exp \Big( C_p' \NormL{l(\cdot)}{1}{[0,t]}^p + \, \chi_p(t) \Big),  
\end{equation*}
for all times $t \in [0,T]$, as well as the velocity estimate 
\begin{equation*}
\begin{aligned}
& \NormL{v(t) - w(t)}{\infty}{\R^d,\R^d;\, \nu(t)} \\
& \hspace{2cm} \leq \eta(t) + L(t) C_p \bigg( W_p(\mu^0,\nu(0))+ \INTSeg{\eta(s)}{s}{0}{t} \bigg) \exp \Big( C_p' \NormL{l(\cdot)}{1}{[0,t]}^p+ \, \chi_p(t) \Big), 
\end{aligned}
\end{equation*}
for $\Lcal^1$-almost every $t \in [0,T]$. 
\end{cor}

\begin{proof}
One can simply repeat the proof strategy of Theorem \ref{thm:Filippov} with $R = +\infty$ and $\Epazo_{\nu}(t,R) = 0$.
\end{proof}


\subsection{Compactness of the solution set and relaxation property}
\label{subsection:Compactness}

In this section, we investigate the topological properties of the solution set associated with the set-valued Cauchy problem \eqref{eq:ContIncCauchy}, which is defined by 
\begin{equation}
\label{eq:SolutionSet}
\Spazo_{[0,T]}(\mu^0) := \bigg\{ \mu(\cdot) \in \AC([0,T],\Pcal_p(\R^d)) ~~\textnormal{s.t.\; $\mu(\cdot)$ is a solution of \eqref{eq:ContIncCauchy} with $\mu(0) = \mu^0$} \bigg\}
\end{equation}
for each $\mu^0 \in \Pcal_p(\R^d)$. Our first result establishes that $\Spazo_{[0,T]}(\mu^0)$ is compact for the topology of uniform convergence when $V : [0,T] \times \Pcal_p(\R^d) \tto C^0(\R^d,\R^d)$ has convex images. Its proof it not written in full details, as it is based on simpler variants of the arguments subtending Theorem \ref{thm:Peano}.

\begin{thm}[Compactness of the solution set]
\label{thm:Compactness}
Let $\mu^0 \in \Pcal_p(\R^d)$ and $V : [0,T] \times \Pcal_p(\R^d) \tto C^0(\R^d,\R^d)$ be a set-valued map satisfying Hypotheses \ref{hyp:CICL} and whose images are convex. Then, the solution set $\Spazo_{[0,T]}(\mu^0) \subset C^0([0,T],\Pcal_p(\R^d))$ associated with the Cauchy problem \eqref{eq:ContIncCauchy} is compact for the topology of uniform convergence. 
\end{thm}

\begin{proof}
Given a sequence $(\mu_n(\cdot),v_n(\cdot)) \subset \AC([0,T],\Pcal_p(\R^d)) \times \Lcal([0,T],C^0(\R^d,\R^d))$ of trajectory-selection pairs for the Cauchy problem \eqref{eq:ContIncCauchy}, one may repeat the arguments of Step 3 in the proof of Theorem \ref{thm:Peano} while using \ref{hyp:CICL}-$(i)$ along with Remark \ref{rmk:Moment} to show the existence a compact set $\Kcal \subset \Pcal_p(\R^d)$ and a constant $c_p > 0$, both depending only on $\mu^0,p$ and $\Norm{m(\cdot)}_1$, such that
\begin{equation*}
\mu_n(t) \in \Kcal \qquad \text{and} \qquad W_p(\mu_n(\tau),\mu_n(t)) \leq c_p \INTSeg{m(s)}{s}{\tau}{t}
\end{equation*}
for all times $0 \leq \tau \leq t \leq T$ and each $n \geq 1$. Hence, by the Ascoli-Arzel\`a theorem, there exists a curve $\mu(\cdot) \in \AC([0,T],\Pcal_p(\R^d))$ such that 
\begin{equation}
\label{eq:AscoliMeasures}
\sup_{t \in [0,T]} W_p(\mu_n(t),\mu(t)) ~\underset{n \to +\infty}{\longrightarrow}~ 0,
\end{equation}
along a subsequence that we do not relabel. By following again the compactness and diagonal argument of Step 3 in the proof of Theorem \ref{thm:Peano}, one can show that there exists a Carathéodory vector field $v(\cdot) \in \Lcal([0,T],C^0(\R^d,\R^d))$ such that
\begin{equation}
\label{eq:WeakConvergenceVel0}
v_n(\cdot) ~\underset{n \to +\infty}{\rightharpoonup}~ v(\cdot)
\end{equation}
weakly in $L^1([0,T],C^0(B(0,R),\R^d))$ for every $R>0$, along possibly different subsequences. In particular, it follows from a direct application of Mazur's lemma that
\begin{equation}
\label{eq:VelEstCompact}
|v(t,x)| \leq m(t) \Big( 1 + |x| + \Mpazo_p(\mu(t)) \Big) \qquad \text{and} \qquad \Lip(v(t) \, ; \R^d) \leq l(t), 
\end{equation}
for $\Lcal^1$-almost every $t \in [0,T]$ and all $x \in \R^d$. We now prove that $(\mu(\cdot),v(\cdot)) \in \AC([0,T],\Pcal_p(\R^d)) \times \Lcal([0,T],C^0(\R^d,\R^d))$ is a solution of  \eqref{eq:ContIncCauchy}. Upon combining the convergence result of \eqref{eq:AscoliMeasures} together with \eqref{eq:VelEstCompact} and Proposition \ref{prop:WassEst}, it is clear that 
\begin{equation}
\label{eq:WeakConvergenceVel1}
\INTSeg{\INTDom{\partial_t \phi(t,x)}{\R^d}{\mu_n(t)(x)}}{t}{0}{T} ~\underset{n \to +\infty}{\longrightarrow}~ \INTSeg{\INTDom{\partial_t \phi(t,x)}{\R^d}{\mu(t)(x)}}{t}{0}{T}
\end{equation}
and
\begin{equation}
\label{eq:WeakConvergenceVel2}
\INTSeg{\INTDom{\big\langle \nabla_x \phi(t,x) , v_n(t,x) \big\rangle}{\R^d}{\big( \mu(t) - \mu_n(t) \big)(x)}}{t}{0}{T} ~\underset{n \to +\infty}{\longrightarrow}~ 0
\end{equation}
for each $\phi \in C^{\infty}_c((0,T) \times \R^d,\R)$. Moreover, by choosing test functions of the form 
\begin{equation}
\label{eq:TestFunction}
\phi(t,x) := \zeta(t) \psi(x), 
\end{equation}
for all $(t,x) \in [0,T] \times \R^d$, with $(\zeta,\psi) \in C^{\infty}_c((0,T),\R) \times \in C^{\infty}_c(\R^d,\R)$, and then setting $\Bnu(t) := \zeta(t) \nabla \psi \cdot \mu(t) \in \Mcal(\supp(\psi),\R^d)$ for all times $t \in [0,T]$, it follows from \eqref{eq:WeakConvergenceVel0} that
\begin{equation}
\label{eq:WeakConvergenceVel3}
\INTSeg{\zeta(t) \INTDom{\big\langle \nabla \psi(x) , v(t,x) - v_n(t,x) \big\rangle}{\R^d}{\mu(t)(x)}}{t}{0}{T} ~\underset{n \to +\infty}{\longrightarrow}~ 0, 
\end{equation}
along a subsequence that depends on $\supp(\psi) \subset \R^d$. Therefore, by merging the convergence results of \eqref{eq:WeakConvergenceVel1}, \eqref{eq:WeakConvergenceVel2} and \eqref{eq:WeakConvergenceVel3} while recalling that the linear span of test functions of the form \eqref{eq:TestFunction} is dense in $C^{\infty}_c((0,T) \times \R^d,\R)$, we finally obtain that
\begin{equation*}
\INTSeg{\INTDom{\Big( \partial_t \phi(t,x) + \big\langle \nabla_x \phi(t,x) , v(t,x) \big\rangle \Big)}{\R^d}{\mu(t)(x)}}{t}{0}{T} = 0
\end{equation*}
for all $\phi \in C^{\infty}_c((0,T) \times \R^d,\R^d)$, which together with the fact that $\mu(0) = \mu^0$ as a consequence of \eqref{eq:AscoliMeasures} equivalently means that the pair $(\mu(\cdot),v(\cdot)) \in \AC([0,T],\Pcal_p(\R^d)) \times \Lcal([0,T],C^0(\R^d,\R^d))$ solves 
\begin{equation}
\label{eq:CompactnessIntermCauchy}
\left\{
\begin{aligned}
& \partial_t \mu(t) + \Div_x(v(t) \mu(t)) = 0, \\
& \mu(0) = \mu^0.
\end{aligned}
\right.
\end{equation}
To conclude, there simply remains to show that $v(t) \in V(t,\mu(t))$ for $\Lcal^1$-almost every $t \in [0,T]$. To this end, observe that as a consequence of Hypothesis \ref{hyp:CICL}-$(iv)$ combined with Lemma \ref{lem:MeasurableSel}-$(a)$ and $(c)$, there exists for each $n \geq 1$ a measurable selection
\begin{equation}
\label{eq:TildeVnDef}
t \in [0,T] \mapsto \tilde{v}_n(t) \in V(t,\mu(t)), 
\end{equation}
which satisfies
\begin{equation*}
\dsf_{\sup}(\tilde{v}_n(t),v_n(t)) \leq L(t) W_p(\mu_n(t),\mu(t)) ~\underset{n \to +\infty}{\longrightarrow}~ 0
\end{equation*}
for $\Lcal^1$-almost every $t \in [0,T]$, where we also used \eqref{eq:AscoliMeasures}. Whence, by a direct application of Lebesgue's dominated convergence theorem, it necessarily follows that 
\begin{equation*}
\INTSeg{\NormC{\tilde{v}_n(t) - v_n(t)}{0}{B(0,R),\R^d}}{t}{0}{T} ~\underset{n \to +\infty}{\longrightarrow}~ 0
\end{equation*}
for each $R>0$, which combined with \eqref{eq:WeakConvergenceVel0} implies in particular that
\begin{equation}
\label{eq:TildeVnConv}
\tilde{v}_n(\cdot) ~\underset{n \to +\infty}{\rightharpoonup}~ v(\cdot) 
\end{equation}
weakly in $L^1([0,T],C^0(B(0,R),\R^d))$, along adequate subsequences which depend on $R>0$. Repeating the reasoning at the end of Step 5 in the proof of Theorem \ref{thm:Peano}, one can finally prove $v(t) \in V(t,\mu(t))$ for $\Lcal^1$-almost every $t \in [0,T]$, which together with \eqref{eq:CompactnessIntermCauchy} concludes the proof of Theorem \ref{thm:Compactness}.
\end{proof}

The compactness result established in Theorem \ref{thm:Compactness} crucially relies on the convexity of the admissible velocities, and does not hold anymore without this assumption. In the following theorem, we show that it is nonetheless possible to precisely characterise the closure of the solution of the solution set in the absence of convexity, and that the latter coincides with that of the convexified dynamics. 

\begin{thm}[Relaxation property for continuity inclusions]
\label{thm:Relaxation}
Let $\mu^0 \in \Pcal_p(\R^d)$ and $V : [0,T] \times \Pcal_p(\R^d) \tto C^0(\R^d,\R^d)$ be a set-valued map satisfying Hypotheses \ref{hyp:CICL}. Then, for every solution $\mu(\cdot) \in \AC([0,T],\Pcal_p(\R^d))$ of the relaxed Cauchy problem
\begin{equation}
\label{eq:RelaxedContInc}
\left\{
\begin{aligned}
& \partial_t \mu(t) \in - \Div_x \Big( \co V(t,\mu(t)) \mu(t) \Big), \\
& \mu(0) = \mu^0,
\end{aligned}
\right.
\end{equation}
and any $\delta > 0$, there exists a solution $\mu_{\delta}(\cdot) \in \AC([0,T],\Pcal_p(\R^d))$ of the Cauchy problem
\begin{equation}
\label{eq:UnRelaxedContInc}
\left\{
\begin{aligned}
& \partial_t \mu_{\delta}(t) \in - \Div_x \Big( V(t,\mu_{\delta}(t)) \mu_{\delta}(t) \Big), \\
& \mu_{\delta}(0) = \mu^0,
\end{aligned}
\right.
\end{equation}
which satisfies
\begin{equation*}
\sup_{t \in [0,T]} W_p(\mu(t),\mu_{\delta}(t)) \leq \delta. 
\end{equation*}
In particular, the solution set of \eqref{eq:UnRelaxedContInc} is dense in that of \eqref{eq:RelaxedContInc} for the topology of uniform convergence over $C^0([0,T],\Pcal_p(\R^d))$.
\end{thm}

The proof of Theorem \ref{thm:Relaxation} will be split into three steps. In Step 1, we start by choosing a solution of \eqref{eq:RelaxedContInc} and an adequate subdivision of the interval $[0,T]$ that we use along with Aumann's theorem to define an intermediate curve, of which it is shown in Step 2 that it is close to the original one in $C^0([0,T],\Pcal_p(\R^d))$. We then conclude in Step 3 by applying the generalised Filippov estimates of Theorem \ref{thm:Filippov}, which provides us with a solution of the Cauchy problem \eqref{eq:UnRelaxedContInc} that remains sufficiently close to the intermediate curve, and thus to the original one.

\begin{proof}[Proof of Theorem \ref{thm:Relaxation}]
We start by observing that, if $V : [0,T] \times \Pcal_p(\R^d) \tto C^0(\R^d,\R^d)$ is a set-valued map satisfying Hypotheses \ref{hyp:CICL} and $\mu(\cdot) \in \AC([0,T],\Pcal_p(\R^d))$ is a solution of either \eqref{eq:RelaxedContInc} or \eqref{eq:UnRelaxedContInc}, then by Proposition \ref{prop:Moment} every measurable selection $v(\cdot) \in \Lcal([0,T],C^0(\R^d,\R^d))$ of either $\co V(\cdot,\mu(\cdot))$ or $V(\cdot,\mu(\cdot))$ necessarily satisfies
\begin{equation}
\label{eq:RelaxationVelEst}
|v(t,x)| \leq (1+\Cpazo)m(t) \big( 1+|x| \big), 
\end{equation}
for $\Lcal^1$-almost every $t \in [0,T]$ and all $x \in \R^d$, where $\Cpazo > 0$ is a constant that only depends on the magnitudes of $p,\Mpazo_p(\mu^0)$ and $\Norm{m(\cdot)}_1$. Thus, denoting by $\Beta \in \Pcal(\R^d \times \Sigma_T)$ the superposition measure associated with such a velocity field via Definition \ref{def:SuperpositionMeas}, it directly follows from Gr\"onwall's lemma that
\begin{equation}
\label{eq:RelaxationCurveEst}
\NormC{\sigma(\cdot)}{0}{[0,T],\R^d} \leq \Cpazo_T(1+|x|), 
\end{equation}
for $\Beta$-almost every $(x,\sigma) \in \R^d \times \Sigma_T$, where $\Cpazo_T := \max\big\{ 1 , (1+\Cpazo) \hspace{-0.1cm} \Norm{m(\cdot)}_1 \hspace{-0.1cm} \big\} \exp \big( (1+\Cpazo) \hspace{-0.1cm} \Norm{m(\cdot)}_1 \hspace{-0.1cm} \big)$.   

\paragraph*{Step 1 -- Construction of an intermediate curve via Aumann's theorem.} Given an arbitrary real number $\delta > 0$, the fact that $\mu^0 \in \Pcal_p(\R^d)$ implies by standard results in measure theory that there exists a positive radius $R_{\delta} > 0$ for which
\begin{equation}
\label{eq:RelaxationMeasureEst}
\bigg( \INTDom{(1+|x|)^p}{\{ x \; \textnormal{s.t.} \; |x| \geq R_{\delta}/ \Cpazo_T-1 \}}{\mu^0(x)} \bigg)^{1/p} \leq \frac{\delta}{2 (1+\Cpazo) (1+\Cpazo_T) (1+\Norm{m(\cdot)}_1)}.
\end{equation}
Besides, remarking that $m(\cdot) \in L^1([0,T],\R_+)$, there exists a subdivision $0 = t_0 < t_1 < \dots < t_N = T$ of the interval $[0,T]$ such that 
\begin{equation}
\label{eq:RelaxationSubdivisionEst}
\INTSeg{m(s)}{s}{t_i}{t_{i+1}} \leq \frac{\delta}{2 (1 + \Cpazo)(1 + R_{\delta})}
\end{equation}
for each $i \in \{ 0, \dots, N-1\}$. From Hypotheses \ref{hyp:CICL} and Lemma \ref{lem:MeasurableSel}, one can check that the restricted set-valued maps $t \in [0,T] \tto V(t,\mu(t))_{|B(0,R_{\delta})}$ and $t \in [0,T] \tto \co V(t,\mu(t))_{|B(0,R_{\delta})}$ are $\Lcal^1$-measurable. Furthermore, they have compact images and are integrably bounded, and using the fact that the topology induced by $\dcc(\cdot,\cdot)$ on $C^0(B(0,R_{\delta}),\R^d)$ coincides with the usual norm topology, one has that
\begin{equation}
\label{eq:ConvHullCoincide}
\Big( \co V(t,\mu(t)) \Big)_{|B(0,R_{\delta})} = \co \Big( V(t,\mu(t))_{|B(0,R_{\delta})} \Big),
\end{equation}
wherein the second convex hull is taken in the Banach space $(C^0(B(0,R_{\delta}),\R^d),\NormC{\cdot}{0}{B(0,R_{\delta}),\R^d})$. Thus, denoting by $t \in [0,T] \mapsto v(t) \in \co V(t,\mu(t))$ the velocity selection associated with a given solution $\mu(\cdot) \in \AC([0,T],\Pcal_p(\R^d))$ of \eqref{eq:RelaxedContInc}, it follows from Theorem \ref{thm:Aumann} that there exist measurable selections $t \in [t_i,t_{i+1}] \mapsto v_i^{\delta}(t) \in V(t,\mu(t))_{|B(0,R_{\delta})}$ such that
\begin{equation}
\label{eq:RelaxationConvexityApprox}
\NormC{ \, \INTSeg{v_{|B(0,R_{\delta})}(s)}{s}{t_i}{t_{i+1}} - \INTSeg{v_i^{\delta}(s)}{s}{t_i}{t_{i+1}} \,}{0}{B(0,R_{\delta}),\R^d} \leq \frac{\delta}{N}
\end{equation}
for every $i \in \{0,\dots,N-1\}$, where the maps $v_i^{\delta}(\cdot),v_{|B(0,R_{\delta})}(\cdot)$ are elements of the separable Banach space $L^1([0,T],C^0(B(0,R_{\delta}),\R^d))$ and the integrals are understood in the sense of Bochner. Notice now that $(t,w) \in [t_i,t_{i+1}] \times C^0(\R^d,\R^d) \mapsto \|w-v_i^{\delta}(t) \|_{C^0(B(0,R_{\delta}),\R^d)}$ are Carathéodory, and also that
\begin{equation*}
t \in [t_i,t_{i+1}] \tto V(t,\mu(t)) \cap \bigg\{ w \in C^0(\R^d,\R^d) ~\, \textnormal{s.t.}~ \|w-v_i^{\delta}(t) \|_{C^0(B(0,R_{\delta}),\R^d)} = 0 \bigg\}
\end{equation*}
have nonempty images for each $i \in \{0,\dots,N-1\}$ by construction. Thence, by Lemma \ref{lem:MeasurableSel}-$(a)$, there exist measurable selections $t \in [t_i,t_{i+1}] \mapsto v_i(t) \in V(t,\mu(t))$ such that $v_i(t)_{|B(0,R_{\delta})} = v_i^{\delta}(t)$ for $\Lcal^1$-almost every $t \in [0,T]$. Therefore, the velocity field $w : [0,T] \times \R^d \to \R^d$ defined by 
\begin{equation*}
w(t,x) := \sum_{i=0}^{N-1} \mathds{1}_{[t_i,t_{i+1})}(t) v_i(t,x)
\end{equation*}
for $\Lcal^1$-almost every $t \in [0,T]$ and all $x \in \R^d$ is Carathéodory, and satisfies Hypotheses \ref{hyp:CE}. As such, by Theorem \ref{thm:Wellposedness}, it generates a unique solution $\nu(\cdot) \in \AC([0,T],\Pcal_p(\R^d))$ to the Cauchy problem 
\begin{equation*}
\left\{
\begin{aligned}
& \partial_t \nu(t) + \Div_x (w(t) \nu(t)) = 0, \\
& \nu(0) = \mu^0.
\end{aligned}
\right.
\end{equation*} 

\paragraph*{Step 2 -- Estimating of the Wasserstein distance between $\mu(\cdot)$ and $\nu(\cdot)$.} Let $\Beta_{\mu},\Beta_{\nu} \in \Pcal(\R^d \times \Sigma_T)$ be two superposition measures given by Theorem \ref{thm:Superposition} which satisfy
\begin{equation*}
(\esf_t)_{\sharp} \Beta_{\mu} = \mu(t) \qquad \text{and} \qquad (\esf_t)_{\sharp} \Beta_{\nu} = \nu(t)
\end{equation*}
for all times $t \in [0,T]$, and $\hat{\Beta}_{\mu,\nu} \in \Gamma(\Beta_{\mu},\Beta_{\nu})$ be a transport plan given by Lemma \ref{lem:SuperpositionPlan}, for which
\begin{equation*}
(\pi_{\R^d},\pi_{\R^d})_{\sharp} \hat{\Beta}_{\mu,\nu} = (\Id,\Id)_{\sharp} \mu^0 \qquad \text{and} \qquad \gamma(t) := (\esf_t,\esf_t)_{\sharp} \hat{\Beta}_{\mu,\nu} \in \Gamma_o(\mu(t),\nu(t)).
\end{equation*}
Then
\begin{equation}
\label{eq:RelaxationEst1}
\begin{aligned}
W_p(\mu(t),\nu(t)) & = \bigg( \INTDom{|x-y|^p}{\R^{2d}}{\gamma(t)(x,y)} \bigg)^{1/p} \\
& = \bigg( \INTDom{|\sigma_{\mu}(t) - \sigma_{\nu}(t)|^p}{(\R^d \times \Sigma_T)^2}{\hat{\Beta}_{\mu,\nu}(x,\sigma_{\mu},y,\sigma_{\nu})} \bigg)^{1/p} \\
& \leq \bigg( \INTDom{\bigg| \INTSeg{\Big( v(s,\sigma_{\mu}(s)) - v(s,\sigma_{\nu}(s)) \Big)}{s}{0}{t} \bigg|^p}{(\R^d \times \Sigma_T)^2}{\hat{\Beta}_{\mu,\nu}(x,\sigma_{\mu},y,\sigma_{\nu})} \bigg)^{1/p} \\
& \hspace{1.85cm} + \bigg( \INTDom{\bigg| \INTSeg{\Big( v(s,\sigma_{\nu}(s)) - w(s,\sigma_{\nu}(s)) \Big)}{s}{0}{t} \bigg|^p}{\R^d \times \Sigma_T}{\Beta_{\nu}(y,\sigma_{\nu})} \bigg)^{1/p}, 
\end{aligned}
\end{equation}
for all times $t \in [0,T]$, where we used the fact that $\Beta_{\mu},\Beta_{\nu}$ are superposition measures in the sense of Definition \ref{def:SuperpositionMeas}. By repeating the computations detailed in Appendices \ref{section:AppendixMoment} and \ref{section:AppendixGronwall}, it can be shown that 
\begin{equation}
\label{eq:RelaxationEst2}
\begin{aligned}
& \bigg( \INTDom{\bigg| \INTSeg{\Big( v(s,\sigma_{\mu}(s)) - v(s,\sigma_{\nu}(s)) \Big)}{s}{0}{t} \bigg|^p}{(\R^d \times \Sigma_T)^2}{\hat{\Beta}_{\mu,\nu}(x,\sigma_{\mu},y,\sigma_{\nu})} \bigg)^{1/p} \\
& \hspace{5cm} \leq \, \NormL{l(\cdot)}{1}{[0,t]}^{(p-1)/p} \bigg( \INTSeg{l(s)W_p^p(\mu(s),\nu(s))}{s}{0}{t} \bigg)^{1/p}, 
\end{aligned}
\end{equation}
for all times $t \in [0,T]$. Concerning the second term in the last inequality of the right-hand side of \eqref{eq:RelaxationEst1}, it can be bounded from above by the sum of two integrals as 
\begin{equation}
\label{eq:RelaxationEst31}
\begin{aligned}
& \bigg( \INTDom{\bigg| \INTSeg{\Big( v(s,\sigma_{\nu}(s)) - w(s,\sigma_{\nu}(s)) \Big)}{s}{0}{t} \; \bigg|^p}{\R^d \times \Sigma_T}{\Beta_{\nu}(y,\sigma_{\nu})} \bigg)^{1/p} \\
& \leq \bigg( \INTDom{\bigg| \INTSeg{\Big( v(s,\sigma_{\nu}(s)) - w(s,\sigma_{\nu}(s)) \Big)}{s}{0}{t} \; \bigg|^p}{\{(y,\sigma_{\nu}) \; \textnormal{s.t.} \; \NormC{\sigma_{\nu}(\cdot)}{0}{[0,T],\R^d} \leq R_{\delta} \}}{\Beta_{\nu}(y,\sigma_{\nu})} \bigg)^{1/p} \\
& \hspace{0.45cm} + \bigg( \INTDom{\bigg| \INTSeg{\Big( v(s,\sigma_{\nu}(s)) - w(s,\sigma_{\nu}(s)) \Big)}{s}{0}{t} \; \bigg|^p}{\{(y,\sigma_{\nu}) \; \textnormal{s.t.} \; \NormC{\sigma_{\nu}(\cdot)}{0}{[0,T],\R^d} > R_{\delta} \}}{\Beta_{\nu}(y,\sigma_{\nu})} \bigg)^{1/p}.
\end{aligned}
\end{equation}
Recalling that $(\pi_{\R^d})_{\sharp} \Beta_{\nu} = \mu^0$ and invoking the sublinearity estimates of \eqref{eq:RelaxationVelEst} and \eqref{eq:RelaxationCurveEst}, the second term in the right-hand side of \eqref{eq:RelaxationEst31} can be estimated as 
\begin{equation}
\label{eq:RelaxationEst32}
\begin{aligned}
& \bigg( \INTDom{\bigg| \INTSeg{\Big( v(s,\sigma_{\nu}(s)) - w(s,\sigma_{\nu}(s)) \Big)}{s}{0}{t} \; \bigg|^p}{\{(y,\sigma_{\nu}) \; \textnormal{s.t.} \; \NormC{\sigma_{\nu}(\cdot)}{0}{[0,T],\R^d} > R_{\delta} \}}{\Beta_{\nu}(y,\sigma_{\nu})} \bigg)^{1/p} \\
& \leq 2 (1+\Cpazo) \Norm{m(\cdot)}_1 \bigg( \INTDom{ \Big( 1 + \NormC{\sigma_{\nu}(\cdot)}{0}{[0,T],\R^d} \Big)^p}{\{(y,\sigma_{\nu}) \; \textnormal{s.t.} \; \NormC{\sigma_{\nu}(\cdot)}{0}{[0,T],\R^d} > R_{\delta} \}}{\Beta_{\nu}(y,\sigma_{\nu})} \bigg)^{1/p} \\
& \leq 2 (1+\Cpazo) (1+\Cpazo_T) \Norm{m(\cdot)}_1 \bigg( \INTDom{(1+|x|)^p}{\{ x \; \textnormal{s.t.} \; |x| > R_{\delta}/\Cpazo_T-1 \}}{\mu^0(x)} \bigg)^{1/p}, 
\end{aligned}
\end{equation}
which together with our choice of $R_{\delta} > 0$ in \eqref{eq:RelaxationMeasureEst} further yields
\begin{equation}
\label{eq:RelaxationEst33}
\bigg( \INTDom{\bigg| \INTSeg{\Big( v(s,\sigma_{\nu}(s)) - w(s,\sigma_{\nu}(s)) \Big)}{s}{0}{t} \; \bigg|^p}{\{(y,\sigma_{\nu}) \; \textnormal{s.t.} \; \NormC{\sigma_{\nu}(\cdot)}{0}{[0,T],\R^d} > R_{\delta} \}}{\Beta_{\nu}(y,\sigma_{\nu})} \bigg)^{1/p} \leq \delta. 
\end{equation}
We now focus our attention on the first term in the right-hand side of \eqref{eq:RelaxationEst31}. By construction, there exists an integer $j \in \{0,\dots,N-1\}$ such that $t \in [t_j,t_{j+1}]$, which means that
\begin{equation}
\label{eq:RelaxationEst34}
\begin{aligned}
& \bigg( \INTDom{\bigg| \INTSeg{\Big( v(s,\sigma_{\nu}(s)) - w(s,\sigma_{\nu}(s)) \Big)}{s}{0}{t} \; \bigg|^p}{\{(y,\sigma_{\nu}) \; \textnormal{s.t.} \; \NormC{\sigma_{\nu}(\cdot)}{0}{[0,T],\R^d} \leq R_{\delta} \}}{\Beta_{\nu}(y,\sigma_{\nu})} \bigg)^{1/p} \\
& \leq \sum_{i=0}^{j-1} \bigg( \INTDom{\bigg| \INTSeg{\Big( v(s,\sigma_{\nu}(s)) - v_i(s,\sigma_{\nu}(s)) \Big)}{s}{t_i}{t_{i+1}} \; \bigg|^p}{\{(y,\sigma_{\nu}) \; \textnormal{s.t.} \; \NormC{\sigma_{\nu}(\cdot)}{0}{[0,T],\R^d} \leq R_{\delta} \}}{\Beta_{\nu}(y,\sigma_{\nu})} \bigg)^{1/p} \\
& \hspace{0.45cm} + \bigg( \INTDom{\bigg| \INTSeg{\Big( v(s,\sigma_{\nu}(s)) - v_j(s,\sigma_{\nu}(s)) \Big)}{s}{t_j}{t} \; \bigg|^p}{\{(y,\sigma_{\nu}) \; \textnormal{s.t.} \; \NormC{\sigma_{\nu}(\cdot)}{0}{[0,T],\R^d} \leq R_{\delta} \}}{\Beta_{\nu}(y,\sigma_{\nu})} \bigg)^{1/p} \\
& \leq \sum_{i=0}^{j-1} \bigg( \INTDom{\bigg| \INTSeg{\Big( v(s,\sigma_{\nu}(s)) - v_i(s,\sigma_{\nu}(s)) \Big)}{s}{t_i}{t_{i+1}} \; \bigg|^p}{\{(y,\sigma_{\nu}) \; \textnormal{s.t.} \; \NormC{\sigma_{\nu}(\cdot)}{0}{[0,T],\R^d} \leq R_{\delta} \}}{\Beta_{\nu}(y,\sigma_{\nu})} \bigg)^{1/p} + \delta, 
\end{aligned}
\end{equation}
where we used the estimates of \eqref{eq:RelaxationVelEst} and \eqref{eq:RelaxationSubdivisionEst} in the second inequality. Now, given an arbitrary integer $i \in \{0,\dots,j-1\}$, the facts that $v(\cdot),v_i(\cdot) \in \Lcal([0,T],C^0(\R^d,\R^d))$ both satisfy the sublinearity estimate \eqref{eq:RelaxationVelEst} along with the regularity assumption of Hypothesis \ref{hyp:CICL}-$(iii)$ allow us to write
\begin{equation}
\label{eq:RelaxationEst35}
\begin{aligned}
& \bigg( \INTDom{\bigg| \INTSeg{\Big( v(s,\sigma_{\nu}(s)) - v_i(s,\sigma_{\nu}(s)) \Big)}{s}{t_i}{t_{i+1}} \; \bigg|^p}{\{(y,\sigma_{\nu}) \; \textnormal{s.t.} \; \NormC{\sigma_{\nu}(\cdot)}{0}{[0,T],\R^d} \leq R_{\delta} \}}{\Beta_{\nu}(y,\sigma_{\nu})} \bigg)^{1/p} \\
& \leq \bigg( \INTDom{\bigg| \INTSeg{\Big( v(s,\sigma_{\nu}(t_i)) - v_i(s,\sigma_{\nu}(t_i)) \Big)}{s}{t_i}{t_{i+1}} \; \bigg|^p}{\{(y,\sigma_{\nu}) \; \textnormal{s.t.} \; \NormC{\sigma_{\nu}(\cdot)}{0}{[0,T],\R^d} \leq R_{\delta} \}}{\Beta_{\nu}(y,\sigma_{\nu})} \bigg)^{1/p} \\
& \hspace{0.45cm} + \bigg( \INTDom{\bigg| \INTSeg{\Big( v(s,\sigma_{\nu}(s)) - v(s,\sigma_{\nu}(t_i)) \Big)}{s}{t_i}{t_{i+1}} \; \bigg|^p}{\{(y,\sigma_{\nu}) \; \textnormal{s.t.} \; \NormC{\sigma_{\nu}(\cdot)}{0}{[0,T],\R^d} \leq R_{\delta} \}}{\Beta_{\nu}(y,\sigma_{\nu})} \bigg)^{1/p} \\
& \hspace{0.45cm} + \bigg( \INTDom{\bigg| \INTSeg{\Big( v_i(s,\sigma_{\nu}(s)) - v_i(s,\sigma_{\nu}(t_i)) \Big)}{s}{t_i}{t_{i+1}} \; \bigg|^p}{\{(y,\sigma_{\nu}) \; \textnormal{s.t.} \; \NormC{\sigma_{\nu}(\cdot)}{0}{[0,T],\R^d} \leq R_{\delta} \}}{\Beta_{\nu}(y,\sigma_{\nu})} \bigg)^{1/p} \\
& \leq \NormC{\, \INTSeg{v(s)_{|B(0,R_{\delta})}}{s}{t_i}{t_{i+1}} - \INTSeg{v_i^{\delta}(s)}{s}{t_i}{t_{i+1}} \, }{0}{B(0,R_{\delta}),\R^d} \\
& \hspace{0.45cm} + 2 \Bigg( \INTDom{\bigg| \INTSeg{l(s)\INTSeg{m(\zeta) \Big( 1 + |\sigma_{\nu}(\zeta)| + \Cpazo \Big)}{\zeta}{t_i}{s}}{s}{t_i}{t_{i+1}} \; \bigg|^p}{\{(y,\sigma_{\nu}) \; \textnormal{s.t.} \; \NormC{\sigma_{\nu}(\cdot)}{0}{[0,T],\R^d} \leq R_{\delta} \}}{\Beta_{\nu}(y,\sigma_{\nu})} \Bigg)^{1/p} \\
& \leq \NormC{\, \INTSeg{v(s)_{|B(0,R_{\delta})}}{s}{t_i}{t_{i+1}} - \INTSeg{v_i^{\delta}(s)}{s}{t_i}{t_{i+1}} \, }{0}{B(0,R_{\delta},\R^d)} \hspace{-0.1cm} + 2 \big(1+ R_{\delta} + \Cpazo \big) \INTSeg{l(s) \INTSeg{m(\zeta)}{\zeta}{t_i}{s}}{s}{t_i}{t_{i+1}} \\
& \leq \frac{\delta}{N} + \delta \INTSeg{l(s)}{s}{t_i}{t_{i+1}}, 
\end{aligned}
\end{equation}
where we leveraged the preliminary estimate \eqref{eq:RelaxationVelEst} in the second and third inequalities, as well as \eqref{eq:RelaxationSubdivisionEst} and \eqref{eq:RelaxationConvexityApprox} in the last one. Hence, by inserting the estimates of \eqref{eq:RelaxationEst35} for $i \in \{0,\dots,j-1\}$ into \eqref{eq:RelaxationEst34}, combining the resulting expression with \eqref{eq:RelaxationEst33} and plugging the latter in \eqref{eq:RelaxationEst31}, we then get 
\begin{equation}
\label{eq:RelaxationEst36}
\bigg( \INTDom{\bigg| \INTSeg{\Big( v(s,\sigma_{\nu}(s)) - w(s,\sigma_{\nu}(s)) \Big)}{s}{0}{t} \bigg|^p}{\R^d \times \Sigma_T}{\Beta_{\nu}(y,\sigma_{\nu})} \bigg)^{1/p} \leq \delta \big( \, 3 \, + \Norm{l(\cdot)}_1 \big).
\end{equation}
By merging \eqref{eq:RelaxationEst2} and \eqref{eq:RelaxationEst36} into \eqref{eq:RelaxationEst1}, raising the resulting inequality to the power $p$, applying Gr\"onwall's lemma and then raising the corresponding expression to the power $1/p$, we finally obtain 
\begin{equation}
\label{eq:RelaxationEst4}
W_p(\mu(t),\nu(t)) \leq \delta \, C_p \big( 3 \, + \Norm{l(\cdot)}_1 \hspace{-0.1cm} \big) \exp \Big( C_p' \Norm{l(\cdot)}_1^p \Big), 
\end{equation}
which holds for all times $t \in [0,T]$. 


\paragraph*{Step 3 -- Back to the unrelaxed problem using Filippov's estimates.} 

At this stage, one should note that $\nu(\cdot) \in \AC([0,T],\Pcal_p(\R^d))$ is not a solution of \eqref{eq:UnRelaxedContInc}, since $w(t) \in V(t,\mu(t))$ for $\Lcal^1$-almost every $t \in [0,T]$. Nevertheless, one has that the mismatch function defined by 
\begin{equation*}
\eta_{\delta}(t) := \dist_{C^0(B(0,R_{\delta}),\R^d)} \Big( w(t) \, ; V(t,\nu(t)) \Big)
\end{equation*}
for $\Lcal^1$-almost every $t \in [0,T]$ is Lebesgue integrable and satisfies
\begin{equation*}
\begin{aligned}
\eta_{\delta}(t) & \leq L(t) W_p(\mu(t),\nu(t)) \\
&\leq \delta C_p L(t) \big( 3 \; + \Norm{l(\cdot)}_1 \hspace{-0.1cm} \big) \exp \Big( C_p' \Norm{l(\cdot)}_1^p \hspace{-0.05cm} \Big),
\end{aligned}
\end{equation*}
as a consequence of Hypothesis \ref{hyp:CICL}-$(iv)$. Thus, by leveraging the results of Theorem \ref{thm:Filippov} while recollecting the a priori bound of \eqref{eq:RelaxationMeasureEst}, there exists a solution $\mu_{\delta}(\cdot) \in \AC([0,T],\Pcal_p(\R^d))$ of the unrelaxed Cauchy problem \eqref{eq:UnRelaxedContInc} which satisfies 
\begin{equation}
\label{eq:RelaxationEst5}
\begin{aligned}
W_p(\mu_{\delta}(t),\nu(t)) & \leq C_p \Bigg( \delta C_p \hspace{-0.05cm} \Norm{L(\cdot)}_1 \hspace{-0.05cm} \big( 3 \; + \Norm{l(\cdot)}_1 \hspace{-0.1cm} \big) \exp \Big( C_p' \Norm{l(\cdot)}_1^p \Big) + \delta \Bigg)\exp \Big( C_p' \NormL{l(\cdot)}{1}{[0,t]}^p \hspace{-0.1cm} + \chi_p(t) \Big) \\
& \leq \delta C_p \bigg( 1 + \big(3 \; + \Norm{l(\cdot)}_1 \hspace{-0.1cm} \big) \Norm{\chi_p(\cdot)}_{\infty} \hspace{-0.1cm} \bigg) \exp \Big( C_p' \Norm{l(\cdot)}_1^p +  \Norm{\chi_p(\cdot)}_{\infty} \hspace{
-0.1cm} \Big), 
\end{aligned}
\end{equation}
for all times $t \in [0,T]$, where the map $\chi_p(\cdot)$ is defined in \eqref{eq:ChiDef}. Thus, by combining \eqref{eq:RelaxationEst4} and \eqref{eq:RelaxationEst5}, applying the triangle inequality and redefining the constant $\delta > 0$ as 
\begin{equation*}
\delta := \frac{\delta}{C_p \bigg( \big( 3 \; + \Norm{l(\cdot)}_1 \hspace{-0.1cm} \big) \Big( 1 + \Norm{\chi_p(\cdot)}_{\infty} \exp \big( \hspace{-0.1cm} \Norm{\chi_p(\cdot)}_{\infty} \hspace{-0.1cm} \big) \Big) + \exp \big( \hspace{-0.1cm} \Norm{\chi_p(\cdot)}_{\infty} \hspace{-0.1cm} \big) \bigg) \exp \Big( C_p' \Norm{l(\cdot)}_1^p \Big)}
\end{equation*}
we can finally conclude that the solution $\mu_{\delta}(\cdot)$ of the unrelaxed Cauchy problem \eqref{eq:UnRelaxedContInc} is such that
\begin{equation*}
\sup_{t \in [0,T]} W_p(\mu(t),\mu_{\delta}(t)) \leq \delta, 
\end{equation*}
which ends the proof of Theorem \ref{thm:Relaxation}. 
\end{proof}


\addcontentsline{toc}{section}{Appendices}
\section*{Appendices}

In this auxiliary section, we detail the proofs of several technical results appearing in the manuscript. 


\setcounter{subsection}{0} 
\renewcommand{\thesubsection}{A} 

\subsection{Proof of Lemma \ref{lem:MeasurableSel}}
\label{section:AppendixMeas}

\setcounter{Def}{0} \renewcommand{\thethm}{A.\arabic{Def}} 
\setcounter{equation}{0} \renewcommand{\theequation}{A.\arabic{equation}}

In this first appendix, we detail parts of the proof of the measurable selection principles of Lemma \ref{lem:MeasurableSel}, as these latter rely on somewhat non-standard assumptions. 

\begin{proof}[Proof of Lemma \ref{lem:MeasurableSel}]
In what follows, we start with the proof of item $(a)$, and proceed with that of item $(b)$. The proof of item $(c)$ is completely standard and can be found e.g. in \cite[Theorem 8.2.8]{Aubin1990}. The only delicate thing that needs proving in the statement of Lemma \ref{lem:MeasurableSel}-$(a)$ is the fact that the set-valued map under consideration is measurable. To this end, we consider the multifunctions
\begin{equation*}
t \in [0,T] \tto \Apazo_n(t) := \Big\{ y \in Y ~\,\text{s.t.}~ \varphi_n(t,y) \leq L(t) \Big\}
\end{equation*}
defined for each $n \geq 1$, which are $\Lcal^1$-measurable by \cite[Theorem 8.2.9]{Aubin1990}. For each closed set $\Cpazo \subset Y$, this implies in particular by \cite[Theorem 8.1.4]{Aubin1990} that
\begin{equation*}
\Dcal_{\Cpazo} := \bigcap_{n \geq 1} \Big\{ t \in [0,T] ~\, \text{s.t.}~ \Fpazo(t) \cap \Apazo_n(t) \cap \Cpazo \neq \emptyset \Big\}
\end{equation*}
is an $\Lcal^1$-measurable set. To conclude the proof of our claim, there remains to show that 
\begin{equation*}
\Dcal_{\Cpazo} = \Big\{ t \in [0,T] ~\, \text{s.t.}~ \Fpazo(t) \cap \Apazo(t) \cap \Cpazo \neq \emptyset \Big\}, 
\end{equation*}
where $\Apazo(t) := \big\{ y \in Y ~\,\text{s.t.}~ \varphi(t,y) \leq L(t)\big\}$ for $\Lcal^1$-almost every $t \in [0,T]$. By construction, one has
\begin{equation*}
\Big\{ t \in [0,T] ~\, \text{s.t.}~ \Fpazo(t) \cap \Apazo(t) \cap \Cpazo \neq \emptyset \Big\} \subset \Dcal_{\Cpazo}, 
\end{equation*}
as a consequence of \eqref{eq:varphisup}. Conversely for any $\tau \in \Dcal_{\Cpazo}$, remark that the sets defined for each $n \geq 1$ by
\begin{equation*}
\Bpazo_n(\tau) := \Fpazo(\tau) \cap \Apazo_n(\tau) \cap \Cpazo
\end{equation*}
form a non-increasing sequence since maps $(\varphi_n(\cdot,\cdot))$ are pointwisely non-decreasing. Moreover under our standing assumptions, the sets $\Bpazo_n(\tau) \subset Y$ are compact and nonempty for each $n \geq 1$. Whence, it follows from Cantor's intersection theorem (see e.g. \cite[Theorem 2.6]{Rudin1987}) that 
\begin{equation*}
\Bpazo(\tau) := \bigcap_{n \geq 1} \Bpazo_n(\tau) \neq \emptyset, 
\end{equation*}
which together with the fact that $\Bpazo(\tau) \subset \Fpazo(\tau) \cap \Apazo(\tau) \cap \Cpazo$ finally yields that 
\begin{equation*}
\tau \in \Big\{ t \in [0,T] ~\, \text{s.t.}~ \Fpazo(t) \cap \Apazo(t) \cap \Cpazo \neq \emptyset \Big\} 
\end{equation*}
and concludes the proof of our claim. 

We now shift our focus to the statements of Lemma \ref{lem:MeasurableSel}-$(b)$, and start by observing that the set-valued map appearing therein can be rewritten as 
\begin{equation*}
t \in [0,T] \tto \Fpazo(t )\cap \bigg\{ y \in Y ~\, \textnormal{s.t.}~ \varphi(t,y) \leq \inf_{z \in \Fpazo(t)} \varphi(t,z) \bigg\}.
\end{equation*}
Thus by what precedes, it is sufficient for our purpose to show that the map 
\begin{equation*}
t \in [0,T] \mapsto \inf_{z \in \Fpazo(t)} \varphi(t,z) \in \R_+
\end{equation*}
is $\Lcal^1$-measurable. It follows from classical measurability results (see e.g. \cite[Theorem 8.2.11]{Aubin1990}) that for each $n \geq 1$, the map 
\begin{equation*}
\varpi_n : t \in [0,T] \mapsto \inf_{z \in \Fpazo(t)} \varphi_n(t,z)
\end{equation*}
is $\Lcal^1$-measurable. Since the sequence $(\varpi_n(\cdot))$ is pointwisely non-decreasing and bounded, the limits
\begin{equation*}
\varpi(t) := \lim_{n \to +\infty} \varpi_n(t) = \inf_{z \in \Fpazo(t)} \varphi(t,z)
\end{equation*}
exist for $\Lcal^1$-almost every $t \in [0,T]$, and the map $\varpi(\cdot)$ is $\Lcal^1$-measurable by \cite[Theorem 8.2.5]{Aubin1990}. 
\end{proof}


\setcounter{subsection}{0} 
\renewcommand{\thesubsection}{B} 

\subsection{Proof of Lemma \ref{lem:CompleteMetric}}
\label{section:AppendixCompleteness}

\setcounter{Def}{0} \renewcommand{\thethm}{B.\arabic{Def}} 
\setcounter{equation}{0} \renewcommand{\theequation}{B.\arabic{equation}}

In this second appendix, we detail for the sake of completeness the proof of Lemma \ref{lem:CompleteMetric}, which mimics standard arguments used to show that the $L^p$-spaces are complete, see e.g. to \cite[Theorem 3.11]{Rudin1987} 

\begin{proof}
Let $(v_n(\cdot)) \subset \Lcal(I,C^0(\R^d,\R^d))$ be a Cauchy sequence in the sense of \eqref{eq:ExtendedCauchy}, and observe that without loss of generality, the latter may be chosen so that 
\begin{equation}
\label{eq:Completeness1}
\INTSeg{\dsf_{\sup}(v_n(t),v_{n+1}(t))}{t}{0}{T} \leq \frac{1}{2^n}
\end{equation}
for all $n \geq 1$, up to extracting a subsequence. Consider now the sequence of real-valued mappings 
\begin{equation}
\label{eq:Completeness2}
d_n(t) := \sum_{k=1}^n \dsf_{\sup}(v_k(t),v_{k+1}(t)), 
\end{equation}
defined for $\Lcal^1$-almost every $t \in I$ and each $n \geq 1$, and observe that by \eqref{eq:Completeness1}, one has that
\begin{equation}
\sup_{n \geq 1} \INTSeg{d_n(t)}{t}{0}{T} \, \leq \, \sum_{k=1}^{+\infty} \INTSeg{\dsf_{\sup}(v_k(t),v_{k+1}(t))}{t}{0}{T} \, \leq \, 1.
\end{equation}
By \cite[Theorem 1.38]{Rudin1987}, the partial sums of series defined in \eqref{eq:Completeness2} converge for $\Lcal^1$-almost every $t \in I$ towards a map $d(\cdot) \in L^1(I,\R_+)$, which must then satisfy 
\begin{equation*}
\INTSeg{d(t)}{t}{0}{T} \leq 1. 
\end{equation*}
Thus, it necessarily follows that $d(t) < +\infty$ for $\Lcal^1$-almost every $t \in I$, which by \eqref{eq:Completeness2} together with the definition \eqref{eq:ExtendedC0} of the extended metric $\dsf_{\sup}(\cdot,\cdot)$ implies that the series of functions defined by
\begin{equation*}
w(t,x) := \sum_{k=1}^{+\infty} \big( v_{k+1}(t,x) - v_k(t,x) \big) \qquad \text{for all $x \in \R^d$} 
\end{equation*}
is normally convergent for $\Lcal^1$-almost every $t \in I$. Consequently, for every $(t,x) \in I \times \R^d$, we can define the mapping $v : I \times \R^d \to \R^d$ by 
\begin{equation*}
v(t,x) := \left\{
\begin{aligned}
& v_1(t,x) + w(t,x) ~~ & \text{if $(d_n(t))$ converges}, \\
& 0 ~~ & \text{otherwise}, 
\end{aligned}
\right. 
\end{equation*}
and observe that it is $\Lcal^1$-measurable with respect to $t \in I$ and continuous with respect to $x \in \R^d$. Moreover, the map $v(\cdot) \in \Lcal(I,C^0(\R^d,\R^d))$ is the pointwise limit of the sequence $(v_n(\cdot))$, since 
\begin{equation*}
\begin{aligned}
\dsf_{\sup}(v_n(t),v(t)) & = \dsf_{\sup} \Bigg( \sum_{k=1}^{n-1} \big( v_{k+1}(t) - v_k(t) \big) \, , \, w(t) \Bigg) \\
& = \dsf_{\sup} \Bigg( \sum_{k=1}^{n-1} \big( v_{k+1}(t) - v_k(t) \big) \, , \, \sum_{k=1}^{+\infty} \big( v_{k+1}(t) - v_k(t) \big) \Bigg) ~\underset{n \to +\infty}{\longrightarrow}~ 0, 
\end{aligned}
\end{equation*}
where we used the fact that a series that is normally convergent over $\R^d$ is also uniformly convergent. By a simple application of Fatou's lemma (see e.g. \cite[Statement 1.28]{Rudin1987}), one can finally show that 
\begin{equation*}
\INTSeg{\dsf_{\sup}(v_n(t),v(t))}{t}{0}{T} ~\underset{n \to +\infty}{\longrightarrow}~ 0, 
\end{equation*}
which concludes the proof of our claim. 
\end{proof}

\setcounter{subsection}{0} 
\renewcommand{\thesubsection}{C} 

\subsection{Proof of Proposition \ref{prop:Moment}}
\label{section:AppendixMoment}

\setcounter{Def}{0} \renewcommand{\thethm}{C.\arabic{Def}} 
\setcounter{equation}{0} \renewcommand{\theequation}{C.\arabic{equation}}

In this third appendix, we detail the proof of the moment and equi-integrability bounds displayed in Proposition \ref{prop:Moment}, which both rely on the superposition principle recalled in Theorem \ref{thm:Superposition}. 

\begin{proof}[Proof of Proposition \ref{prop:Moment}]

We do not detail the proof of \eqref{eq:MomentEst} as it is almost identical to that of \cite[Proposition 2]{ContInc}. Concerning the equi-integrability estimate \eqref{eq:UnifIntegEst}, let $\Beta_{\mu} \in \R^d \times \Sigma_T$ be a superposition measure given by Theorem \ref{thm:Superposition} such that $\mu(t) = (\esf_t)_{\sharp} \Beta_{\mu}$ for all times $t \in [\tau,T]$, and observe then that 
\begin{equation*}
|v(t,\sigma_{\mu}(t))| \leq m(t)(1+\sigma_{\mu}(t))
\end{equation*}
holds for $\Lcal^1$-almost every $t \in [\tau,T]$ and $\Beta_{\mu}$-almost every $(x,\sigma_{\mu}) \in \R^d \times \Sigma_T$ under our working assumptions. Thence, it follows from a simple application of Gr\"onwall's lemma that
\begin{equation}
\label{eq:SublinApp}
|\sigma_{\mu}(t)| \leq \bigg( |x| + \INTSeg{m(s)}{s}{\tau}{t} \bigg) \exp \bigg( \INTSeg{m(s)}{s}{\tau}{t} \bigg) \leq C_T(1+|x|), 
\end{equation}
for all times $t \in [\tau,T]$ and $\Beta_{\mu}$-almost every $(x,\sigma) \in \R^d \times \Sigma_T$, where we introduced the constant $C_T = \max\{ 1,\Norm{m(\cdot)}_1 \} \exp(\Norm{m(\cdot)}_1)$. In particular, for every $R > 0$, the latter inequality implies 
\begin{equation*}
\begin{aligned}
\bigg( \INTDom{|x|^p}{\{ x \; \textnormal{s.t.} \; |x| \geq R \}}{\mu(t)(x)} \bigg)^{1/p} & = \bigg( \INTDom{|\esf_t(x,\sigma_{\mu})|^p}{\{ (x,\sigma_{\mu}) \; \textnormal{s.t.} \; |\esf_t(x,\sigma_{\mu})| \geq R \}}{\Beta_{\mu}(x,\sigma_{\mu})} \bigg)^{1/p} \\
& \leq C_T \bigg( \INTDom{(1 +|x|)^p}{\{ (x,\sigma_{\mu}) \; \textnormal{s.t.} \; C_T(1+|x|) \geq R \}}{\Beta_{\mu}(x,\sigma_{\mu})} \bigg)^{1/p} \\
& = C_T \bigg( \INTDom{(1+|x|)^p}{\{ x \; \textnormal{s.t.} \; |x| \geq R/C_T-1\}}{\mu^0(x)} \bigg)^{1/p}
\end{aligned}
\end{equation*}
for all times $t \in [\tau,T]$, which ends the proof of our claim. 
\end{proof}

\setcounter{subsection}{0} 
\renewcommand{\thesubsection}{D} 

\subsection{Proof of Proposition \ref{prop:Gronwall}}
\label{section:AppendixGronwall}

\setcounter{Def}{0} \renewcommand{\thethm}{D.\arabic{Def}} 
\setcounter{equation}{0} \renewcommand{\theequation}{D.\arabic{equation}}

In this fourth and last appendix, we detail the proof of the stability estimates of Proposition \ref{prop:Gronwall}. 

\begin{proof}[Proof of Proposition \ref{prop:Gronwall}]

The demonstration of the global stability estimate \eqref{eq:Gronwall1} is omitted, as it is almost identical to that of \cite[Proposition 2]{ContInc}. Regarding the local stability estimate, fix some $R > 0$ and observe that under our working assumptions, the map
\begin{equation*}
t \in [\tau,T] \mapsto \NormL{v(t) - w(t)}{\infty}{B(0,R),\R^d; \, \nu(t)},
\end{equation*}
is Lebesgue integrable. In that case, one can prove by repeating the series of computations much like those of \cite[Appendix B]{ContInc} that
\begin{equation}
\label{eq:StabilityIneq5}
\begin{aligned}
W_p(\mu(t),\nu(t)) \leq C_p \Bigg( W_p(\mu_{\tau},\nu_{\tau}) + \bigg( \INTDom{\bigg(\INTSeg{|v(s,\sigma_{\nu}(s)) - w(s,\sigma_{\nu}(s))|}{s}{\tau}{t} \bigg)^p}{\R^d \times \Sigma_T}{\Beta_{\nu}(y,\sigma_{\nu})}\bigg)^{1/p} \Bigg) & \\
\times \exp \Big( C_p' \NormL{l(\cdot)}{1}{[\tau,t]}^p \hspace{-0.05cm} \Big) \, & , 
\end{aligned}
\end{equation}
for all times $t \in [\tau,T]$, wherein $\Beta_{\nu} \in \Pcal(\R^d \times \Sigma_T)$ is a superposition measure generating $\nu(\cdot) \in \AC([\tau,T],\Pcal_p(\R^d))$, whose existence follows from Theorem \ref{thm:Superposition}. We focus our attention on the integral term appearing in the right-hand side of \eqref{eq:StabilityIneq5}, and observe that the latter can be split into two parts 
\begin{equation}
\label{eq:StabilityIneq6}
\begin{aligned}
& \INTDom{\bigg(\INTSeg{|v(s,\sigma_{\nu}(s)) - w(s,\sigma_{\nu}(s))|}{s}{\tau}{t} \bigg)^p}{\R^d \times \Sigma_T}{\Beta_{\nu}(y,\sigma_{\nu})} \\
& = \INTDom{\bigg(\INTSeg{|v(s,\sigma_{\nu}(s)) - w(s,\sigma_{\nu}(s))|}{s}{\tau}{t} \bigg)^p}{\{(y,\sigma_{\nu}) \; \textnormal{s.t.} \; \NormC{\sigma_{\nu}(\cdot)}{0}{[\tau,T],\R^d} \leq R \}}{\Beta_{\nu}(y,\sigma_{\nu})} \\
& \hspace{0.45cm} + \INTDom{\bigg(\INTSeg{|v(s,\sigma_{\nu}(s)) - w(s,\sigma_{\nu}(s))|}{s}{\tau}{t} \bigg)^p}{\{(y,\sigma_{\nu}) \; \textnormal{s.t.} \; \NormC{\sigma_{\nu}(\cdot)}{0}{[\tau,T],\R^d} > R \}}{\Beta_{\nu}(y,\sigma_{\nu})}, 
\end{aligned}
\end{equation}
for all $t \in [0,T]$. The first term in the right-hand side of \eqref{eq:StabilityIneq6} can be estimated straightforwardly as 
\begin{equation}
\label{eq:StabilityIneq61}
\begin{aligned}
& \INTDom{\bigg(\INTSeg{|v(s,\sigma_{\nu}(s)) - w(s,\sigma_{\nu}(s))|}{s}{\tau}{t} \bigg)^p}{\{(y,\sigma_{\nu}) \; \textnormal{s.t.} \; \NormC{\sigma_{\nu}(\cdot)}{0}{[\tau,T],\R^d} \leq R \}}{\Beta_{\nu}(y,\sigma_{\nu})} \\
& \hspace{7cm} \leq \bigg( \INTSeg{\NormL{v(s) - w(s)}{\infty}{B(0,R),\R^d; \, \nu(s)}}{s}{\tau}{t} \bigg)^p, 
\end{aligned}
\end{equation}
for all times $t \in [\tau,T]$. Concerning the second term in the right-hand side of \eqref{eq:StabilityIneq6}, it follows from the sublinearity assumptions made on $v,w : [0,T] \times \R^d \to \R^d$ and H\"older's inequality that 
\begin{equation}
\label{eq:StabilityIneq62}
\begin{aligned}
& \INTDom{\bigg(\INTSeg{\big| v(s,\sigma_{\nu}(s)) - w(s,\sigma_{\nu}(s)) \big|}{s}{\tau}{t} \bigg)^p}{\{(y,\sigma_{\nu}) \; \textnormal{s.t.} \; \NormC{\sigma_{\nu}(\cdot)}{0}{[\tau,T],\R^d} > R \}}{\Beta_{\nu}(y,\sigma_{\nu})} \\
& \hspace{0.45cm} \leq 2^p \INTDom{\bigg(\INTSeg{m(s) \big(1 + |\sigma_{\nu}(s)| \big)}{s}{\tau}{t} \bigg)^p}{\{(y,\sigma_{\nu}) \; \textnormal{s.t.} \; \NormC{\sigma_{\nu}(\cdot)}{0}{[\tau,T],\R^d} > R \}}{\Beta_{\nu}(y,\sigma_{\nu})} \\
& \hspace{0.45cm} \leq 2^p \NormL{m(\cdot)}{1}{[\tau,t]}^{p/q} \INTSeg{m(s) \bigg( \INTDom{\big( 1 + |\sigma_{\nu}(s)| \big)^p}{\{(y,\sigma_{\nu}) \; \textnormal{s.t.} \; \NormC{\sigma_{\nu}(\cdot)}{0}{[\tau,T],\R^d} > R \}}{\Beta_{\nu}(y,\sigma_{\nu})} \bigg)}{s}{\tau}{t}. 
\end{aligned}
\end{equation}
Recall now that by \eqref{eq:SublinApp} above, it also holds that 
\begin{equation*}
\NormC{\sigma_{\nu}(\cdot)}{0}{[\tau,T],\R^d} \leq C_T(1+|y|), 
\end{equation*}
for $\Beta_{\nu}$-almost every $(y,\sigma_{\nu})$, which together with \eqref{eq:StabilityIneq62} then yields
\begin{equation}
\label{eq:StabilityIneq63}
\begin{aligned}
& \INTDom{\bigg(\INTSeg{\big| v(s,\sigma_{\nu}(s)) - w(s,\sigma_{\nu}(s)) \big|}{s}{\tau}{t} \bigg)^p}{\{(y,\sigma_{\nu}) \; \textnormal{s.t.} \; \NormC{\sigma_{\nu}(\cdot)}{0}{[\tau,T],\R^d} > R \}}{\Beta_{\nu}(y,\sigma_{\nu})} \\
& \hspace{3.5cm} \leq 2^p \NormL{m(\cdot)}{1}{[\tau,t]}^p \INTDom{\Big( 1 + C_T(1+|y|) \Big)^p}{\{ y \; \textnormal{s.t.} \; |y| \geq R/C_T - 1 \}}{\nu_{\tau}&(y)}, 
\end{aligned}
\end{equation}
for all times $t \in [\tau,T]$. Therefore, by merging \eqref{eq:StabilityIneq61} and \eqref{eq:StabilityIneq63} in \eqref{eq:StabilityIneq5} while applying H\"older's inequality, we finally obtain that 
\begin{equation*}
\begin{aligned}
W_p(\mu(t),\nu(t)) \leq C_p \Bigg( W_p(\mu_{\tau},\nu_{\tau}) & + \INTSeg{\NormL{v(s) - w(s)}{\infty}{B(0,R),\R^d;\, \nu(s)}}{s}{\tau}{t} \\
& + 2(1+C_T) \NormL{m(\cdot)}{1}{[\tau,t]} \bigg( \INTDom{(1+|y|)^p}{\{ y \; \textnormal{s.t.} \; |y| \geq R/C_T-1 \}}{\nu_{\tau}(y)} \bigg)^{\hspace{-0.05cm} 1/p} \, \Bigg) \\
& \hspace{7.25cm} \times \exp \Big( C_p' \NormL{l(\cdot)}{1}{[\tau,t]}^p \hspace{-0.05cm} \Big),
\end{aligned}
\end{equation*}
for all times $t \in [\tau,T]$, which concludes the proof of our claim. 
\end{proof}


\bibliographystyle{plain}
{\footnotesize
\bibliography{../../ControlWassersteinBib}
}

\end{document}